\documentclass[12pt]{amsart}

\usepackage{amssymb,stmaryrd}
\usepackage{booktabs}
\usepackage{cases}
\usepackage{array} 
\usepackage{float} 
\usepackage{rotating }

\usepackage[foot]{amsaddr}
\makeatletter
\renewcommand{\email}[2][]{%
  \ifx\emails\@empty\relax\else{\g@addto@macro\emails{,\space}}\fi%
  \@ifnotempty{#1}{\g@addto@macro\emails{\textrm{(#1)}\space}}%
  \g@addto@macro\emails{#2}%
}
\makeatother
\usepackage{mathtools}
\usepackage{enumerate,pdflscape}
\usepackage{multirow}
\usepackage{color}

 \newcommand{\ssl}{\mathfrak{sl}(2,\mathbb{C})}
 \newcommand{\gtwoC}{\mathfrak{gl}(2,\mathbb{C})}

 \newcommand{\ad}{\text{ad}}

 \newcommand{\soo}{\mathfrak{sp}(4,\mathbb{C})}
 \newcommand{\st}{\mathfrak{sl}(3,\mathbb{C})}
  \theoremstyle{definition}
  \newtheorem{definition}{Definition}[section]

  \theoremstyle{plain}
  \newtheorem{lemma}[definition]{Lemma}
  
  \newtheorem{theorem}[definition]{Theorem}
  
\newcommand\Cb{\mathbb{C}}
\newcommand\alg{\mathfrak{sp}(4,\Cb)}
\newcommand\g{\mathfrak{g}}
\newcommand\borel{\mathfrak{b}}
\newcommand\p{\mathfrak{p}}
\newcommand\nf{\mathfrak{n}}
\newcommand\algn{\mathfrak{n}}
\newcommand\alga{ \frak{a}}
\newcommand\algt{ \frak{t}}

\newcommand\sz{\frak{s}_{Z}}

\newcommand\SptwoC{Sp(4, \Cb)}

\title{The subalgebras of the rank two symplectic Lie algebra}

\begin{document}
\date{\today}                                           

\author[Andrew Douglas]{Andrew Douglas$^{1,2,3}$}
\address{$^1$Ph.D. Programs in Mathematics and Physics, The Graduate Center, City University of New York, New York, NY,  10016, USA}
\address{$^2$Department of Mathematics, New York City College of Technology, City University of New York, Brooklyn, NY, 11201, USA}

\author[Joe Repka]{Joe Repka$^3$}
\address{$^3$Department of Mathematics, University of Toronto, Toronto, ON, M5S 2E4, Canada}
\email{adouglas2@gc.cuny.edu,repka@math.toronto.edu}



\begin{abstract}
 The semisimple subalgebras of the rank $2$ symplectic Lie algebra $\soo$ are well-known, and we recently classified its Levi decomposable subalgebras.  In this article, we classify the solvable subalgebras of $\soo$, up to inner automorphism.   This completes the classification of the subalgebras of $\soo$.  More broadly speaking, in completing the classification of the subalgebras of $\soo$ we have completed the classification of the subalgebras of the rank $2$ semisimple Lie algebras.
\end{abstract}

\maketitle


\section{Introduction}

Semisimple subalgebras of semisimple Lie algebras have been extensively studied  \cite{degraafd, dynkin, dynkin2, lorent, min}.   For instance, the semisimple subalgebras of the exceptional Lie algebras have been classified, up to inner automorphism \cite{min}.  As another important example,   de Graaf \cite{degraafd} classified the semisimple subalgebras of the simple Lie algebras of ranks $\le 8$, up to linear equivalence, which is somewhat weaker than a classification up to inner automorphism.

Much less research has examined general subalgebras of semisimple Lie algebras.  By Levi's Theorem [\cite{levi}, Chapter II, Section 2],  a subalgebra of a complex semisimple Lie algebra is either semisimple, solvable, or a nontrivial semidirect sum of the first two.  A subalgebra that is a  nontrivial semidirect sum of a semisimple subalgebra with a solvable subalgebra is called a Levi decomposable subalgebra. 

We have made considerable progress towards classifying both solvable and Levi decomposable subalgebras of semisimple Lie algebras.   Most relevant among this work to the present paper is our classification of the solvable, and Levi decomposable subalgebras of the rank $2$, semisimple Lie algebras $\ssl\oplus \ssl$ \cite{drsof}, and $\st$ \cite{a2}. Since the classifications of semisimple subalgebras of  $\ssl\oplus \ssl$, and $\st$ are well-known, our work completes the classification of subalgebras of these rank $2$, semisimple Lie algebras. 

The aim of the current paper is to complete the classification of subalgebras of the rank $2$, symplectic Lie algebra $\soo$--the remaining rank $2$, classical, semisimple Lie algebra whose subalgebras have not been classified.  The semisimple  subalgebras of $\soo$ are well-known \cite{degraafd}, and the authors recently classified its Levi decomposable subalgebras \cite{dc2}.  

Hence, in this article, we classify the most difficult case: the solvable subalgebras of $\soo$, up to inner automorphism  (equivalently,  up to conjugacy by the symplectic group $Sp(4,\mathbb{C})$).   By Levi's theorem, this completes the classification of the subalgebras of $\soo$.  

 In addition, Mayanskiy \cite{may} recently posted a classification of the subalgebras of the exceptional Lie algebra $G_2$. In light of the above mentioned work,  this article completes of the classification of the subalgebras of the rank $2$ semisimple Lie algebras.

In addition to the intrinsic mathematical significance of  classifications of  subalgebras of Lie algebras (or classifications of subgroups of corresponding  Lie groups), such classifications also have physical significance and mathematical application, some of which are listed below:
\begin{itemize}
\item If a system of differential equations is invariant under a Lie group, then its
subgroups can be used to construct group invariant solutions \cite{pjoliver}. 
\item Subgroups of the symmetry groups of nonlinear partial differential equations provide a method for performing
symmetry reduction (reducing  the number of independent
variables) \cite{wlong, ghp,dklw}.
\item A knowledge of the subgroup structure of a Lie group $G$ is needed if we are interested in considering all possible contractions of $G$ to other groups \cite{pw77}.
\item Physical models--such as the vibron model, and the interacting boson model-- use chains of subalgebras, and these subalgebras need to be explicitly described in application  \cite{iachello}.
\end{itemize}

The article is organized as follows. In Section \ref{solvable}, we describe two partial classifications of solvable Lie algebras that will be used in our
classification of solvable subalgebras of $\soo$: The classification of de Graaf \cite{degraafa}, and that described by  \v{S}nobl and Winternitz in \cite{levi}.  
Section $3$ contains preliminary background on $\soo$, and in Section $4$ we classify the one-dimensional subalgebras of $\soo$.

In Section $5$  we develop preliminary results that will be used in classifications of subalgebras of dimension greater than one. In Sections $6$, $7$, and $8$ we classify the two-, three-, and four-dimensional solvable subalgebras, respectively. Section $9$  contains the classification of the five-, and six-dimensional solvable subalgebras.

Finally, in Section \ref{isomorphs}, we identify our classification of solvable subalgebras of $\soo$ with respect to the classification of solvable Lie algebras of de Graaf \cite{degraafa}, and that described by \v{S}nobl and Winternitz in \cite{levi}. The complete classification of the subalgebras of $\soo$ is summarized in Tables \ref{onedimuu} to \ref{fivesixdimuu}.  

All Lie algebras and representations in this article  are finite dimensional, and over the complex numbers, unless otherwise stated.

\section{Solvable Lie algebras of small dimension}\label{solvable}

A full classification of solvable Lie algebras is not known and thought to be an impossible task.  However, partial classifications of solvable Lie algebras do exist.  Two such partial classifications are that of de Graaf \cite{degraafa}, and that described by \v{S}nobl and Winternitz in \cite{levi}. The classification of solvable subalgebras of $\soo$ in this article will be described with respect to both of these classifications, and  we consider both of these classifications in this section.

De Graaf classified the solvable Lie algebras in dimensions $\leq 4$ over a field $\mathbb{F}$ of any characteristic \cite{degraafa}. In his classification, de Graaf does not distinguish between indecomposable and decomposable Lie algebras.  We describe the classification in its entirety up to and including dimension $3$, and include only those four-dimensional solvable subalgebras which appear in this article:
\begin{equation}
\begin{array}{llll}
J & \text{The abelian Lie algebra of dimension $1$}
\end{array}
\end{equation}
\begin{equation}
\arraycolsep=2pt\def\arraystretch{1.5}
\begin{array}{llll}
K^1 & \text{The abelian Lie algebra of dimension $2$}\\
K^2 & [x_1, x_2]=x_1 
\end{array}
\end{equation}
\begin{equation}
\arraycolsep=2pt\def\arraystretch{1.5}
\begin{array}{llll}
L^1 & \text{The abelian Lie algebra of dimension $3$}\\
L^2 & [x_3,x_1]=x_1, [x_3, x_2]=x_2 \\
L^3_{A} & [x_3, x_1]=x_2, [x_3, x_2]=A x_1+x_2 \\
L^{4}_A & [x_3, x_1]=x_2, [x_3, x_2]=A x_1 
\end{array}
\end{equation}
Note that $L^3_{A}\cong L^3_{B}$ if and only if $A=B$; and 
$L^{4}_A \cong L^{4}_B$ if and only if there is an $\alpha \in \mathbb{F}^*$ with $A=\alpha^2 B$.  
\begin{small}
\begin{equation}
\arraycolsep=2pt\def\arraystretch{1.5}
\begin{array}{llll}
M^2 & [x_4, x_1]=x_1, [x_4, x_2]=x_2,  [x_4, x_3]=x_3\\
M^6_{A, B} &  [x_4,x_1]=x_2, [x_4,x_2]=z_3, [x_4,x_3]=Ax_1 + Bx_2+x_3\\
M^7_{A, B} &  [x_4, x_1]=x_2, [x_4,x_2]=x_3, [x_4,x_3]=Ax_1 + Bx_2\\
M^8  &  [x_1, x_2]=x_2, [x_3,x_4]=x_4\\
M^{12} & [x_4,x_1]=x_1, [x_4,x_2]=2x_2,  [x_4,x_3]=x_3, [x_3, x_1]=x_2\\
M^{13}_{A} & [x_4,x_1]=x_1+Ax_3, [x_4,x_2]=x_2, [x_4,x_3]=x_1, [x_3,x_1]=x_2\\
M^{14}_{A} &  [x_4, x_1]=Ax_3,[x_4,x_3]=x_1,[x_3,x_1]=x_2, A\neq 0
\end{array}
\end{equation}
\end{small}
Note that $M^6_{A, B} \cong M^6_{C, D}$ if and only if $A=C$ and $B=D$;  $M^7_{A,B} \cong M^7_{C,D}$ if and only if there is an 
$\alpha \in \mathbb{F}^*$ with $A=\alpha^3 C$ and $B=\alpha^2 D$; 
$M^{13}_A \cong M^{13}_B$ if and only if $A=B$;  and $M^{14}_A \cong  M^{14}_B$ if and only if there is an 
$\alpha \in \mathbb{F}^*$ with $A=\alpha^2 B$.

An alternate classification of solvable Lie algebras is presented  by  \v{S}nobl and Winternitz in \cite{levi}, which is up to and including dimension  $6$ and includes only indecomposable Lie algebras.  This classification is an amalgam of results from various sources (e.g., \cite{bianchi, kruch, lie, moro, patera, shaban, turko, turko2}).

We present the classification from \cite{levi} in its entirety up to and including dimension  $3$. We  give a 
partial description of the classification in dimensions $4$, $5$ and $6$, including just those algebras which appear in this article:
\begin{small}
\begin{equation}\label{onesnobl}
\begin{array}{llll}
\mathfrak{n}_{1,1} & \text{The abelian Lie algebra of dimension $1$}\\
\end{array}
\end{equation}

\begin{equation}
\arraycolsep=2pt\def\arraystretch{1.5}
\begin{array}{llll}
 \mathfrak{s}_{2,1} &      [e_2, e_1]=e_1 
\end{array}
\end{equation}

\begin{equation}\label{threesnobl}
\arraycolsep=2pt\def\arraystretch{1.5}
\begin{array}{llll}
 \mathfrak{n}_{3,1} &   [e_2, e_3]=e_1 \\
 \mathfrak{s}_{3,1} &      [e_3, e_1]=e_1, [e_3, e_2]=A e_2   \\
 & 0<|A|\leq 1, ~\text{if}~|A|=1 ~\text{then}~\arg(A)\leq \pi \\
  \mathfrak{s}_{3,2} &   [e_3, e_1]=e_1, [e_3, e_2]= e_1+e_2
\end{array}
\end{equation}

\begin{equation}\label{foursnobl}
\arraycolsep=2pt\def\arraystretch{1.5}
\begin{array}{llll}
 \mathfrak{n}_{4,1} &  [e_2, e_4]=e_1, [e_3, e_4]=e_2\\
 \mathfrak{s}_{4,2} &  [e_4, e_1]=e_1, [e_4, e_2]=e_1+e_2, [e_4, e_3]=e_2+e_3 \\
   \mathfrak{s}_{4,3} & [e_4, e_1]= e_1, [e_4, e_2]=Ae_2, [e_4, e_3]=Be_3  \\
   & 0 <|B| \leq |A| \leq 1, (A,B) \neq (-1, -1) \\
  \mathfrak{s}_{4,6} &  [e_2, e_3]=e_1, [e_4, e_2]=e_2, [e_4, e_3]=-e_3 \\
    \mathfrak{s}_{4,8} &  [e_2, e_3]=e_1, [e_4, e_1]=(1+A)e_1, [e_4, e_2]=e_2, [e_4,e_3]=Ae_3 \\
    & 0<|A|\leq 1, ~\text{if}~|A|=1 ~\text{then}~\arg(A)< \pi\\
      \mathfrak{s}_{4,10} &  [e_2, e_3]=e_1, [e_4, e_1]=2e_1, [e_4, e_2]=e_2, [e_4, e_3]=e_2+e_3 \\
        \mathfrak{s}_{4,11} &  [e_2, e_3]=e_1, [e_4, e_1]=e_1, [e_4, e_2]=e_2 \\
          \mathfrak{s}_{4,12} &  [e_3, e_1]=e_1, [e_3, e_2] = e_2, [e_4, e_1]=-e_2, [e_4, e_2]=e_1
\end{array}
\end{equation}

\begin{equation}
\arraycolsep=2pt\def\arraystretch{1.5}
\begin{array}{llll}
 \mathfrak{s}_{5,33} &  [e_2, e_4]=e_1, [e_3, e_4]=e_2, [e_5, e_2]=-e_2,\\
 &  [e_5, e_3]=-2e_3, [e_5, e_4]=e_4\\
 \mathfrak{s}_{5,35} &  [e_2, e_4]=e_1, [e_3, e_4]=e_2, [e_5, e_1]=(A+2)e_1,\\
 &  [e_5, e_2]=(A +1)e_2, [e_5, e_3]=A e_3, [e_5, e_4]= e_4,\\
 & A \neq 0, -2\\
  \mathfrak{s}_{5,36} &   [e_2, e_4]=e_1, [e_3, e_4]=e_2, [e_5, e_1]=2e_1,\\
 &  [e_5, e_2]=e_2, [e_5, e_4]=e_4\\
 \mathfrak{s}_{5,37} &  [e_2, e_4]=e_1, [e_3, e_4]=e_2, [e_5, e_1]=e_1,\\
 &  [e_5, e_2]=e_2, [e_5, e_3]=e_3\\
  \mathfrak{s}_{5,41} &  [e_4, e_1]=e_1, [e_4, e_3]=A e_3,\\
  &  [e_5, e_2]=e_2, [e_5, e_3]=B e_3,\\& 0< |B | \leq |A | \leq 1\\
   \mathfrak{s}_{5,44} & [e_2,e_3]=e_1, [e_4,e_1]=e_1, [e_4,e_2]=e_2,\\
   & [e_5,e_2]=e_2, [e_5,e_3]=-e_3
\end{array}
\end{equation}

\begin{equation}
\arraycolsep=2pt\def\arraystretch{1.5}
\begin{array}{llll}
 \mathfrak{s}_{6,242} &  [e_2, e_4]=e_1, [e_3, e_4]=e_2, [e_5, e_1]=2e_1, \\
  & [e_5, e_2]=e_2, [e_5, e_4]=e_4, [e_6, e_1]=e_1,\\
  & [e_6, e_2]=e_2, [e_6, e_3]=e_3
\end{array}
\end{equation}
\end{small}

In the classification from \cite{levi}, an algebra designated with $\mathfrak{n}$ is nilpotent, and one with $\mathfrak{s}$
is solvable, but not nilpotent. The first subscript indicates the dimension, and the second index is for enumeration. So, $\mathfrak{s}_{6, 242}$  is the $242^{\text{nd}}$ six-dimensional, solvable, non-nilpotent Lie algebra in the classification.

\section{Preliminaries}\label{sl}

The symplectic algebra $\soo$ is the Lie algebra of $4\times 4$ complex matrices $X$ satisfying $J X^{t}J=X$, where $J$ is the $4\times 4$ matrix
\begin{equation}\label{eq:matrixJ}
  J = \begin{pmatrix}
    0&0&1&0\\ 
    0&0&0&1\\
    -1&0&0&0\\
    0&-1&0&0\\
    \end{pmatrix}.
\end{equation}
The corresponding Lie group is the symplectic group $Sp(4, \Cb)$ given by $\{ g \in GL(4,\Cb) \mid gJg^t=J \}$. 

Let $\algt$ be the diagonal Cartan subalgebra,  and $T$ the corresponding Cartan subgroup. 
For $a,b \in \Cb$,  define 
 \begin{equation} \label{Eq:DiagonalElement}
T_{a,b} = {\rm diag}(a,b,-a,-b).
\end{equation}
If we choose the positive root vectors to be
\begin{eqnarray}\label{eq:rootvectors}
X_{\alpha} = \begin{pmatrix}
0&0&0&0\\
0&0&0&1\\
0&0&0&0\\
0&0&0&0\\
\end{pmatrix}, 
&X_{\beta} = \begin{pmatrix}
0&1&0&0\\
0&0&0&0\\
0&0&0&0\\
0&0&-1&0\\
\end{pmatrix}, \\
X_{\alpha+\beta} = \begin{pmatrix} \label{eq:rootvectors2}
0&0&0&1\\
0&0&1&0\\
0&0&0&0\\
0&0&0&0\\
\end{pmatrix}, 
&X_{\alpha + 2\beta} = \begin{pmatrix}
0&0&1&0\\
0&0&0&0\\
0&0&0&0\\
0&0&0&0\\
\end{pmatrix},
\end{eqnarray}
then $X_{\alpha}$ and $X_{\alpha +  2 \beta}$ correspond to the long roots, while 
$X_{ \beta}$ and $X_{\alpha +   \beta}$ correspond to the short roots.  
The corresponding Borel subalgebra is
\begin{equation}
\borel =    \begin{pmatrix} 
      * & * & * & * \\
      0 & * & * & * \\
      0 & 0 & * & 0 \\
            0 & 0 & * & * \\
   \end{pmatrix} 
   \cap \alg,
\end{equation}
and we let  $B  \subset Sp(4, \Cb)$ be the corresponding subgroup.  
Let $\frak{n}$ be its nilpotent radical  
\begin{equation}
\frak{n} =    \begin{pmatrix} 
      0 & * & * & * \\
      0 & 0 & * & * \\
      0 & 0 & 0 & 0 \\
            0 & 0 & * & 0 \\
   \end{pmatrix} 
   \cap \alg 
  = \left\{
   \begin{pmatrix} 
      0 & t & a & b \\
      0 & 0 & b & d \\
      0 & 0 & 0 & 0 \\
            0 & 0 & -t & 0 \\
   \end{pmatrix} 
  \right\},
\end{equation}
and let  $N  \subset Sp(4, \Cb)$ be the corresponding subgroup.  
Also,  let 
$\p$ be the maximal parabolic subalgebra                  
\begin{equation}
  \p =  \begin{pmatrix} 
      * & * & * & * \\
      *  & * & * & * \\
      0 & 0 & * & * \\
            0 & 0 & * & * \\
   \end{pmatrix} 
   \cap \alg,
\end{equation}
and let $P \subset Sp(4, \Cb)$  be the corresponding subgroup. 

Let $\frak{n}_{\p}$ be the nilpotent radical of $\p$: 
\begin{equation}
\frak{n}_{\p} =    \begin{pmatrix} 
      0 & 0 & * & * \\
      0 & 0 & * & * \\
      0 & 0 & 0 & 0 \\
            0 & 0 & 0 & 0 \\
   \end{pmatrix} 
   \cap \alg 
   = \left\{
   \begin{pmatrix} 
      0 & 0 & a & b \\
      0 & 0 & b & d \\
      0 & 0 & 0 & 0 \\
            0 & 0 & 0 & 0 \\
   \end{pmatrix} \right\},
\end{equation}
and let  $N_{P}  \subset Sp(4, \Cb)$ be the corresponding subgroup.  

The following lemmas will be used below.
\begin{lemma}\label{lemmapairs}
The nonzero eigenvalues of an element of $\alg$ occur in negative pairs   
(that is, its eigenvalues are of the form $a, b, -a,  -b$).  
The nonzero generalized eigenvalues of an element of $\alg$ occur in negative pairs.  
\end{lemma}

\begin{proof}
Any element is conjugate to an element of $\borel$.  The two diagonal blocks of 
such an element are negative transposes,  so have the negatives of each other's 
eigenvalues and generalized eigenvalues.
\end{proof}

\begin{lemma}  \label{lemma:ssinborel}
Suppose $X \in \borel$ is semisimple.  Then there is $b \in B$ such that $bXb^{-1} \in \algt$, 
i.e.,  any semisimple element of $\borel$ is conjugate to an element of $\algt$ by an element 
of $B$. 

Moreover,  if $X = T+N$, with $T \in \algt$,  $N \in \algn$,  then 
$bXb^{-1} = T$.
\end{lemma}

\begin{proof} 
Since $X\in \mathfrak{b}$ is semisimple, it is an element of a Cartan subalgebra of $\borel$. 
Since all Cartan 
subalgebras of $\mathfrak{b}$ are conjugate under $B$ [\cite{humphreys}, Theorem 16.2], the first assertion follows. 

For the second assertion,  let $b\in B$ such that $bXb^{-1}\in \mathfrak{t}$.  Then
$bXb^{-1} = b(T+N)b^{-1}$.  Clearly $bNb^{-1} \in \algn$,  and $bTb^{-1} 
= T + N'$,  with $N' \in \algn$.  We must have $N' = - bNb^{-1}$ and $bXb^{-1}=T$.
\end{proof}  

We summarize the classification of the conjugacy classes of $\soo$ in the following two tables, 
which distinguish conjugacy by $B$,  $P$,  and $G = \SptwoC$.  

\begin{table}[H]
\renewcommand{\arraystretch}{1.1}
   \centering
\begin{tabular}{|c|c|c|c|c|}
  \hline
Representative    &   Conditions &   $\qquad P\qquad$ &   $\qquad Sp(4, \mathbb{C})\qquad$ &   Eigenvalues   \\
  \hline
$T_{a,b}$ & $a \ne 0$,  & $(a,b) \sim (b,a)$ &$ (a,b)\qquad\quad\quad $ & $4$ \\
& $b \ne 0$ &  &$\quad \sim (-a,-b)$ & \\
&$a \ne \pm b$&&$\quad \sim (a,-b)$ & \\
  \hline
$T_{a,0}$ & $a \ne 0$ & $(a,0) \sim (0,a)$ & $(a,0) \qquad\quad\quad$& $3$ \\
 &  &  & $\quad \sim (-a,0)$& \ \\
$T_{0,a}$ & $a \ne 0$ & $(0,a) \sim (a,0)$ & $(0,a) \qquad\quad\quad$& $3$ \\
  &  &  & $\quad \sim (0,-a)$& \ \\
  \hline
$T_{a,-a}$ & $a \ne 0$ & $(a,-a) \qquad \, \, \, \,$ & $(a,-a)\qquad\quad\quad $ & $2$ \\
 &  & $\quad  \sim (-a,a)$ & $\quad \sim (a,a)$ &  \\
$T_{a, a}$ & $a \ne 0$ &&$(a,a) $ \qquad\qquad\quad & $2$ \\
  &  &&$\quad \sim (a,-a)$  &  \\
    &  &&$\quad \sim (-a,-a)$  &  \\
  \hline
$T_{0,0}$ &  &  & & $1$ \\
  \hline
   \end{tabular}
   \bigskip
   \caption{Semisimple classes.  
   Representatives
   with distinct values of $a,b$ give classes which 
 are inequivalent under the Adjoint action of $B$.  
 The third  column indicates 
 equivalences under the action of $P$ while 
 the fourth indicates additional 
 equivalences under the action of $Sp(4, \mathbb{C})$.  
  The last column gives the number of distinct eigenvalues.
  The pairs of $B$-classes with 
  the same
  $3$ distinct eigenvalues are equivalent under $P$.  
  The pairs 
  of 
  $B$-classes with  
  the same
  $2$ distinct eigenvalues are equivalent under $Sp(4, \mathbb{C})$.  
  Note that the regular elements are exactly those with $4$ distinct eigenvalues. 
}\label{table:ssclasses}
\end{table}
\medskip

\begin{table}[H]
\renewcommand{\arraystretch}{1.2}
   \centering
 \begin{tabular}{|c|c|c|c|c|} 
 \hline
  Representative    & Conditions & $\qquad P \qquad$ & $\qquad Sp(4, \mathbb{C})$ & JCF \\
  \hline
$T_{a,0}+X_{\alpha}$ & $a \ne 0$ & 
$\begin{array}{c}
{\rm conjugate} \\
{\rm to \,\, below} \\
   \end{array}$
   &
   $\begin{array}{l}
       (a,0,-a,0) \\
        \sim (-a,0,a,0) \\
\end{array}
$
& 
$\begin{array}{l}
       (2,1,1) \\
        0, a, -a \\
\end{array}
$
\\
&&&&\\
$T_{0,a}+X_{\alpha+2\beta}$ & $a \ne 0$ &
$\begin{array}{c}
{\rm conjugate} \\
{\rm to \,\, above} \\
   \end{array}$ &
      $\begin{array}{l}
      (0,a,0,-a) \\
        \sim (0,-a,0,a) \\
   \end{array}$
& $\begin{array}{l}
       (2,1,1) \\
        0, a, -a \\
\end{array}
$
 \\
\hline
$T_{a,a}+X_{\beta}$ & $a \ne 0$ & &
$
\begin{array}{l}
 (a,a) \\ \sim (-a,-a) \\
 {\rm and \,\, to \,\, below} 
\end{array}$
& $\begin{array}{l}
       (2,2) \\
        a, -a \\
\end{array}$
 \\
 &&&&\\
 $T_{a,-a}+X_{\alpha+\beta}$ & $a \ne 0$ & &
$
\begin{array}{l}
 (a,-a) \\ \sim (-a,a) \\
 {\rm and \,\, to \,\, above} 
\end{array}$
& $\begin{array}{l}
       (2,2) \\
        a, -a \\
\end{array}
$ \\
\hline
$X_{\beta}$ &  &  
  
$\begin{array}{c}
\\
{\rm all} \\
   \end{array}$ 
    & & $\begin{array}{l}
       (2,2) \\
        \,\,\, 0,0 \\
\end{array}
$ \\ 
$X_{\alpha +\beta}$ 
&  &  
$\begin{array}{c}
{\rm three} \\
   \end{array}$
&   & $\begin{array}{l}
       (2,2) \\
        \,\,\, 0,0 \\
\end{array}
$ \\
$X_{\alpha}+X_{\alpha+2\beta}$ &  
&
$\begin{array}{c}
{\rm conjugate} \\
\\
   \end{array}$
   &
& $\begin{array}{l}
       (2,2) \\
      \,\,\,  0,0 \\
\end{array}
$ \\
\hline
$X_{\alpha+2\beta}$ &  &  
$\begin{array}{c}
{\rm conjugate} \\
{\rm to \,\, below} \\
   \end{array}$
&& $\begin{array}{l}
       (2,1,1) \\
      \,\,  0,0,0 \\
\end{array}
$ \\
&&&&\\
$X_{\alpha}$ &  &  
$\begin{array}{c}
{\rm conjugate} \\
{\rm to \,\, above} \\
   \end{array}$
&& $\begin{array}{l}
       (2,1,1) \\
      \,\,  0,0,0 \\
\end{array}
$ \\
\hline
$X_{\alpha}+X_{\beta}$
&&&& $\begin{array}{l}
       (4) \\
      \,\,  0 \\
\end{array}
$ \\
\hline
   \end{tabular}
   \bigskip
   \caption{The Nonsemisimple Classes.  Distinct values of the parameter $a$ give 
   representatives for the nonsemisimple classes 
up to equivalence under conjugation by $B$. The third column 
  lists equivalences under conjugation by $P$,  and the fourth 
  lists additional equivalences under conjugation by $Sp(4, \mathbb{C})$.
   The final column gives the block sizes of the Jordan normal form 
and the corresponding eigenvalues.  
} \label{tab:nonss}
\end{table}

\section{One-dimensional subalgebras of $\soo$}

Any solvable subalgebra of $\soo$ is contained in a Borel subalgebra and hence is 
conjugate to a subalgebra of $\borel$.  Accordingly,  we shall focus on 
solvable subalgebras of $\borel$.

In this section, we classify the one-dimensional (solvable) subalgebras of $\soo$ 
by separating them into
 three cases:  subalgebras with semisimple generators (Theorem \ref{semiclass}), subalgebras with nilpotent generators (Theorem \ref{nilclass}),  and subalgebras with generators that have  non-trivial Jordan decompositions (Theorem \ref{ntjd}).  The results are summarized in Table \ref{onedimuu}.  



\begin{theorem}\label{semiclass}
Every semisimple element of $\soo$ is conjugate to an element $T_{a,b}$ in $\mathfrak{t}$ (c.f., 
Eq. (\ref{Eq:DiagonalElement})).

 A complete list of   one-dimensional subalgebras of $\soo$ with semisimple generators, up to conjugacy in
$Sp(4, \mathbb{C})$, is
\begin{equation} 
\begin{array}{llllllll}
\langle T_{1,b} \rangle \cong \langle T_{1,b^{-1}} \rangle,  b \neq 0,  \pm 1; &\langle T_{1,0} \rangle, ;&
\langle T_{1, 1}\rangle. 
\end{array}
\end{equation} 
If $a,b \ne 0$,  $b \ne \pm a$, then 
the subalgebras $\langle T_{a,b} \rangle$, $\langle T_{a,0} \rangle$ and
$\langle T_{a, a}\rangle$ are pairwise inequivalent; $\langle T_{a,0} \rangle =\langle T_{1,0} \rangle$;
$\langle T_{a,a} \rangle =\langle T_{1,1} \rangle$,   which is conjugate to $\langle T_{1,-1} \rangle$; 
and $\langle T_{a,b} \rangle$ is conjugate to $\langle T_{a,'b'} \rangle$ if and  only if 
$\{a,b\}=\{ \lambda a', \pm \lambda b'\}$,
for some $\lambda \in \mathbb{C}^{\times}$.
\end{theorem}
\begin{proof}
Every semisimple element $T$ of $\soo$ is conjugate to 
an element in $\mathfrak{t}$ [\cite{collingwood}, Corollary 2.2.2], so we may assume $T\in \mathfrak{t}$. 

Two elements in $\mathfrak{t}$ are $Sp(4,\mathbb{C})$-conjugate if and only if they are $W$-conjugate [\cite{collingwood}, Theorem 2.2.4] , where $W$ is the Weyl group corresponding to $\mathfrak{t}$.  The Weyl group $W$ of $\soo$  has generator $s_\alpha$ and $s_\beta$ such that  
$s_\alpha(T_{a,b})=T_{a,-b}$ and $s_\beta(T_{a,b})=T_{b,a}$. The result follows. 
\end{proof} 



\begin{theorem}\label{nilclass}
 A complete list of  inequivalent, one-dimensional subalgebras of $\soo$ with nilpotent generators, up to conjugacy in
$Sp(4, \mathbb{C})$, is
\begin{equation}
\begin{array}{llllllll}
\langle X_{\beta} \rangle; &\langle X_{\alpha} \rangle;&
\langle X_{\alpha}+X_{\beta} \rangle.
\end{array}
\end{equation}
\end{theorem}
\begin{proof}
There  are precisely three nonzero nilpotent orbits of $\soo$ [\cite{collingwood}, Theorem 5.1.3] with representatives $X_\alpha, X_\beta$, and $X_\alpha +X_\beta$.
\end{proof}


\begin{theorem}\label{ntjd}
 A complete list of  one-dimensional subalgebras of $\soo$ with generators 
 having a nontrivial Jordan decomposition, 
 i.e.,  that are neither semisimple nor nilpotent,  
 up to conjugacy in $Sp(4, \mathbb{C})$, is
\begin{equation}
\begin{array}{llllllll}
\langle T_{1, 0}+X_{\alpha} \rangle; &\langle T_{1,1}+ X_{\beta} \rangle.
\end{array}
\end{equation}

The subalgebras $\langle T_{a, 0}+X_{\alpha} \rangle$ and $\langle T_{0, a}+X_{\alpha + 2 \beta} \rangle$ are 
conjugate to $\langle T_{1, 0}+ X_{\alpha} \rangle$, for any $a\neq 0$. The 
subalgebras $\langle T_{a, a}+X_{\beta} \rangle$ and $\langle T_{a, -a}+X_{\alpha + \beta} \rangle$ are  
conjugate to $\langle T_{1, 1}+ X_{\beta} \rangle$, for any $a\neq 0$.
\end{theorem}

\begin{proof}
The generator of such a subalgebra of $\borel$ must have Jordan decomposition $X = T + N$,  with $T \in \algt$,  $N \in \algn$,  and $T \ne 0$,  $N \ne 0$.  Moreover, 
$T$ cannot be regular or $X$ would be semisimple.  After multiplying by a scalar,  we can assume $T = T_{1,1}, T_{1,-1}, T_{1,0}$, or $T_{0,1}$.  Since $N$ 
and $T$ must commute,  the possibilities are 
\begin{equation*}
T_{1,1} + c X_{\beta}  \qquad T_{1,-1} + c X_{\alpha + \beta}  \qquad  T_{1,0} + c X_{\alpha } \qquad   T_{0,1} + c X_{\alpha + 2 \beta }   , 
\end{equation*}
for some $c \ne 0$.  After conjugation by a suitable diagonal element,  we can assume $c = 1$, so the possibilities are 
\begin{equation*}
T_{1,1} +  X_{\beta}  \qquad T_{1,-1} +  X_{\alpha + \beta}  \qquad  T_{1,0} + X_{\alpha } \qquad   T_{0,1} +  X_{\alpha + 2 \beta }   . 
\end{equation*}

Let
\begin{equation}\label{Eq:matrixW}
W = \begin{pmatrix}
0&1&0&0\\
-1&0&0&0 \\
0&0&0&1 \\
0&0&-1&0\\
\end{pmatrix} \in Sp(4, \Cb).
\end{equation} Then
\begin{equation}
W  (T_{1,0} + X_{\alpha } )W^{-1} =   T_{0,1} +  X_{\alpha + 2 \beta } .  
\end{equation}

Let 
\begin{equation}\label{eq:matrixAinSp}
A = \begin{pmatrix}
1&0&0&0\\
0&0&0&-1 \\
0&0&1&0 \\
0&1&0&0\\
\end{pmatrix} \in Sp(4, \Cb).
\end{equation} 
Then
\begin{equation}
A (T_{1,1} +  X_{\beta} ) A^{-1} = T_{1,-1} +  X_{\alpha + \beta} .
\end{equation} 

Since $T_{1, 0}+X_{\alpha}$ has rank $3$ and 
$ T_{1,1}+ X_{\beta}$ has rank $4$,  the subalgebras 
$\langle T_{1, 0}+X_{\alpha} \rangle$ and $\langle T_{1,1}+ X_{\beta} \rangle$ 
are not equivalent.  

The result follows.  
\end{proof}

\section{Preliminary results to be used in classification of higher dimensional subalgebras}



\subsection{Two-dimensional subalgebras of $\nf_{\p}$}\label{subsection:2dimsubalgnp}
If $Z$ is a symmetric $2 \times 2$ matrix,  then 
\begin{equation}
\sz = \{ X \in \gtwoC | ~X = X^{t} ~ {\rm and} ~
{\rm tr} (XZ^{t}) = 0 \}
\end{equation}
is a two-dimensional subspace of the symmetric matrices,  and every such subspace
is of this form for some such $Z$.  Note that 
\begin{equation}
\begin{array}{lllll}
&{\rm tr} (XZ^{t})  \\=  &{\rm tr} (gXZ^{t}g^{-1})  =  
 {\rm tr} (gXg^{t}g^{-t}Z^{t}g^{-1})  =
  {\rm tr} \bigl(gXg^{t} (g^{-t}Zg^{-1})^{t} \bigr).
  \end{array}
\end{equation}

Accordingly,  conjugating 
$\begin{pmatrix} 
0 & \sz \\ 0 & 0 \\
\end{pmatrix}$
by 
$ \begin{pmatrix} 
g & 0 \\ 0 & g^{-t} \\
\end{pmatrix} $ 
amounts to acting on $\sz$ by 
$X \mapsto gXg^{t}$, which in turn amounts to acting on $Z$ by 
$Z \mapsto g^{-t}Z g^{-1}$.  

Consequently,  there are two conjugacy classes of two-dimensional 
subspaces of $\nf_{\p}$, corresponding to 
$Z = \begin{pmatrix} 
1 & 0 \\ 0 & 0 \\
\end{pmatrix}$
and 
$Z = \begin{pmatrix} 
0 & 1 \\
1 & 0 \\ 
\end{pmatrix}$,  
respectively.   
Since these matrices $Z$ correspond to inequivalent forms,  we see that these two classes 
are inequivalent.  
We have just proved the following result.

\begin{lemma}\label{lemma:2dimSubsp.np}
There are two conjugacy classes of two-dimensional 
subalgebras
of $\nf_{\p}$,
with representatives 
\begin{equation}
\langle X_{\alpha}, X_{\alpha+ \beta}   \rangle 
\qquad {\rm and } \qquad 
\langle X_{\alpha}, X_{\alpha+2\beta}   \rangle. 
\end{equation} 
\end{lemma}

\subsection{Semisimple elements in $\borel$}

\begin{lemma}\label{lemmaCartan}
Suppose $\alga$ is a solvable subalgebra of $\borel$.  If $\alga$ contains semisimple 
elements, then it is possible to find $b \in B$ so that the conjugate 
$\alga^{b} = {\rm Ad}(b) \alga$ intersects $\algt$.  Moreover,  if $\alga$ contains a Cartan 
subalgebra,  a two-dimensional algebra of commuting semisimple elements,  then
it is possible to find $b \in B$ so that 
the conjugate $\alga^{b} $ contains $\algt$.  
\end{lemma}

\begin{proof}
The first assertion follows from Lemma \ref{lemma:ssinborel}.   Since all Cartan 
subalgebras of $\mathfrak{b}$ are conjugate under $B$ [\cite{humphreys}, Theorem 16.2], the second assertion follows.
\end{proof}

\begin{lemma}\label{lemmaFullss}
Suppose $\alga$ is a solvable subalgebra of $\borel$.  
Suppose $\alga$ contains $X = T+N$, $X' = T' + N'$,  with $T,T' \in \algt$, 
$N,N' \in \algn$. Suppose moreover that $T$ and $T'$ are linearly independent.  

Then it is possible to find $b \in B$ so that the conjugate 
$\alga^{b} = {\rm Ad}(b) \alga$ contains $\algt$.  
\end{lemma}

\begin{proof}
It is possible to find a linear combination of $X$ and $X'$ that can be written as 
$X'' = T'' + N''$, with $N'' \in \algn$ and $T'' \in \algt$ such that      $\ad(T'')$ has distinct eigenvalues on $\mathfrak{n}$. 
By Lemma \ref{lemma:ssinborel},  we can 
perform
 a conjugation 
 so that $X'' = T'' \in \algt$.  Then if 
$X \in \algt$, we are done.  
Otherwise,  
$X''$ and $X$ are linearly independent.  If 
$X = T+N$,  with $N \ne 0$,   
we can write $N = \sum_{i} N_{i}$, where each $N_{i}$ is an eigenvector of $\ad(T'')$ with 
eigenvalue $\lambda_{i}$,  and the $\lambda_{i}$ are distinct.  
Then $\alga$ contains 
\begin{equation}
\begin{array}{llll}
\ad(T'')(X) &=& \sum_{i} \lambda_{i} N_{i},\\
\ad(T'')^{2}(X) &=& \sum_{i} \lambda_{i}^{2} N_{i},\\
\ad(T'')^{3}(X) &=& \sum_{i} \lambda_{i}^{3} N_{i},\\
\ad(T'')^{4}(X) &=& \sum_{i} \lambda_{i}^{4} N_{i}, 
\end{array}
\end{equation}
and hence $\alga$ contains all the $N_{i}$.  

From this we find that $\alga$ contains $T = X - \sum_{i} N_{i}$,  and hence 
$\alga \supseteq \langle T, T'' \rangle = \algt$.  
\end{proof}

\section{Two-dimensional subalgebras of $\soo$}

In this section, we classify the two-dimensional (solvable) subalgebras of $\soo$ 
according to 
two cases:  Subalgebras containing a semisimple element (see Theorems \ref{prop:regSS2} and \ref{lemma:T1-1}), and subalgebras not containing any semisimple elements (see Theorems \ref{lemma:T10} and \ref{2dimNilpp}).  Again, without loss of generality, we assume that each solvable subalgebra is in the Borel subalgebra $\borel$. The results are summarized in Table \ref{twodimuu}.

\subsection{Two-dimensional subalgebras containing a semisimple element}  

\subsubsection{Regular Semisimple Elements}

Suppose $\alga \subset \borel$  is a solvable subalgebra of dimension $2$.  
By Lemma \ref{lemmaCartan},  we  can assume that if it contains semisimple elements,  then $\alga$ 
contains elements of $\algt$,  and that if it contains a Cartan subalgebra,  then it contains and hence 
equals $\algt$.  Suppose 
it contains a regular diagonal element $T_{a,b} = diag(a,b,-a,-b)$,  i.e.,  one such that 
the restriction of ${\rm ad}(T_{a,b})$ to $\algn$ is nonsingular.  It is easily seen that this amounts to   
$a, b  \ne 0, a \ne \pm b$.  Then 
we can assume $b = 1$,  $a \ne 0, \pm 1$.    

It is easy to check that,  for $a \ne 0, \pm 1$,  the eigenvalues of   ${\rm ad}(T_{a,1})$ restricted to 
$\algn$ are distinct, with the single exception of $a = 3$:  ${\rm ad}(T_{3,1})$ has the same 
eigenvalue for $X_{\alpha}$ and $X_{\beta}$.  This means that for any $r,s \in \Cb$, not both zero, 
$\langle T_{3,1} , rX_{\alpha} + s X_{\beta} \rangle$ is two-dimensional.  

\begin{lemma}\label{Lemma:T31}
If $r,s$ are both nonzero, then there is a diagonal element in 
$G$ that conjugates 
$rX_{\alpha} + s X_{\beta} $ to $X_{\alpha} + X_{\beta}$; it also fixes $T_{3,1}$.  
\end{lemma}
\begin{proof}
If $u^{2} = r$, then the diagonal element 
${\rm diag}( \frac{1}{su} , \frac{1}{u}, su, u ) \in Sp(4, \Cb)$ conjugates 
$rX_{\alpha} + s X_{\beta} $ to $X_{\alpha} + X_{\beta}$ and, being diagonal,  fixes $T_{3,1}$.  
\end{proof}

By Lemma \ref{Lemma:T31},  
$\langle T_{3,1} , rX_{\alpha} + s X_{\beta} \rangle  \sim \langle T_{3,1} , X_{\alpha} + X_{\beta} \rangle$,
provided $r,s \ne 0$.  

If $a \ne 3$,  then  ${\rm ad}(T_{a,1})$ has distinct eigenvalues on $\algn$, so the only 
two-dimensional subalgebras of $\borel$ containing $T_{a,1}$ are of the form 
$\langle   T_{a,1} , X_{\gamma} \rangle$,  where $\gamma$ is one of $\alpha, \beta, 
\alpha+ \beta, \alpha+ 2\beta$.  
\medskip

 Let $W\in Sp(4,\mathbb{C})$ be as in Eq. \eqref{Eq:matrixW}. Then 
\begin{equation} \label{eq:Wconjugacy}
\begin{array}{llllll}
WT_{a,b}W^{-1} &=& T_{b,a},  \\
WX_{\alpha}W^{-1}  &=&  X_{\alpha+2\beta}, \\
WX_{\beta}W^{-1}  &=& - X_{\beta}^{t} \notin \borel, \\
WX_{\alpha + \beta}W^{-1}  &= & - X_{\alpha + \beta},  \\
WX_{\alpha +2 \beta}W^{-1}  &=&  X_{ \alpha}. 
\end{array}
\end{equation}

Accordingly,  under the Adjoint action of $P$,  
\begin{equation}
\begin{array}{lllllllll}
\langle T_{a,1}, X_{\alpha + \beta} \rangle &\sim&
\langle T_{1,a}, X_{\alpha + \beta} \rangle
&=& \langle T_{a^{-1},1}, X_{\alpha + \beta} \rangle, \\ 
\langle T_{a,1}, X_{\alpha} \rangle &\sim&
\langle T_{1,a}, X_{\alpha +2 \beta} \rangle
&=& \langle T_{a^{-1},1}, X_{\alpha +2 \beta} \rangle.  
\end{array}
\end{equation}

The element $A$ of Eq.  \eqref{eq:matrixAinSp} conjugates $X_{\beta}$ into 
$X_{\alpha+\beta}$ and $T_{a,1}$ into $- T_{-a,1}$, so 
\begin{equation}
\langle T_{a,1} , X_{\beta} \rangle 
\sim \langle T_{-a,1}, X_{\alpha+\beta} \rangle 
\sim \langle T_{-a^{-1},1}, X_{\alpha +\beta} \rangle 
\sim \langle T_{a^{-1},1}, X_{\beta} \rangle.  
\end{equation}   

Similarly,  if $J$ is the matrix of Eq.  \eqref{eq:matrixJ},  then $AJ$ 
conjugates $T_{a,1}$ to $T_{-a,1}$.  It fixes $X_{\alpha}$ and takes 
$X_{\beta}$, $X_{\alpha + \beta}$, and $X_{\alpha + 2 \beta}$ out of $\borel$.  
In particular,  $\langle T_{a,1} , X_{\alpha} \rangle \sim 
\langle T_{-a,1} , X_{\alpha} \rangle$. 
 
We summarize:
any two-dimensional algebra containing a regular diagonal matrix  
but not $\algt$ 
is conjugate under $Sp(4,\Cb)$ to $\langle T_{3,1} , X_{\alpha} + X_{\beta} \rangle$ or to
$\langle T_{a,1}, X_{\alpha} \rangle$
or $\langle T_{a,1}, X_{\beta} \rangle$,  for some $a \in \Cb$, 
with $a \ne 0, \pm 1$,  
with the understanding that
 $\langle T_{a,1}, X_{\alpha} \rangle \sim 
\langle T_{-a,1}, X_{\alpha} \rangle$ and 
$\langle T_{a,1}, X_{\beta} \rangle \sim \langle T_{a^{-1},1}, X_{\beta} \rangle$. 

Next we consider the possibility of other equivalences between pairs of these 
algebras.  Suppose $X \in \algn$ is such that 
$\langle T_{a,1} , X \rangle$ 
is two-dimensional,  with $a \ne 0, \pm 1$.  
Note that every element of $\langle T_{a,1} , X \rangle$
is of the form $c T_{a,1} + d X$,  for some $c,d \in \Cb$.  If $c \ne 0$, 
this matrix has distinct eigenvalues and hence is semisimple.  However,  if $c=0$,  the matrix 
is nilpotent,  and unless $d=0$,  it has rank equal to the rank of $X$.  
Now ${\rm rank}(X_{\alpha}) = 1$,
${\rm rank}(X_{\beta}) = 2$,  and
${\rm rank}(X_{\alpha} +  X_{\beta}) = 3$,   so
$\langle T_{3,1} , X_{\alpha} + X_{\beta} \rangle$ 
is not equivalent to 
$\langle T_{a,1}, X_{\alpha} \rangle$ or to 
$\langle T_{a,1}, X_{\beta} \rangle$,  for any $a \ne 0, \pm 1$.   
Moreover,  
$\langle T_{a,1}, X_{\alpha} \rangle$ cannot be equivalent to 
$\langle T_{b,1}, X_{\beta} \rangle$,  for any $a,b \ne 0, \pm 1$.

Suppose $\langle T_{a,1} ,  X_{\gamma} \rangle$ is conjugate by 
$g \in Sp(4, \Cb)$ to 
$\langle T_{b,1} ,  X_{\gamma} \rangle$,    
for some $b \in \Cb$,  $b \ne  0, \pm 1$,  where $\gamma = \alpha$ or
$\beta$.  

The elements of $\langle T_{b,1} ,  X_{\gamma} \rangle$ are all of the form 
$cT_{b,1} + dX_{\gamma}$,  for $c,d \in \Cb$,  and the semisimple ones 
are those for which $c \ne 0$.

For $z \in \Cb$,  consider \begin{small}
\begin{equation}\label{eq:rootconj}
\begin{array}{rl}
&(id + z X_{\gamma})(cT_{b,1} + dX_{\gamma})(id + z X_{\gamma})^{-1} \\
&=(id + z X_{\gamma}) (cT_{b,1} + dX_{\gamma})(id - z X_{\gamma}) \\
&= (cT_{b,1} + dX_{\gamma}) + z [X_{\gamma},cT_{b,1} + dX_{\gamma}] 
- z^{2}X_{\gamma} (cT_{b,1} + dX_{\gamma}) X_{\gamma}  \\
&= (cT_{b,1} + dX_{\gamma}) -cz [T_{b,1},X_{\gamma}] 
-c z^{2}X_{\gamma} T_{b,1}  X_{\gamma}  \\
&= (cT_{b,1} + dX_{\gamma}) -cz \cdot \gamma(T_{b,1})X_{\gamma} 
-c z^{2}X_{\gamma} T_{b,1}  X_{\gamma}  \\
&= cT_{b,1} + (d-cz \cdot \gamma(T_{b,1})) X_{\gamma} 
-c z^{2}X_{\gamma} T_{b,1}  X_{\gamma}. 
\end{array}
\end{equation}
\end{small}

A simple calculation shows that the last term is zero,  and,  if $c \ne 0$,  choosing 
$z = \frac{d}{c \gamma(T_{b,1})}$ shows that the semisimple element $cT_{b,1} + dX_{\gamma} 
\in \langle T_{b,1} ,  X_{\gamma} \rangle$ is conjugate to $cT_{b,1}$  
by an element of the form $(id + z X_{\gamma})$,  which 
normalizes $\langle T_{b,1} ,  X_{\gamma} \rangle$.

Since $gT_{a,1}g^{-1}$ is a semisimple element 
of $\langle T_{b,1} , X_{\gamma} \rangle$ there must be  $c,d$ with $c \ne 0$ so that
the eigenvalues $\pm a, \pm 1$ of $T_{a,1}$ equal those of $cT_{b,1}$.  The 
only possibilities are $c = \pm 1$ and $b = \pm a$ or $c = \pm a$ and $b = \pm \frac{1}{a}$.   

This amounts to saying that $T_{a,1}$ can be conjugated to 
$\pm T_{a,1}$,  $\pm T_{-a,1}$, $\pm T_{1,a}$,  or $\pm T_{1,-a}$,  and we have already seen
that all of these are possible.        

Moreover,  the eigenvalue of $\ad(T_{a,1})$ or $\ad(T_{-a,1})$ corresponding to the 
eigenvector $X_{\alpha}$ is $\alpha(T_{a,1} )= 2$ or $\alpha(T_{-a,1}) = 2$,  
while the eigenvalue of $\ad(T_{1,a})$ or $\ad(T_{1,-a})$  is $2a \ne \pm 2$ or $-2a \ne \pm 2$,
respectively.  This shows that there are no equivalences between 
$\langle T_{a,1} , X_{\alpha} \rangle$ and 
$\langle T_{1, \pm a} , X_{\alpha} \rangle$. 

Similarly,  the eigenvalue of $\ad(T_{a,1})$ or $\ad(-T_{1,a})$ corresponding to the 
eigenvector $X_{\beta}$ is $\beta(T_{a,1} )= a-1$ or $\beta(-T_{1,a}) = a-1$,  
while the eigenvalue of $\ad( \pm T_{-a,1})$ or $\ad(\pm T_{1,-a})$  is $\pm (-a-1) \ne a-1$ 
or $\pm (1+a) \ne a-1$,
respectively.  This shows that there are no equivalences between 
$\langle T_{a,1}, X_{\beta} \rangle$ and 
$\langle T_{1, - a}, X_{\beta} \rangle$ or
$\langle T_{- a,1}, X_{\beta} \rangle$. 

We combine the above remarks.

\begin{theorem}\label{prop:regSS2}
Up to equivalence under the 
action of $Sp(4, \Cb)$,  a complete set of representatives for the two-dimensional subalgebras 
containing a regular semisimple element   
is
\begin{equation}
\begin{array}{llll}
&\algt  \\
&\langle T_{3,1}, X_{\alpha} + X_{\beta} \rangle, \\
&\langle T_{a,1}, X_{\alpha} \rangle, \quad  a \ne 0, \pm 1,  \\
&\langle T_{a,1}, X_{\beta} \rangle,  \quad a \ne 0, \pm 1,   
\end{array}
\end{equation}
with the understanding that 
$\langle T_{a,1}, X_{\alpha} \rangle 
\sim \langle T_{-a,1}, X_{\alpha} \rangle 
$
and 
$\langle T_{a,1} , X_{\beta} \rangle 
\sim \langle T_{a^{-1},1} , X_{\beta} \rangle$, but 
that otherwise these algebras are pairwise  not equivalent under $Sp(4, \Cb)$. 

We also note that 
\begin{equation}
\begin{array}{lllll}
\langle T_{a,1} , X_{\alpha} \rangle 
&\sim&  \langle T_{\pm a^{-1},1} , X_{\alpha + 2 \beta} \rangle,~ {\rm and}\\
\langle T_{a,1} , X_{\beta} \rangle 
&\sim&  \langle T_{-a,1} , X_{\alpha +  \beta} \rangle
\sim  \langle T_{- a^{-1},1} , X_{\alpha +  \beta} \rangle. 
\end{array}
\end{equation}
\end{theorem}

\subsubsection{Algebras Containing Non-Regular Semisimple Elements}  \label{subsection.nonreg.ss}

Any non-regular semisimple element of $\algt$ is a scalar multiple 
of $T_{0,1}$,  $T_{1,0}$, $T_{1,1}$,  or $T_{1,-1}$.

\paragraph{$T_{0,1}$}\label{subsubT01dim2}

Suppose a two-dimensional solvable subalgebra $\alga \subset \borel$ contains 
$T_{0,1}$ but not $\algt$. Eigenvectors of $\ad(T_{0,1} )$ in $\algn$ are 
scalar multiples of the root vectors $X_{\alpha}$, $X_{\beta}$, $X_{\alpha + \beta}$, or $X_{\alpha + 2 \beta}$, 
which have the distinct eigenvalues $2, -1, 1$,  and $0$,  respectively.  The corresponding 
algebras are 
$\langle  T_{0,1}, X_{\alpha} \rangle$, 
$\langle  T_{0,1}, X_{\beta} \rangle$, 
$\langle  T_{0,1}, X_{\alpha + \beta} \rangle$, and 
$\langle  T_{0,1}, X_{\alpha + 2 \beta} \rangle$.  

If $A$ is the matrix defined in Eq. \eqref{eq:matrixAinSp}, 
then $A$ conjugates $T_{0,1}$ to $-T_{0,1}$,  $X_{\beta}$ to $X_{\alpha + \beta}$,  and 
$X_{\alpha + \beta}$ to $-X_{\beta}$.  This shows that $\langle  T_{0,1}, X_{\beta} \rangle$ is equivalent to 
$\langle  T_{0,1}, X_{\alpha + \beta} \rangle$,  so it suffices to consider $\langle  T_{0,1}, X_{\beta} \rangle$.  

A comparison of the eigenvalues of $\ad (T_{0,1})$ shows that 
$\langle  T_{0,1}, X_{\alpha} \rangle$  cannot be equivalent to $\langle  T_{0,1}, X_{ \beta} \rangle$.  

We observe that $\langle  T_{0,1}, X_{\alpha + 2 \beta} \rangle$   
is abelian,  so it is inequivalent to $\langle  T_{0,1}, X_{\alpha } \rangle$ and $\langle  T_{0,1}, X_{ \beta} \rangle$.

We summarize in the following lemma.

\begin{lemma}\label{lemma:T01}
The algebras 
$\langle  T_{0,1}, X_{\alpha} \rangle$, 
$\langle  T_{0,1}, X_{\beta} \rangle$, 
and 
$\langle  T_{0,1}, X_{\alpha + 2 \beta} \rangle$ 
are all inequivalent.  
Up to equivalence,  they are representatives of all 
two-dimensional solvable subalgebras $\alga \subset \borel$ containing 
$T_{0,1}$ but not $\algt$.

Also,  $\langle  T_{0,1}, X_{\beta} \rangle \sim  
\langle  T_{0,1}, X_{\alpha + \beta} \rangle$.   Moreover,  $\langle  T_{0,1}, X_{\alpha + 2 \beta} \rangle$  is abelian.  
\end{lemma}

\paragraph{$T_{1,0}$}\label{subsubT10dim2}
Suppose a two-dimensional solvable subalgebra $\alga \subset \borel$ contains 
$T_{1,0}$ but not $\algt$. Eigenvectors of $\ad(T_{1,0} )$ in $\borel$ are 
scalar multiples of $X_{\alpha}$ or $X_{\alpha + 2 \beta}$ or linear combinations of 
$X_{\beta}$ and $X_{\alpha + \beta}$,  with eigenvalues $0,  2,  1$,  respectively.

If 
$\begin{pmatrix}
 a & b \\
c &d 
\end{pmatrix}  \in SL(2, \Cb)$,  then  
the matrix 
\begin{equation} \label{Matrix24rotate}
\begin{pmatrix}
1&0&0&0 \\
0& a &0 & b \\
0&0&1&0 \\
0&c &0 &d 
\end{pmatrix} 
\end{equation}
centralizes $T_{1,0}$ and conjugates $X_{\beta}$ to 
\begin{equation}\begin{pmatrix} 
0& d &0 & - b \\
0&0&-b&0 \\
0&0&0&0 \\
0&0&-d&0 
\end{pmatrix},\end{equation}  
an arbitrary nonzero linear combination of $X_{\beta}$ and $X_{\alpha + \beta}$.  

We find that the algebras 
$\langle  T_{1,0}, cX_{\beta} +dX_{\alpha + \beta}  \rangle$ for any $c, d \in \Cb$, not both zero
are all pairwise equivalent.  In particular,  it suffices to consider 
$\langle  T_{1,0}, X_{\beta}   \rangle$. 

Since $\langle  T_{1,0}, X_{\alpha} \rangle$ is abelian,  it is not equivalent to  
$\langle  T_{1,0}, X_{\beta}  \rangle$ or to 
$\langle  T_{1,0}, X_{\alpha + 2 \beta} \rangle$.   
A comparison of the eigenvalues of $\ad(T_{1,0})$ shows that $\langle  T_{1,0}, X_{\beta} \rangle$ cannot
be equivalent to $\langle  T_{1,0}, X_{\alpha + 2 \beta} \rangle$.

We summarize in the following lemma.

\begin{lemma}\label{lemma:T10}
The algebra 
$\langle  T_{1,0}, X_{\alpha} \rangle$ is abelian.
The algebras
$\langle  T_{1,0}, cX_{\beta} +d X_{\alpha + \beta}  \rangle$ for any $c, d \in \Cb$, not both zero,  
are all equivalent.  
The algebras  $\langle  T_{1,0}, X_{\alpha} \rangle$,  $\langle  T_{1,0}, X_{\beta} \rangle$,
and   
$\langle  T_{1,0}, X_{\alpha + 2 \beta} \rangle$ 
are all inequivalent.  
Up to equivalence,  they are representatives of all 
two-dimensional solvable subalgebras $\alga \subset \borel$ containing 
$T_{1,0}$ but not $\algt$.
\end{lemma}

\paragraph{Equivalences between subalgebras containing $T_{1,0}$ and $T_{0,1}$}

From Eq. \eqref{eq:Wconjugacy},  we see that  
$\langle T_{0,1} , X_{\alpha} \rangle 
\sim \langle T_{1,0} , X_{\alpha + 2 \beta} \rangle$ and 
$\langle T_{1,0} , X_{\alpha} \rangle 
\sim \langle T_{0,1} , X_{\alpha + 2 \beta} \rangle$.  We also see that 
$\langle  T_{1,0}, X_{\beta} \rangle \sim \langle  T_{1,0}, X_{\alpha + \beta} \rangle \sim   \langle  T_{0,1}, X_{\alpha + \beta} \rangle$.

Since $\langle T_{1,0} , X_{\alpha} \rangle$  is abelian and $\langle T_{0,1} , X_{\alpha} \rangle $ is not,  they are inequivalent.  
A comparison of the eigenvalues of $\ad(T_{1,0})$ and $\ad(T_{0,1})$ shows that neither 
$\langle  T_{1,0}, X_{\beta} \rangle$ nor $\langle  T_{0,1}, X_{\beta} \rangle$ can
be equivalent to $\langle  T_{1,0}, X_{\alpha } \rangle$, 
$\langle  T_{1,0}, X_{\alpha + 2 \beta} \rangle$,
$\langle  T_{0,1}, X_{\alpha} \rangle$, or 
$\langle  T_{0,1}, X_{\alpha + 2 \beta} \rangle$.

We summarize in the following lemma.

\begin{lemma}\label{lemma:T10T01}
A two-dimensional solvable algebra containing $T_{0,1}$ or $T_{1,0}$ 
is equivalent to one of the following: 
\begin{equation}
\begin{array}{lllllll}
\algt, \qquad \langle  T_{1,0}, X_{\alpha} \rangle, &   \langle  T_{1,0}, X_{\beta} \rangle, & \langle  T_{1,0}, X_{\alpha  +2 \beta} \rangle.    
\end{array}
\end{equation}
These algebras are all pairwise inequivalent.
The algebra $\langle  T_{1,0}, X_{\alpha} \rangle$ is abelian.
\end{lemma}

\paragraph{$T_{1,1}$}\label{subsubT11dim2}

Suppose a two-dimensional solvable subalgebra $\alga \subset \borel$ contains 
$T_{1,1}$ but not $\algt$. Eigenvectors of $\ad(T_{1,1} )$ in $\borel$ are 
scalar multiples of $X_{\beta}$ or linear combinations of 
$X_{\alpha}$,  $X_{\alpha + \beta}$, and $X_{\alpha + 2 \beta}$.  In other words,
the $2$-eigenspace of $\ad(T_{1,1} )$ in $\borel$ is $\nf_{\p}$.  

Conjugating a matrix $T_{1,1} + N$,  with $N \in \algn_{\p}$,  by a matrix of the block form $\begin{pmatrix}
g & 0 \\ 
0 & g^{-t} \\
\end{pmatrix}$
fixes $T_{1,1}$ and takes $N$ to $gNg^{t}$,  so by Lemma $4.3$,
every subalgebra of the form 
$\langle  T_{1,1}, N \rangle$,  with $0 \ne N \in \algn_{\p}$,  is equivalent to either
$\langle  T_{1,1}, X_{\alpha} \rangle$ or $\langle  T_{1,1}, X_{\alpha+ \beta} \rangle$,  
and these two subalgebras are inequivalent   
because their nilpotent elements have different ranks.  

Since these subalgebras are nonabelian and $\langle  T_{1,1}, X_{\beta} \rangle$ is abelian,
the algebras 
$\langle  T_{1,1}, X_{\alpha} \rangle$, 
$\langle  T_{1,1}, X_{\beta} \rangle$, and
$\langle  T_{1,1}, X_{\alpha +  \beta} \rangle$
are inequivalent.  Any semisimple element in any of these subalgebras has eigenvalues 
that are a multiple of the eigenvalues of $T_{1,1}$.   Similarly, any semisimple element in 
any of the subalgebras mentioned in Lemma \ref{lemma:T10T01} has eigenvalues 
that are a multiple of the eigenvalues of $T_{1,0}$.
So none of these subalgebras containing $T_{1,1}$ can be equivalent to any of 
the subalgebras mentioned in Lemma \ref{lemma:T10T01}.

Finally,  the $2$-eigenspace of $\ad(T_{1,1})$ is $\algn_{\p}$,  and the equivalence class of 
$\langle  T_{1,1}, N \rangle$, for $N \in \algn_{\p}$,  is determined by ${\rm rank} (N)$.
We summarize.

\begin{lemma} \label{lemma:T11}
Every two-dimensional solvable subalgebra containing $T_{1,1}$ but not $\algt$ 
is equivalent to one of the inequivalent subalgebras 
$\langle  T_{1,1}, X_{\alpha} \rangle$, $\langle  T_{1,1}, X_{\beta} \rangle$, or $\langle  T_{1,1}, X_{\alpha+ \beta} \rangle$.  
None of these subalgebras is equivalent to any of 
the subalgebras mentioned in Lemma \ref{lemma:T10T01} or in Theorem \ref{prop:regSS2}.

We also note that,  for $a,b, c \in \Cb$, not all zero,  $\langle  T_{1,1}, a X_{\alpha} +b X_{\alpha + \beta} + cX_{\alpha+ 2\beta} \rangle \sim \langle  T_{1,1}, X_{\alpha + \beta} \rangle$ unless $  a X_{\alpha} +b X_{\alpha + \beta} + cX_{\alpha+ 2\beta} $ has rank $1$,  in which case
$\langle  T_{1,1}, a X_{\alpha} +b X_{\alpha + \beta} + cX_{\alpha+ 2\beta} \rangle \sim \langle  T_{1,1}, X_{\alpha } \rangle$. 
In particular,   $\langle  T_{1,1}, X_{\alpha+ 2\beta} \rangle \sim \langle  T_{1,1}, X_{\alpha} \rangle$.
\end{lemma}


\paragraph{$T_{1,-1}$}\label{subsubT1-1dim2}

The $Sp(4, \Cb)$ matrix $A$ in Eq. \eqref{eq:matrixAinSp} conjugates $T_{1,1}$ to $T_{1,-1}$.  

It also takes $X_{\beta}$ to $-X_{\alpha + \beta}$ and vice versa,  fixes $X_{\alpha+2\beta}$,  and takes 
$X_{\alpha}$ out of $\borel$.

The $Sp(4, \Cb)$ matrix $AJ$ conjugates $T_{1,1}$ to $- T_{1,-1}$.  
It also fixes $X_{\alpha}$ and takes $X_{\beta}$, $X_{\alpha + \beta}$, and 
$X_{\alpha+2\beta}$ out of $\borel$.

We conclude that 
\begin{equation}
\begin{array}{llllllll}
\langle  T_{1,1}, X_{\alpha} \rangle &\sim & \langle  T_{1,-1}, X_{\alpha} \rangle \\
\langle  T_{1,1}, X_{\beta} \rangle   &\sim & \langle  T_{1,-1}, X_{\alpha+ \beta} \rangle \\
\langle  T_{1,1}, X_{\alpha+ \beta} \rangle  &\sim &   \langle  T_{1,-1}, X_{\beta} \rangle \\
\langle  T_{1,1}, X_{\alpha+ 2\beta} \rangle  &\sim &   \langle  T_{1,-1}, X_{\alpha+ 2\beta} \rangle.  
\end{array}
\end{equation} 
Since $\langle  T_{1,1}, X_{\alpha+ 2\beta} \rangle  \sim \langle  T_{1,1}, X_{\alpha} \rangle$,  we conclude 
that $ \langle  T_{1,-1}, X_{\alpha+ 2\beta} \rangle \sim \langle  T_{1,-1}, X_{\alpha} \rangle$.  
We also note that for $c \ne 0$,  $\langle  T_{1,-1}, X_{\beta} + c X_{\alpha+ 2\beta} \rangle \sim   \langle  T_{1,-1}, X_{\beta} \rangle$.
We summarize.

\begin{theorem} \label{lemma:T1-1}
Every two-dimensional solvable subalgebra containing $T_{1,1}$ or $T_{1,-1}$ but not $\algt$ 
is equivalent to one of the inequivalent subalgebras 
$\langle  T_{1,1}, X_{\alpha} \rangle$, $\langle  T_{1,1}, X_{\beta} \rangle$, or $\langle  T_{1,1}, X_{\alpha+ \beta} \rangle$.  
None of these subalgebras is equivalent to any of 
the subalgebras mentioned in Lemma \ref{lemma:T10T01} or in Theorem \ref{prop:regSS2}.

Moreover,  $\langle T_{1,-1}, X_{\alpha+ 2\beta} \rangle \sim \langle  T_{1,1}, X_{\alpha} \rangle$,  and, 
for any $c \in \Cb$,        
$\langle  T_{1,-1}, X_{\beta}$ $+ c X_{\alpha+ 2\beta} \rangle \sim   \langle  T_{1,-1}, X_{\beta} \rangle$. 
\end{theorem}

\subsection{Two-dimensional subalgebras not containing any semisimple elements}

\subsubsection{Two-dimensional algebras containing a nonsemisimple element that is not nilpotent}
Suppose $\alga$ is a two-dimensional subalgebra of $\borel$ which does not contain any 
semisimple elements,  but whose elements are not all nilpotent.  

If $\alga = \langle X, X' \rangle$,  where $X = T+N$,  $X' = T' + N'$,  with $T, T' \in \algt$ and $N,N' \in \algn$,  
then we can assume $T \ne 0$.  If $T'$ is not a scalar multiple of $T$,  then $T$ and $T'$ span $\algt$,  and some linear 
combination of $X, X'$ would have distinct eigenvalues and therefore would be semisimple.  So $T'$ must be a multiple of $T$, 
and $X' $ can be replaced by $X'' = N'' \in \algn$.  Since $\dim(\alga) = 2$,  $[X, X'']$ must be a scalar multiple of $X''$.  

Without loss of generality,  we can assume that $T$ has the form of one of the first four 
classes under the action of $B$ in Table \ref{tab:nonss}. In fact,  we can even assume that the parameter $a$ in these classes 
is equal to $1$.  

For example,  it is clear that $X = T_{1,1}+X_{\beta}$ commutes with $X_{\beta}$,  so 
$\langle T_{1,1}+X_{\beta}, X_{\beta} \rangle = \langle T_{1,1}, X_{\beta} \rangle$ is a commutative two-dimensional algebra 
containing a semisimple element.  
A simple calculation shows that the only other possibility for $X'' = N'' \in \algn$ satisfying $[X, X''] = vX''$, 
for some $v \in \Cb$, is for $X''$ to be a multiple of $X_{\alpha + 2 \beta}$.  
So $\langle T_{1,1}+X_{\beta}, X_{\alpha + 2 \beta} \rangle$ is a noncommutative two-dimensional algebra.

Similarly,  it is clear that $X = T_{1,-1}+X_{\alpha +\beta}$ commutes with $X_{\alpha + \beta}$,  so 
$\langle T_{1,-1}+X_{\alpha + \beta}, X_{\alpha + \beta} \rangle = \langle T_{1,-1}, X_{\alpha + \beta} \rangle$ 
is a commutative two-dimensional algebra containing a semisimple element.  
A simple calculation shows that the only other possibilities for $X'' = N'' \in \algn$ satisfying $[X, X''] = vX''$, 
for some $v \in \Cb$, are for $X''$ to be a multiple of $X_{\alpha}$ or of $X_{\alpha + 2 \beta}$.  
So $\langle T_{1,-1}+X_{\alpha + \beta}, X_{\alpha } \rangle$ 
and $\langle T_{1,-1}+X_{\alpha + \beta}, X_{\alpha + 2 \beta} \rangle$ are noncommutative two-dimensional algebras.  
Note that they are conjugate by the element $W$ of Eq. \eqref{Eq:matrixW}.

We will use the following lemma. 
\begin{lemma}\label{lemma:Aconjugacy}
The $Sp(4, \Cb)$ matrix $A$ in Eq. \eqref{eq:matrixAinSp} conjugates 
\begin{equation}
\begin{pmatrix}
\lambda &1&0&0 \\
0& \lambda &0&0 \\ 
0&0& - \lambda &0 \\
0&0&-1& -\lambda \\
\end{pmatrix}
~\text{into}~ 
\begin{pmatrix}
\lambda &0&0&1 \\
0& - \lambda &1&0 \\ 
0&0& - \lambda &0 \\
0&0&0& \lambda \\
\end{pmatrix},
\end{equation}
  for any $\lambda \in \Cb$.  
\end{lemma}

\begin{proof}
Straightforward calculation.
\end{proof}

Note too that by Lemma \ref{lemma:Aconjugacy} and an easy calculation, 
the $Sp(4, \Cb)$ matrix $A$ in Eq.~(\ref{eq:matrixAinSp}) conjugates 
$\langle T_{1,1}+X_{\beta}, X_{\alpha + 2 \beta} \rangle$ 
to 
$\langle T_{1,-1}+X_{\alpha + \beta},$ $X_{\alpha + 2 \beta} \rangle$
$\sim$ $\langle T_{1,-1}+X_{\alpha + \beta}, X_{\alpha } \rangle$.   

For $X = T_{1,0} + X_{\alpha}$,  clearly 
$\langle T_{1,0} + X_{\alpha}, X_{\alpha} \rangle = \langle T_{1,0} , X_{\alpha} \rangle$ is a commutative two-dimensional 
algebra containing a semisimple element.  
A simple calculation shows that the only other possibilities for $X'' = N'' \in \algn$ satisfying $[X, X''] = vX''$, 
for some $v \in \Cb$, are for $X''$ to be a multiple of $X_{\alpha+\beta}$ or of $X_{\alpha + 2 \beta}$.  
So $\langle T_{1,0}+X_{\alpha}, X_{\alpha + \beta} \rangle$   
and $\langle T_{1,0}+X_{\alpha}, X_{\alpha + 2 \beta} \rangle$,    are noncommutative two-dimensional algebras.  
Note that they are not conjugate by any element of $Sp(4, \Cb)$ because the elements which 
have eigenvalues $1,-1,0,0$,  i.e.,  $\pm T_{1,0}+N$,  for some $N \in \algn$,  have eigenvalue
$\pm 1$ on  $X_{\alpha + \beta}$ and 
eigenvalue $\pm 2$ on  $X_{\alpha + 2\beta}$.  


For $X = T_{0,1} + X_{\alpha + 2 \beta}$,  clearly 
$\langle T_{0,1} + X_{\alpha + 2 \beta}, X_{\alpha + 2 \beta} \rangle = \langle T_{0,1} , X_{\alpha + 2 \beta} \rangle$ 
is a commutative two-dimensional algebra containing a semisimple element.    
A simple calculation shows that the only other possibilities for $X'' = N'' \in \algn$ satisfying $[X, X''] = vX''$, 
for some $v \in \Cb$, are for $X''$ to be a multiple of $X_{\alpha}$,  $X_{\beta}$,  or $X_{\alpha+\beta}$.  
So $\langle T_{0,1}+X_{\alpha + 2 \beta}, X_{\alpha } \rangle$, 
$\langle T_{0,1}+X_{\alpha + 2 \beta}, X_{ \beta} \rangle$,      
and $\langle T_{0,1}+X_{\alpha + 2 \beta}, X_{\alpha +  \beta} \rangle$ are noncommutative two-dimensional algebras.  
Note that the last two of these are conjugate by the element $A$ of Eq. \eqref{eq:matrixAinSp}.  

However,  
$\langle T_{0,1}+X_{\alpha + 2 \beta}, X_{\alpha } \rangle$ 
and $\langle T_{0,1}+X_{\alpha + 2 \beta}, X_{ \beta} \rangle$     
are not conjugate by any element of $Sp(4, \Cb)$ because the elements which 
have eigenvalues $1,-1,0,0$,  i.e.,  $\pm T_{0,1}+N$,  for some $N \in \algn$,  have eigenvalue
$\pm 2$ on  $X_{\alpha}$   
and eigenvalue $\pm 1$ on  $X_{ \beta}$. 

Moreover,  the element $W$ of Eq. \eqref{Eq:matrixW} conjugates 
$\langle T_{0,1}+X_{\alpha + 2 \beta}, X_{\alpha } \rangle$ 
to
$\langle T_{1,0}+X_{\alpha }, X_{\alpha + 2\beta } \rangle$  
and $ \langle T_{0,1}+X_{\alpha + 2 \beta}, X_{ \alpha + \beta} \rangle 
\sim \langle T_{0,1}+X_{\alpha + 2 \beta}, X_{ \beta} \rangle$     
to
$\langle T_{1,0}+X_{\alpha }, X_{\alpha + \beta } \rangle$.   

Finally,  note that none of these two-dimensional algebras containing a singular matrix $T_{1,0}+X_{\alpha}$
or $T_{0,1} + X_{\alpha + 2 \beta}$,  but no semisimple elements, can be conjugate to any 
of the algebras containing an invertible matrix $T_{1,1}+X_{\beta}$ or 
$T_{1,-1}+X_{\alpha + \beta}$.  

\begin{theorem} \label{lemma:T10}
A two-dimensional subalgebra of $\borel$ which does not contain any 
semisimple elements,  but whose elements are not all nilpotent,  must be 
equivalent to one of the following inequivalent algebras:   
\begin{equation}
\begin{array}{llllll}
 \langle T_{1,1}+X_{\beta}, X_{\alpha + 2 \beta} \rangle, &
 \langle T_{1,0}+X_{\alpha }, X_{\alpha + \beta } \rangle, & 
 \langle T_{1,0}+X_{\alpha }, X_{\alpha + 2\beta } \rangle.   
\end{array}
\end{equation}
\end{theorem}

\subsubsection{Two-dimensional algebras whose elements are all nilpotent}
Suppose $\alga \subseteq \g$ is a two-dimensional algebra whose elements are all nilpotent.
If $\alga \subset \nf_{\p}$,  then there are exactly two possible classes,  as listed in Lemma 
\ref{lemma:2dimSubsp.np}: 
\begin{equation}  
\langle X_{\alpha}, X_{\alpha + \beta} \rangle, \quad 
\langle X_{\alpha}, X_{\alpha + 2 \beta} \rangle.
\end{equation}  

Otherwise,  we can assume that $\alga$ is generated by 
$X = X_{\beta} + s X_{\alpha} + t X_{\alpha + \beta}  + u X_{\alpha + 2\beta} $ 
and 
$X'' =  r X_{\alpha} + v X_{\alpha + \beta}  + w X_{\alpha + 2\beta} $. 
A simple calculation shows that the only way that $[X, X'']$ can be a 
scalar multiple of $X''$ is for $X''$ to be a (nonzero) multiple of $X_{\alpha+ 2 \beta}$,  
in which case $\alga$ is commutative.  

Now $\langle X_{\beta} + s X_{\alpha} + t X_{\alpha + \beta}  + u X_{\alpha + 2\beta},$ $X_{\alpha+ 2 \beta} \rangle
\sim  \langle X_{\beta} + s X_{\alpha} + t X_{\alpha + \beta},$ $X_{\alpha+ 2 \beta} \rangle$, 
and $Id + tX_{\alpha}$ conjugates the latter to 
$\langle X_{\beta} + s X_{\alpha}  , X_{\alpha+ 2 \beta} \rangle$.  Then,  if $s \ne 0$ and $z^{2} = s$, ${\rm diag}(1/z, 1/z, z, z)$ 
conjugates this into $\langle X_{\beta} + X_{\alpha}  , X_{\alpha+ 2 \beta} \rangle$.  

On  the other hand,  if $s = 0$,  then $\alga = \langle X_{\beta} , X_{\alpha+ 2 \beta} \rangle$.  
If $A$ and $W$ are the matrices in Eqs.  \eqref{eq:matrixAinSp} and \eqref{Eq:matrixW},  respectively,  then 
$WA$ conjugates 
$\langle X_{\beta}, X_{\alpha+ 2 \beta} \rangle$ to 
$\langle X_{\alpha}, X_{\alpha+  \beta} \rangle$.   

The algebra $\langle X_{\beta} + X_{\alpha}  , X_{\alpha+ 2 \beta} \rangle$ contains elements 
of rank 3,  so it is inequivalent to either of the algebras contained in $\nf_{\p}$,  whose elements 
have rank 2 or less.  

Note that $ \langle X_{\alpha} ,  X_{\alpha + 2 \beta}  \rangle$ contains two lines consisting of elements of rank $1$,
while $  \langle X_{\alpha}, X_{\alpha + \beta} \rangle$ contains only one such line.  Hence they are inequivalent. 
And all the elements in either of them have rank at most $2$, whereas 
$\langle X_{\beta} + X_{\alpha}  , X_{\alpha+ 2 \beta} \rangle $ contains elements of rank 3.  Accordingly, these 
three algebras are pairwise inequivalent.  

We summarize.  

\begin{theorem} \label{2dimNilpp}
Each two-dimensional subalgebra whose elements are all nilpotent is conjugate to 
exactly one of the following inequivalent abelian algebras: 
\begin{equation}
\begin{array}{llllll}
& \langle X_{\alpha}, X_{\alpha + \beta} \rangle, &
& \langle X_{\alpha},  X_{\alpha + 2 \beta}  \rangle,&
& \langle X_{\beta} + X_{\alpha}, X_{\alpha+ 2 \beta} \rangle.
\end{array}
\end{equation}
\end{theorem}

\section{Three-dimensional solvable subalgebras of $\soo$}  

In this section, we classify the three-dimensional  solvable subalgebras of $\soo$, dividing them into four
cases:  
Three-dimensional subalgebras containing a Cartan subalgebra (see Lemma \ref{3dimWithCartan});
three-dimensional subalgebras containing a semisimple element but not a Cartan subalgebra (see Theorems \ref{2dimNilp}, \ref{lemma:singularss:dim3} and \ref{lemma:singularss:dim3t});  non-nilpotent three-dimensional solvable algebras containing no semisimple elements (see Theorem \ref{prop:nonSSnonNilpDim3}); and three-dimensional nilpotent subalgebras (see Theorem \ref{prop:NilpDim3}).  Again, without loss of generality, we assume that each solvable subalgebra is in the Borel subalgebra $\borel$. The results are summarized in Table \ref{threedimuu}. We begin with a lemma.

\begin{lemma} \label{3dimWithCartan}
Each three-dimensional solvable algebra containing a Cartan subalgebra is conjugate to 
exactly one of the following inequivalent subalgebras: 
\begin{equation}
\langle \algt, X_{\alpha} \rangle, \quad 
 \langle \algt, X_{\beta} \rangle.  
\end{equation}  
Moreover,  $\langle \algt, X_{\alpha} \rangle 
\sim \langle \algt, X_{\alpha + 2 \beta} \rangle$ 
and 
$\langle \algt, X_{\beta} \rangle \sim \langle \algt, X_{\alpha + \beta} \rangle$.  
\end{lemma}  

\begin{proof}
We can assume that such an algebra is contained in $\borel$ and contains $\algt$.  Then it must contain exactly 
one of the positive root vectors.  

But the matrix $W$ of Eq. \eqref{Eq:matrixW}
conjugates $\langle \algt, X_{\alpha} \rangle$ to 
$\langle \algt, X_{\alpha + 2 \beta} \rangle $, and  
the matrix $A$ of Eq. \eqref{eq:matrixAinSp}  
conjugates $ \langle \algt, X_{\beta} \rangle$ to 
$\langle \algt, X_{\alpha + \beta} \rangle$.  

Since $ \langle \algt, X_{\beta} \rangle$ contains nilpotent elements 
of rank $2$ while the nilpotent elements of 
$\langle \algt, X_{\alpha} \rangle$ 
all have rank $1$,  these two algebras are 
inequivalent.  
\end{proof}

\subsection{Three-dimensional subalgebras containing a semisimple element but not a Cartan subalgebra}

\subsubsection{Regular semisimple elements}\label{subsubsection:regss3d}

Suppose $\alga$ is a three-dimensional solvable subalgebra of $\borel$ which contains a regular 
semisimple element but not $\algt$.  Then we can assume it 
contains $T_{a,1}$,   for some $a \ne 0, \pm 1$,   but not $\algt$.  
By Lemma \ref{lemmaFullss},  we can assume $\alga = \langle T_{a,1}, N, N' \rangle$,  with $N, N' \in \algn$.  

If $a \ne 3$,  the only possibilities are for $N, N'$ to be two commuting positive root vectors:
\begin{equation}
\begin{array}{llllll}
 \langle  T_{a,1} , X_{\alpha}, X_{\alpha+ \beta} \rangle, &
\langle T_{a,1} , X_{\alpha}, X_{\alpha+ 2\beta} \rangle, \\
\langle T_{a,1} , X_{\alpha+ \beta}, X_{\alpha+ 2\beta} \rangle, &
\langle T_{a,1} , X_{ \beta}, X_{\alpha+ 2\beta} \rangle. 
\end{array}
\end{equation}

If $a=3$,  there is the additional possibility $\langle T_{3,1} , X_{\alpha}+ X_{\beta}, X_{\alpha+ 2\beta} \rangle$.
It is inequivalent to any of the others because it contains the nilpotent element 
$X_{\alpha}+ X_{\beta}$, which has rank $3$,  while all the others contain nilpotent 
elements of rank at most $2$.  

The element $W$ of Eq. \eqref{Eq:matrixW} conjugates 
$ \langle  T_{a,1} , X_{\alpha}, X_{\alpha+ \beta} \rangle$ to 
$\langle T_{a^{-1},1},$ $X_{\alpha+ \beta}, X_{\alpha+ 2\beta} \rangle$ 
and
$ \langle  T_{a,1} , X_{\alpha}, X_{\alpha+ 2\beta} \rangle$ to 
$\langle T_{a^{-1},1} , X_{\alpha},$ $X_{\alpha+ 2\beta} \rangle$.
Also,  the matrix $A$ of Eq. \eqref{eq:matrixAinSp}  
conjugates 
$ \langle  T_{a,1} , X_{\beta}, X_{\alpha+ 2 \beta} \rangle$ to 
$ \langle  T_{-a,1} , X_{\alpha + \beta},$ $X_{\alpha+ 2 \beta} \rangle$. 
Accordingly,   any three-dimensional solvable subalgebra containing a regular semisimple element 
but not all of $\algt$ must be equivalent to one of the following:  
\begin{equation}
\begin{array}{llllll}
 \langle  T_{a,1} , X_{\alpha}, X_{\alpha+ \beta} \rangle , \quad a \ne 0, \pm 1  \\
\langle T_{a,1} , X_{\alpha}, X_{\alpha+ 2\beta} \rangle , \quad a \ne 0, \pm 1   \\
\langle T_{3,1} , X_{\alpha}+ X_{\beta}, X_{\alpha+ 2\beta} \rangle.   
\end{array}
\end{equation}

However,  we do know that for $a \ne 0, \pm 1$,  
$ \langle  T_{a,1} , X_{\alpha}, X_{\alpha+ 2\beta} \rangle$ 
$\sim$  $\langle  T_{a^{{-1}},1},$ $X_{\alpha}, X_{\alpha+ 2 \beta} \rangle$.
Otherwise,  we claim that these algebras are inequivalent.  Indeed,  
by consideration of the ranks of the nilpotent elements,  we have seen that the 
algebra in the last line is inequivalent to any of the others. 
Moreover,  those in the first line contain nilpotent elements of rank $2$ except 
for those on the single line spanned by $X_{\alpha}$, while  
those in the second line contain nilpotent elements of rank $2$ except 
for those on the two lines spanned by $X_{\alpha}$ and by $X_{{\alpha + 2 \beta}}$.  
None of the algebras on the first line can be equivalent to any of the algebras 
on the second,  even for different values of $a \ne 0, \pm 1$. 

So consider $\alga =  \langle  T_{a,1} , X_{\alpha}, X_{\alpha+ \beta} \rangle $, with $a \ne 0, \pm 1$.   
The only other elements of $\algt$ with the same eigenvalues as $T_{a,1}$ are $- T_{ a,1}$, $\pm T_{- a,1}$, $\pm T_{1, \pm a}$.  
So we can assume that if $g \in Sp(4, \Cb)$ conjugates $\alga$ to 
an algebra $ \langle  T_{b,1} , X_{\alpha}, X_{\alpha+ \beta} \rangle $,  
other than 
$\alga$ itself, then it conjugates $T_{a,1}$ to one of these seven matrices.  
The vectors $X_{\alpha}$ and $X_{\alpha+ \beta}$ are eigenvectors of $\ad(T_{a,1})$ with 
eigenvalues $2$ and $a+1$,  respectively.  The element $g \in Sp(4, \Cb)$ must conjugate eigenvectors 
to eigenvectors.  But since they have different ranks,  it must conjugate 
each to a multiple of itself.

The only one of the seven matrices above whose adjoint has eigenvalue $2$ on $X_{\alpha}$ is 
$T_{-a,1}$,  and its eigenvalue on $X_{\alpha+ \beta}$ is $1-a$,  which cannot equal $1+a$, 
since $a \ne 0$.  We conclude that the algebras  $\langle  T_{a,1} , X_{\alpha}, X_{\alpha+ \beta} \rangle $, 
with $a \ne 0, \pm 1$,  are pairwise inequivalent.    

Now consider $\alga =  \langle  T_{a,1} , X_{\alpha}, X_{\alpha+ 2 \beta} \rangle $, with $a \ne 0, \pm 1$.   
The vectors $X_{\alpha}$ and $X_{\alpha+ 2 \beta} \rangle $ are eigenvectors of $\ad ( T_{a,1})$ 
with eigenvalues $2, 2a$,  respectively.  The only one of the seven conjugates of $T_{a,1}$ in 
$\algt$ with the same adjoint eigenvalues is $T_{1,a}$,  corresponding to the equivalence we have 
already seen arising from conjugation by $W$.    

\begin{theorem} \label{2dimNilp}
Each three-dimensional solvable subalgebra which contains a regular semisimple element but not a 
Cartan subalgebra is conjugate to one of the following: 
\begin{equation}
\begin{array}{llllll}
\langle  T_{a,1} , X_{\alpha}, X_{\alpha+ \beta} \rangle,  \quad a \ne 0, \pm 1,  \\
\langle T_{a,1} , X_{\alpha}, X_{\alpha+ 2\beta} \rangle \sim \langle T_{a^{-1},1} , X_{\alpha}, X_{\alpha+ 2\beta} \rangle, \quad a \ne 0, \pm 1,  \\
\langle T_{3,1} , X_{\alpha}+ X_{\beta},  X_{\alpha+ 2\beta} \rangle. 
\end{array}
\end{equation} 
Apart from the equivalences noted,  the algebras listed above are pairwise inequivalent.  
\end{theorem}  

\subsubsection{Singular semisimple elements}

Suppose $\alga$ is a three-dimensional solvable subalgebra of $\borel$ which contains 
a singular semisimple element but not a Cartan subalgebra.  
We can assume $\alga$ contains one of 
$T_{1,1}$,  $T_{1,-1}$, $T_{1,0}$, or $T_{0,1}$, but not $\algt$.    

\paragraph{$T_{0,1}$}\label{subsubT01dim3}

Suppose a three-dimensional solvable subalgebra $\alga \subset \borel$ contains 
$T_{0,1}$ but not $\algt$. The eigenvectors of $\ad(T_{0,1} )$ in $\algn$ are 
scalar multiples of the root vectors $X_{\alpha}$, $X_{\beta}$, $X_{\alpha + \beta}$, or $X_{\alpha + 2 \beta}$, 
which have the distinct eigenvalues $2, -1, 1$,  and $0$,  respectively.  The only possible three-dimensional algebras  
are 
$\langle  T_{0,1}, X_{\alpha}, X_{\alpha + \beta} \rangle$, 
$\langle  T_{0,1}, X_{\alpha}, X_{\alpha + 2\beta} \rangle$, 
$\langle  T_{0,1}, X_{\alpha + \beta}, X_{\alpha + 2\beta} \rangle$,  and
$\langle  T_{0,1}, X_{\beta}, X_{\alpha + 2\beta} \rangle$. 

If $A$ is the matrix defined in Eq.  \eqref{eq:matrixAinSp}, 
then $A$ conjugates $T_{0,1}$ to $-T_{0,1}$,  $X_{\beta}$ to $X_{\alpha + \beta}$,   
$X_{\alpha + \beta}$ to $-X_{\beta}$, and $X_{\alpha + 2\beta}$ to itself.  This shows that 
$\langle  T_{0,1}, X_{\alpha + \beta}, X_{\alpha + 2\beta} \rangle$ is equivalent to 
$\langle  T_{0,1}, X_{\beta}, X_{\alpha + 2\beta} \rangle$.  

A comparison of the eigenvalues of $\ad (T_{0,1})$ shows that the three 
remaining algebras are inequivalent.  

We summarize in the following lemma.

\begin{lemma}\label{lemma:T01:dim3}
Every solvable three-dimensional subalgebra of  $\borel$ which contains $T_{0,1}$ but not $\algt$ is 
equivalent to one of the following inequivalent algebras:  
\begin{equation}
\begin{array}{llllll}
\langle  T_{0,1}, X_{\alpha}, X_{\alpha + \beta} \rangle,  &
\langle  T_{0,1}, X_{\alpha}, X_{\alpha + 2\beta} \rangle, &
\langle  T_{0,1}, X_{\alpha + \beta}, X_{\alpha + 2\beta} \rangle. 
\end{array}
\end{equation}
Also,  
$\langle  T_{0,1}, X_{\alpha + \beta}, X_{\alpha + 2\beta} \rangle
\sim \langle  T_{0,1}, X_{\beta}, X_{\alpha + 2\beta} \rangle$.  
\end{lemma}

\paragraph{$T_{1,0}$}
Suppose a three-dimensional solvable subalgebra $\alga \subset \borel$ contains 
$T_{1,0}$ but not $\algt$. From the eigenvalues and eigenvectors of $\ad(T_{1,0} )$ in $\algn$ 
as described in \S\ref{subsubT10dim2}, we find that 
the only possible three-dimensional algebras  
are 
$\langle  T_{1,0}, X_{\alpha}, X_{\alpha + \beta} \rangle$,  
$\langle  T_{1,0}, X_{\alpha}, X_{\alpha + 2 \beta} \rangle$, and
$\langle  T_{1,0}, aX_{\beta} + bX_{\alpha + \beta}, X_{\alpha + 2\beta} \rangle$.
\medskip  

The matrix defined in Eq. \eqref{Matrix24rotate} centralizes $T_{1,0}$   
and conjugates $X_{\beta}$ to an arbitrary linear combination of $X_{\beta}$
and $X_{\alpha+ \beta}$. This shows that 
$\langle  T_{1,0}, aX_{\beta} + bX_{\alpha + \beta}, X_{\alpha + 2\beta} \rangle 
\sim \langle  T_{1,0}, X_{\beta} , X_{\alpha + 2\beta} \rangle$,  for any $a,b \in \Cb$, 
not both zero. 

A comparison of the eigenvalues of $\ad (T_{1,0})$ shows that there are no 
additional equivalences.  
 
We summarize in the following lemma.

\begin{lemma}\label{lemma:T10:dim3}
Every solvable three-dimensional subalgebra of $\borel$ which contains $T_{1,0}$ but not $\algt$ is 
equivalent to one of the following inequivalent algebras:  
\begin{equation}
\begin{array}{llllll}
\langle  T_{1,0}, X_{\alpha}, X_{\alpha + \beta} \rangle,  &
\langle  T_{1,0}, X_{\alpha}, X_{\alpha + 2 \beta} \rangle, &
\langle  T_{1,0}, X_{\alpha + \beta}, X_{\alpha + 2\beta} \rangle. 
\end{array}
\end{equation}
Also,  for any $a,b \in \Cb$, not both zero,  
$\langle  T_{1,0}, aX_{\beta} + bX_{\alpha + \beta}, X_{\alpha + 2\beta} \rangle 
\sim \langle  T_{1,0}, X_{\beta}, X_{\alpha + 2\beta} \rangle 
\sim \langle  T_{1,0}, X_{\alpha + \beta}, X_{\alpha + 2\beta} \rangle
$.  
\end{lemma}

\begin{theorem}\label{lemma:singularss:dim3}
Every solvable three-dimensional subalgebra of  $\borel$ which contains a semisimple element 
which has zero as an eigenvalue but does not contain all of $\algt$ is 
equivalent to one of the following inequivalent algebras:  
\begin{equation}
\begin{array}{llllll} 
\langle  T_{1,0}, X_{\alpha}, X_{\alpha + \beta} \rangle,  & 
\langle  T_{1,0}, X_{\alpha}, X_{\alpha + 2 \beta} \rangle, &
\langle  T_{1,0}, X_{\alpha + \beta}, X_{\alpha + 2\beta} \rangle. 
\end{array}
\end{equation}
\end{theorem}

\begin{proof}
The element $W$ defined in Eq. \eqref{Eq:matrixW} 
takes $T_{0,1}$ to $T_{1,0}$, 
$X_{\alpha+\beta}$ to $-X_{\alpha+\beta}$,  
$X_{\alpha}$ to $-X_{\alpha+2\beta}$,  and
$X_{\alpha+2\beta}$ to $-X_{\alpha}$.  The result follows from Lemmas
\ref{lemma:T01:dim3} and \ref{lemma:T10:dim3}.    
\end{proof}

\paragraph{$T_{1,-1}$}	From the eigenvalues and eigenvectors of $\ad(T_{1,-1} )$ in $\algn$ 
as described in \S\ref{subsubT1-1dim2}, we find that 
the three-dimensional subalgebras of $\borel$ which contain $T_{1,-1}$ but not $\algt$ 
are:
\begin{equation}
\begin{array}{llllll}
\langle T_{1,-1}, X_{\alpha},  X_{\alpha + \beta} \rangle, \\
\langle T_{1,-1}, X_{\alpha},  X_{\alpha + 2\beta} \rangle, \\
\langle T_{1,-1}, X_{\alpha + \beta},  X_{\alpha + 2\beta} \rangle,  \\
\langle T_{1,-1}, X_{\beta},  X_{\alpha + 2\beta} \rangle. 
\end{array}
\end{equation}

The element $W \in Sp(4, \Cb)$ conjugates 
$\langle T_{1,-1} , X_{\alpha},  X_{\alpha + \beta} \rangle $
to
$\langle T_{1,-1},$ $X_{\alpha + \beta},  X_{\alpha + 2\beta} \rangle$,  and 
comparison of the eigenvalues of $\ad(T_{1,-1})$ shows that this is 
the only equivalence.  

\begin{lemma}\label{lemma:T1-1:dim3}
Any three-dimensional solvable subalgebra of $\borel$ which contains $T_{1,-1}$ but 
not $\algt$ is equivalent to one of the following inequivalent algebras:  
\begin{equation}
\begin{array}{llllll}
\langle T_{1,-1}, X_{\alpha  + \beta},  X_{\alpha + 2 \beta} \rangle, \\
\langle T_{1,-1}, X_{\alpha},  X_{\alpha + 2\beta} \rangle, \\
\langle T_{1,-1}, X_{\beta},  X_{\alpha + 2\beta} \rangle. 
\end{array}
\end{equation}
Moreover, 
$\langle T_{1,-1}, X_{\alpha + \beta},  X_{\alpha + 2\beta} \rangle 
\sim \langle T_{1,-1}, X_{\alpha},  X_{\alpha + \beta} \rangle$. 
\end{lemma}

\paragraph{$T_{1,1}$} 
We recall the eigenvalues and eigenvectors of $\ad(T_{1,1} )$ in $\algn$ 
as described in \S\ref{subsubT11dim2}.
Because the $2$-eigenspace of $\ad(T_{1,1} )$ in $\algn$ is $\algn_{\p}$,  
we find from Lemma~\ref{lemma:2dimSubsp.np} that any  
three-dimensional solvable subalgebra $\alga \subset \borel$ containing 
$T_{1,1}$ but not $\algt$ is conjugate to one of the following:  
\begin{equation}
\begin{array}{llllll}
\langle T_{1,1}, X_{\alpha +  \beta}, X_{\alpha+ 2 \beta} \rangle, \\
\langle T_{1,1}, X_{\alpha}, X_{\alpha+ 2 \beta}  \rangle, \\
\langle T_{1,1}, X_{\beta}, X_{\alpha+ 2 \beta} \rangle. 
\end{array}
\end{equation}

The matrix $A$ defined in Eq. \eqref{eq:matrixAinSp} conjugates  
$\langle T_{1,-1}, X_{\alpha + \beta},$  $X_{\alpha + 2 \beta} \rangle$
to $\langle T_{1,1}, X_{\beta}, X_{\alpha+ 2 \beta} \rangle$
and $\langle T_{1,-1}, X_{\beta},  X_{\alpha + 2\beta} \rangle$
to $\langle T_{1,1}, X_{\alpha +  \beta}, X_{\alpha+ 2 \beta} \rangle$.  

Since the eigenvalues of $\ad(T_{1,-1})$ on $X_{\alpha}$ and $X_{\alpha + 2\beta}$
are $-2$ and $2$,  respectively,  while those of 
$\ad(T_{1,1})$ on $X_{\alpha} $ and $X_{\alpha + 2 \beta}$ 
are both $2$,  we see that $\langle T_{1,1} , X_{\alpha} , X_{\alpha+ 2 \beta}  \rangle$ 
is not equivalent to $\langle T_{1,-1} , X_{\alpha},  X_{\alpha + 2\beta} \rangle$.   
Since the eigenvalue of $\ad(T_{1,1})$ on $X_{ \beta}$ is $0$,  we also see that 
$\langle T_{1,1}, X_{\alpha} , X_{\alpha+ 2 \beta}  \rangle$
is not equivalent to 
$\langle T_{1,1} , X_{\beta}, X_{\alpha+ 2 \beta} \rangle$.

Finally, the two-dimensional space spanned by $X_{\alpha + \beta}$ and $X_{\alpha + 2 \beta}$ 
consists of elements which are all of rank $2$, except for those on a single line, which are of rank $1$.  
However, the span of $X_{\alpha}$ and $X_{\alpha+ 2 \beta}$ contains two lines of elements of rank $1$. 
This shows that 
$\langle T_{1,1}, X_{\alpha + \beta},  X_{\alpha + 2\beta} \rangle$ is not equivalent 
to $\langle T_{1,1}, X_{\alpha}, X_{\alpha + 2 \beta} \rangle$.  

We summarize:

\begin{theorem}\label{lemma:singularss:dim3t} 
Any three-dimensional solvable subalgebra of $\borel$ containing a semisimple element with eigenvalues 
$1,1,-1,-1$ but not containing $\algt$ is equivalent to 
one of the following inequivalent algebras: 
\begin{equation}
\begin{array}{llllll}
\langle T_{1,-1}, X_{\alpha  + \beta},  X_{\alpha + 2 \beta} \rangle& \sim &  \langle T_{1,1}, X_{\beta},  X_{\alpha + 2\beta} \rangle, \\
\langle T_{1,-1}, X_{\alpha},  X_{\alpha + 2\beta} \rangle,& &\\
\langle T_{1,-1}, X_{\beta},  X_{\alpha + 2\beta} \rangle& \sim& \langle T_{1,1} , X_{\alpha  + \beta},  X_{\alpha + 2 \beta} \rangle, \\
\langle T_{1,1}, X_{\alpha}, X_{\alpha + 2 \beta} \rangle.&& 
\end{array}
\end{equation}
\end{theorem}

\subsection{Non-nilpotent three-dimensional solvable algebras containing no semisimple elements}\label{subsection:NossNotNilpDim3}
Any three-dimensional solvable algebra $\alga$ containing no semisimple elements but which does not 
consist entirely of nilpotent elements must contain an element conjugate to one of the first four 
matrices in Table~\ref{tab:nonss}.  Possibly after a conjugation,  we can assume that $\alga$ contains 
one of the following elements:
\begin{equation}
T_{1,1} + X_{\beta}, \quad
T_{1,-1} + X_{\alpha+ \beta}, \quad
T_{1,0} + X_{\alpha}, \quad
T_{0,1} + X_{\alpha + 2 \beta}.
\end{equation}

Note that each of these elements is of the form $T+X$,  where $0 \ne T\in \algt$ is non-regular and $X$ is the 
only positive root vector that commutes with $T$.  Note that if $N$ is any linear combination 
of the other positive root vectors,  then $T+N$ is semisimple.   
By Lemma \ref{lemmaFullss},  $\alga$ has a basis consisting of one of the above elements and two elements 
of $\algn$ which are linear combinations of the positive root vectors that do not commute with it.

\subsubsection{$T_{1,1} + X_{\beta}$}
If $\alga$ contains $T_{1,1} + X_{\beta}$,  then it cannot contain an element of the form
$X_{\alpha} + b X_{\alpha+\beta} + cX_{\alpha+2\beta}$,  because any such algebra would
contain all of $\algn_{\p}$ and thus have dimension $4$.  The only possible algebra is 
$\langle T_{1,1} + X_{\beta}, X_{\alpha+\beta}, X_{\alpha+2\beta} \rangle$.

\subsubsection{$T_{1,-1} + X_{\alpha + \beta}$}
If $\alga$ contains $T_{1,-1} + X_{\alpha + \beta}$,  then it cannot contain an element of the form
$aX_{\beta} + bX_{\alpha}  + cX_{\alpha+2\beta}$,  with $a, b  \ne 0$,  because any such algebra would
contain all of $\algn$ and hence $T_{1,-1}$.  The only possibilities are  
$\langle T_{1,-1} + X_{\alpha+\beta} , X_{\alpha}, X_{\alpha+2\beta} \rangle$
and  $\langle T_{1,-1} + X_{\alpha+\beta} , X_{\beta}, X_{\alpha+2\beta} \rangle$.

\subsubsection{$T_{1,0} + X_{\alpha}$}
If $\alga$ contains $T_{1,0} + X_{\alpha}$,  then it cannot contain an element of the form
$X_{\beta} + b X_{\alpha + \beta} + cX_{\alpha+2\beta}$,  because any such algebra would
contain all of $\algn_{\p}$ and hence $T_{1,0}$.  The only possibility is 
$\langle T_{1,0} + X_{\alpha} , X_{\alpha + \beta}, X_{\alpha+2\beta} \rangle$.

\subsubsection{$T_{0,1} + X_{\alpha + 2 \beta}$} Note that $X_{\alpha + 2 \beta}$ commutes with 
all the other positive root vectors,  and they have distinct eigenvalues for $\ad(T_{0,1})$ and hence 
for $\ad(T_{0,1} + X_{\alpha + 2 \beta})$.   
So if $\alga$ contains $T_{0,1} + X_{\alpha + 2 \beta}$,  then it has a basis consisting of 
$T_{0,1} + X_{\alpha + 2 \beta}$ and two positive root vectors 
other than $X_{\alpha + 2 \beta}$.

It cannot contain $X_{\alpha}$  and $X_{\beta}$ because they generate $\algn$.  
It cannot contain $X_{\beta}$ and $X_{\alpha + \beta}$ because then it would contain 
$X_{\alpha+2\beta}$ and hence $T_{0,1}$.  The only possibility is 
$\langle T_{0,1} + X_{\alpha + 2 \beta}, X_{\alpha}, X_{\alpha+\beta} \rangle$.

\subsubsection{Equivalences}
The algebras 
$\langle T_{1,1} + X_{\beta}, X_{\alpha+\beta} , X_{\alpha+2\beta} \rangle$ and 
$\langle T_{1,-1}$ $+ X_{\alpha+\beta}, X_{\alpha}, X_{\alpha+2\beta} \rangle$ 
cannot be equivalent, since in the former,  the restriction of 
$\ad(T_{1,1} + X_{\beta})$ to the nilpotent subspace has a 
two-dimensional generalized eigenspace with eigenvalue $2$,  whereas in the 
latter, the restriction of $\ad(T_{1,-1} + X_{\alpha+\beta} )$ to the nilpotent subspace
has distinct eigenvalues $2, -2$.

However,  the element $W$ of Eq. \eqref{Eq:matrixW}
takes 
$\langle T_{1,0} + X_{\alpha} , X_{\alpha + \beta},$ $X_{\alpha+2\beta} \rangle$ to 
$\langle T_{0,1} + X_{\alpha + 2 \beta}, X_{\alpha}, X_{\alpha+\beta} \rangle$,  showing  
that they are equivalent. 
The matrix $A$ defined in Eq. \eqref{eq:matrixAinSp} 
conjugates 
$\langle T_{1,1} + X_{\beta}, X_{\alpha+\beta}, X_{\alpha+2\beta} \rangle$ to 
$\langle T_{1,-1} + X_{\alpha+\beta}, X_{\beta}, X_{\alpha+2\beta} \rangle$.

Finally, we note that $\langle T_{1,0} + X_{\alpha} , X_{\alpha + \beta}, X_{\alpha+2\beta} \rangle$ 
consists entirely of singular elements, so it cannot be conjugate to either
$\langle T_{1,1} + X_{\beta}, X_{\alpha+\beta},$ $X_{\alpha+2\beta} \rangle$ or
$\langle T_{1,-1}$ $+ X_{\alpha+\beta}, X_{\alpha}, X_{\alpha+2\beta} \rangle$, 
both of which contain invertible matrices.  

\begin{theorem}\label{prop:nonSSnonNilpDim3}
Any three-dimensional solvable algebra $\alga$ containing no semisimple elements but which does not 
consist entirely of nilpotent elements must be equivalent to one of the following inequivalent 
subalgebras:
\begin{equation}
\begin{array}{llllll}
\langle T_{1,1} + X_{\beta}, X_{\alpha+\beta} , X_{\alpha+2\beta} \rangle  & \sim & 
\langle T_{1,-1} + X_{\alpha+\beta} , X_{\beta}, X_{\alpha+2\beta} \rangle, \\ 
\langle T_{1,-1} + X_{\alpha+\beta} , X_{\alpha}, X_{\alpha+2\beta} \rangle, && \\
\langle T_{1,0} + X_{\alpha} , X_{\alpha + \beta}, X_{\alpha+2\beta} \rangle  & \sim &
\langle T_{0,1} + X_{\alpha + 2 \beta}, X_{\alpha}, X_{\alpha+\beta} \rangle. 
\end{array}
\end{equation}
\end{theorem}

\subsection{Three-dimensional nilpotent subalgebras}
Since 
 $[X_{\beta}, X_{\alpha}] = X_{\alpha+ \beta}$ and 
 $[X_{\beta}, X_{\alpha+ \beta}] = 2X_{\alpha+ 2\beta}$, 
 $X_{\beta}$ and $X_{\alpha}$ generate $\algn$ and 
$X_{\beta}$ and $X_{\alpha + \beta}$ generate 
$\langle X_{\beta},  X_{\alpha + \beta}, X_{\alpha + 2 \beta} \rangle$. 
In fact,  any algebra containing an element of the form 
$X_{\beta} + N$,  with $N \in \algn_{\p}$, and an element of the form 
$X_{\alpha} + bX_{\alpha + \beta} + cX_{\alpha + 2 \beta}$ must contain $\algn$.  
For any $r, s \ne 0$,  the algebra 
$\langle rX_{\beta} + sX_{\alpha},  X_{\alpha + \beta}, X_{\alpha + 2 \beta} \rangle$
is conjugate by an element of the diagonal subgroup $T \subset Sp(4, \Cb)$ to 
$\langle X_{\beta} + X_{\alpha},  X_{\alpha + \beta}, X_{\alpha + 2 \beta} \rangle$.
We summarize.

\begin{theorem}\label{prop:NilpDim3}
Any three-dimensional solvable subalgebra $\alga$ 
consisting entirely of nilpotent elements must be equivalent to one of the following inequivalent 
subalgebras:
\begin{equation}
\algn_{\p}, \quad
\langle X_{\beta},  X_{\alpha + \beta}, X_{\alpha + 2 \beta} \rangle, \quad
\langle X_{\beta} + X_{\alpha},  X_{\alpha + \beta}, X_{\alpha + 2 \beta} \rangle. 
\end{equation}
\end{theorem}

\section{Four-dimensional solvable subalgebras of $\soo$}

In this section, we classify the four-dimensional  solvable subalgebras of $\soo$ into five cases:  Four-dimensional subalgebras containing a Cartan subalgebra (see Theorem \ref{containingCartandim4}); solvable subalgebras containing a regular 
semisimple element but not all of $\algt$ (see Theorem \ref{containingRegSSdim4});  solvable subalgebras containing a non-regular 
semisimple element but not all of $\algt$ (see Theorem \ref{prop:nonregssdim4}); non-nilpotent solvable subalgebras containing no semisimple elements (see Theorem \ref{fourfouree}); and nilpotent subalgebras (see Theorem \ref{fourfourff}).
Again, without loss of generality, we assume that each solvable subalgebra is in the Borel subalgebra $\borel$. The results are summarized in Table \ref{fourdimuu}.

\subsection{Four-dimensional Subalgebras Containing a Cartan Subalgebra}
The four-dimensional subalgebras of $\borel$ that contain $\algt$ are 
\begin{equation}
\langle \algt, X_{\alpha}, X_{\alpha + \beta} \rangle, \\ 
\langle \algt, X_{\alpha}, X_{\alpha + 2\beta} \rangle, \\ 
\langle \algt, X_{\alpha + \beta}, X_{\alpha + 2 \beta} \rangle, \\ 
\langle \algt, X_{ \beta}, X_{\alpha + 2 \beta} \rangle.
\end{equation}
The element $W$ of Eq. \eqref{Eq:matrixW} conjugates 
$\langle \algt, X_{\alpha}, X_{\alpha + \beta} \rangle$ to  
$\langle \algt, X_{\alpha + \beta},$ $X_{\alpha + 2 \beta} \rangle$,  and the 
element $A$ of Eq. \eqref{eq:matrixAinSp} conjugates 
$\langle \algt, X_{ \beta}, X_{\alpha + 2 \beta} \rangle$ to  
$\langle \algt, X_{\alpha + \beta}, X_{\alpha + 2 \beta} \rangle$.

Consider the eight elements $T_{\pm 2, \pm 1}, T_{\pm 1, \pm2} \in \algt$.  They are all the elements 
of $\algt$ that have distinct eigenvalues $\pm 1, \pm 2$.  Acting on $\algn$ via the 
adjoint representation,  they all have the property that they act with even eigenvalues on 
the root vectors corresponding to long roots and 
with odd eigenvalues on the root vectors corresponding to short roots.  This shows 
that  $\langle \algt, X_{\alpha}, X_{\alpha + \beta} \rangle$ 
is not equivalent to $\langle \algt, X_{\alpha}, X_{\alpha + 2\beta} \rangle$. 

We summarize.

\begin{theorem}\label{containingCartandim4}
Any four-dimensional subalgebra of $\borel$ that contains $\algt$ is 
equivalent to one of the following inequivalent subalgebras.   
\begin{equation}
\langle \algt, X_{\alpha}, X_{\alpha + \beta} \rangle, \quad
\langle \algt, X_{\alpha}, X_{\alpha + 2\beta} \rangle.  
\end{equation}
\end{theorem}

\subsection{Four-dimensional solvable subalgebras containing a regular 
semisimple element but not a Cartan subalgebra}
Suppose $a \ne 0, \pm 1$.  If $a \ne 3$,  then it is easy to see that 
any four-dimensional subalgebra $\alga \subset \borel$ which contains $T_{a,1}$ but not 
$\algt$ must be one of the following:
\begin{equation}
\langle T_{a,1}, \algn_{\p} \rangle, \quad
\langle T_{a,1}, X_{\beta}, X_{\alpha + \beta}, X_{\alpha + 2 \beta} \rangle.
\end{equation}
In the first case,  $[ \alga, \alga] = \algn_{\p}$,  which is abelian, 
while in the second, 
$[ \alga, \alga] = \langle  X_{\beta}, X_{\alpha + \beta}, X_{\alpha + 2 \beta}  \rangle$,  which is not.  In particular,  the two algebras are 
not equivalent.  

We also observe that the matrix $W$ of Eq. \eqref{Eq:matrixW} conjugates 
$\langle T_{a,1}, \algn_{\p} \rangle$ to 
$\langle T_{a^{-1},1}, \algn_{\p} \rangle$.  
As in \S\ref{subsubsection:regss3d},  we argue that any other equivalence between pairs of 
such algebras must take $T_{a,1}$ to one of $\pm T_{a,1}$,  $\pm T_{1,a}$,  $\pm T_{-a,1}$, or $ \pm T_{1,-a}$. 
The only possibilities with the right eigenvalues on $\algn_{\p}$ are $T_{a,1}$ and $T_{1,a}$,  so 
the above equivalence is the only one.  

Similarly,  the matrix $A$ of Eq. \eqref{eq:matrixAinSp} conjugates 
$\langle T_{a,1}, X_{\beta}, X_{\alpha + \beta},$ $X_{\alpha + 2 \beta} \rangle$ 
to
$\langle T_{-a,1}, X_{\beta}, X_{\alpha + \beta},$ $X_{\alpha + 2 \beta} \rangle$. 
In any equivalence between two algebras of this form,  
the only rank $1$ eigenvector $X_{\alpha + 2 \beta} $ must go to a multiple of itself,   so 
its eigenvalue must be preserved,  and 
the only possibilities are that $T_{a,1}$ 
must go to itself or $T_{a,-1}$. Again the above equivalence is the only one.  
\medskip

If $a = 3$,  then two analogous algebras occur,  with the same equivalences, but for any $r,s \ne 0$,  
there is also the algebra 
$\langle T_{3,1}, rX_{\alpha}+ s  X_{\beta}, X_{\alpha + \beta}, X_{\alpha + 2 \beta} \rangle$.  
By Lemma \ref{Lemma:T31},  for any nonzero $r,s$, this algebra is equivalent to 
$\alga = \langle T_{3,1}, X_{\alpha} + X_{\beta}, X_{\alpha + \beta}, X_{\alpha + 2 \beta} \rangle$.  

This algebra $\alga$ is not equivalent to 
$\langle T_{a,1}, \algn_{\p} \rangle$,  for any $a \ne 0, \pm 1$, because 
$[ \alga , \alga] = \langle  X_{\alpha}+ X_{\beta}, X_{\alpha + \beta}, X_{\alpha + 2 \beta}  \rangle$,  which is not abelian.  
Moreover, it is not equivalent to 
$\langle T_{a,1}, X_{\beta}, X_{\alpha + \beta}, X_{\alpha + 2 \beta} \rangle$,  for any $a \ne 0, \pm 1$, because the 
corresponding commutator subalgebras are 
$\langle X_{\alpha} + X_{\beta}, X_{\alpha + \beta},$ $X_{\alpha + 2 \beta} \rangle$ 
and 
$\langle X_{\beta}, X_{\alpha + \beta}, X_{\alpha + 2 \beta} \rangle$,  
whose elements are generically of rank $3$ and rank $2$,  respectively.  

We summarize.  

\begin{theorem}\label{containingRegSSdim4}
Any four-dimensional subalgebra of $\borel$ that contains a regular semisimple element but not a Cartan subalgebra is 
equivalent to one of the following subalgebras,  which are inequivalent apart from the equivalences noted.   
\begin{footnotesize}
\begin{equation}
\begin{array}{lllll}
\langle T_{a,1}, \algn_{\p} \rangle &\sim& \langle T_{a^{-1},1}, \algn_{\p} \rangle,\\
\langle T_{a,1}, X_{\beta}, X_{\alpha + \beta}, X_{\alpha + 2 \beta} \rangle &\sim& 
\langle T_{-a,1}, X_{\beta}, X_{\alpha + \beta}, X_{\alpha + 2 \beta} \rangle, \\
\langle T_{3,1}, X_{\alpha} + X_{\beta} , X_{\alpha + \beta}, X_{\alpha + 2 \beta} \rangle .
&
\end{array}
\end{equation}\end{footnotesize}
where $a \ne 0, \pm 1$.
\end{theorem}


\subsection{Four-dimensional solvable subalgebras containing a non-regular 
semisimple element but not a Cartan subalgebra}

As was observed in \S\ref{subsection.nonreg.ss},  the non-regular elements in $\algt$ are multiples of 
$T_{1,0}$, $T_{0,1}$, $T_{1,1}$,  and $T_{1,-1}$.

\subsubsection{$T_{0,1}$}  Since the positive root vectors have different eigenvalues for 
$\ad(T_{0,1})$,  the only possibilities are algebras spanned by three of them together with 
$T_{0,1}$.  The only possibilities are 
$\langle T_{0,1}, \algn_{\p} \rangle$ and
$\langle T_{0,1}, X_{\beta}, X_{\alpha + \beta},$ $X_{\alpha + 2 \beta} \rangle$.  
Since the eigenvalues of $\ad(T_{0,1})$ on $\algn_{\p}$ are $0,1,2$ while those 
on $\langle X_{\beta}, X_{\alpha + \beta}, X_{\alpha + 2 \beta} \rangle$
are $-1, 1,0$,  we see they are inequivalent.  

\subsubsection{$T_{1,0}$}  The eigenvalues of the positive root vectors for 
$\ad(T_{1,0})$ are $1$ for $X_{\beta}$ and for $X_{\alpha + \beta}$, $0$ for $X_{\alpha}$,  
and $2$ for $X_{\alpha + 2\beta}$. 

It is not difficult to check that the only possibilities are 
$\langle T_{1,0}, \algn_{\p} \rangle$ and
$\langle T_{1,0}, X_{\beta}, X_{\alpha + \beta}, X_{\alpha + 2 \beta} \rangle$.  
Since the eigenvalues of $\ad(T_{1,0})$ on $\algn_{\p}$ are $0,1,2$ while those 
on $\langle X_{\beta}, X_{\alpha + \beta}, X_{\alpha + 2 \beta} \rangle$
are $1, 1, 2$,  we see they are inequivalent.  

The matrix $W$ of Eq. \eqref{Eq:matrixW}  conjugates 
$\langle T_{0,1}, \algn_{\p} \rangle$ to 
$\langle T_{1,0}, \algn_{\p} \rangle$,  but 
  $\langle T_{0,1}, X_{\beta}, X_{\alpha + \beta}, X_{\alpha + 2 \beta} \rangle$  
and $\langle T_{1,0}, X_{\beta}, X_{\alpha + \beta}, X_{\alpha + 2 \beta} \rangle$ 
are inequivalent because in the former,  $\ad(T_{0,1})$ is singular on the 
nilpotent subalgebra while in the latter,  $\ad(T_{1,0})$ is nonsingular on the 
nilpotent subalgebra.  

We summarize.

\begin{lemma}\label{lemma:T10T01dim4}
A four-dimensional solvable algebra containing $T_{0,1}$ or $T_{1,0}$ but not a Cartan subalgebra 
is equivalent to one of the following: 
\begin{equation}
\begin{array}{lllll}
\langle T_{0,1}, \algn_{\p} \rangle \sim \langle T_{1,0}, \algn_{\p} \rangle,   \\
\langle T_{0,1}, X_{\beta}, X_{\alpha + \beta}, X_{\alpha + 2 \beta} \rangle, \\
\langle T_{1,0}, X_{\beta}, X_{\alpha + \beta}, X_{\alpha + 2 \beta} \rangle.    
\end{array}
\end{equation}
These algebras are all pairwise inequivalent.
\end{lemma}

\subsubsection{$T_{1,1}$}  The eigenvalues of the positive root vectors for 
$\ad(T_{1,1})$ are $0$ for $X_{\beta}$ and $2$ for $X_{\alpha}$,  $X_{\alpha + \beta}$,  
and $X_{\alpha + 2\beta}$. 

It is not difficult to check that the only possibilities are 
$\langle T_{1,1}, \algn_{\p} \rangle$ and
$\langle T_{1,1}, X_{\beta}, X_{\alpha + \beta}, X_{\alpha + 2 \beta} \rangle$.  
Since the eigenvalues of $\ad(T_{1,1})$ on $\algn_{\p}$ are $2,2,2$ while those 
on $\langle X_{\beta}, X_{\alpha + \beta}, X_{\alpha + 2 \beta} \rangle$
are $0, 2, 2$,  we see they are inequivalent.  

\subsubsection{$T_{1,-1}$}  The eigenvalues of the positive root vectors for 
$\ad(T_{1,-1})$ are $2$ for $X_{\beta}$ and $X_{\alpha + 2\beta}$,  $-2$ for $X_{\alpha}$,  and $0$ for $X_{\alpha + \beta}$. 
It is not difficult to check that the only possibilities are 
$\langle T_{1,-1}, \algn_{\p} \rangle$ and
$\langle T_{1,-1}, X_{\beta}, X_{\alpha + \beta},$ $X_{\alpha + 2 \beta} \rangle$.  
Since the eigenvalues of $\ad(T_{1,-1})$ on $\algn_{\p}$ are $-2,0,2$ while those 
on $\langle X_{\beta}, X_{\alpha + \beta}, X_{\alpha + 2 \beta} \rangle$
are $2, 0, 2$,  we see they are inequivalent.  

The matrix $A$ of Eq. \eqref{eq:matrixAinSp} conjugates 
$\langle T_{1,1}, X_{\beta}, X_{\alpha + \beta}, X_{\alpha + 2 \beta} \rangle$
to $\langle T_{1,-1}, X_{\beta}, X_{\alpha + \beta}, X_{\alpha + 2 \beta} \rangle$. 
Consideration of the eigenvalues of $\ad(T_{1,-1})$ on $\algn_{\p}$ shows 
that $\langle T_{1,1}, \algn_{\p} \rangle$ and $\langle T_{1,-1}, \algn_{\p} \rangle$ 
are not equivalent.  

We summarize.

\begin{theorem}\label{prop:nonregssdim4}
A four-dimensional solvable algebra containing a non-regular semisimple element but not a Cartan subalgebra 
is equivalent to one of the following: 
\begin{equation}
\begin{array}{lllll}
\langle T_{0,1}, \algn_{\p} \rangle \sim \langle T_{1,0}, \algn_{\p} \rangle,  \\
\langle T_{0,1}, X_{\beta}, X_{\alpha + \beta}, X_{\alpha + 2 \beta} \rangle, \\ 
\langle T_{1,0}, X_{\beta}, X_{\alpha + \beta}, X_{\alpha + 2 \beta} \rangle,     \\
\langle T_{1,1}, \algn_{\p} \rangle, \\ \langle T_{1,-1}, \algn_{\p} \rangle,    \\
\langle T_{1,1}, X_{\beta}, X_{\alpha + \beta}, X_{\alpha + 2 \beta} \rangle  
\sim \langle T_{1,-1}, X_{\beta}, X_{\alpha + \beta}, X_{\alpha + 2 \beta} \rangle,   
\end{array}
\end{equation}
These algebras are all pairwise inequivalent.
\end{theorem}

\subsection{Non-nilpotent four-dimensional solvable algebras containing no semisimple elements}\label{subsection:NossNotNilpDim4}
As in \S\ref{subsection:NossNotNilpDim3},  we see that, possibly after a conjugation,
any four-dimensional solvable algebra $\alga$ containing no semisimple elements but which does not 
consist entirely of nilpotent elements must contain one of the following elements:
\begin{equation}
T_{1,1} + X_{\beta}, \quad
T_{1,-1} + X_{\alpha+ \beta}, \quad
T_{1,0} + X_{\alpha}, \quad
T_{0,1} + X_{\alpha + 2 \beta}.
\end{equation}
By Lemma \ref{lemmaFullss} and the argument given at the beginning of \S\S\ref{subsection:NossNotNilpDim3},  $\alga$ has a basis consisting of one of the above elements and three elements 
of $\algn$ which are linear combinations of the positive root vectors that do not commute with it.  
However,  since $\langle T_{1,-1} + X_{\alpha+ \beta}, X_{\alpha},  X_{\beta}, X_{\alpha+ 2\beta}   \rangle$ and 
$\langle  T_{0,1} + X_{\alpha + 2 \beta},  X_{\alpha},  X_{\beta}, X_{\alpha+ \beta}   \rangle$ both contain 
$X_{\alpha}$ and $X_{\beta}$,  they both contain all of $\algn$,  and hence the former contains the semisimple 
element $T_{1,-1}$ and the latter contains $T_{0,1}$.  

The remaining two algebras 
$\langle T_{1,1} + X_{\beta}, X_{\alpha}, X_{\alpha+ \beta}, X_{\alpha+ 2\beta} \rangle$ and 
$\langle  T_{1,0} + X_{\alpha}, X_{\beta},  X_{\alpha+ \beta}, X_{\alpha+ 2\beta} \rangle$   
are clearly inequivalent,  because the former contains invertible elements while the latter 
does not.  

\begin{theorem}\label{fourfouree}
Any four-dimensional solvable algebra $\alga$ containing no semisimple elements but which does not 
consist entirely of nilpotent elements must be conjugate to one of the following inequivalent algebras:
\begin{equation}
\begin{array}{lllll}
\langle T_{1,1} + X_{\beta}, X_{\alpha}, X_{\alpha+ \beta}, X_{\alpha+ 2\beta} \rangle, \\
\langle  T_{1,0} + X_{\alpha}, X_{\beta},  X_{\alpha+ \beta}, X_{\alpha+ 2\beta} \rangle. 
\end{array}
\end{equation}
\end{theorem}

\subsection{Four-dimensional nilpotent subalgebras}

\begin{theorem}\label{fourfourff}
The only four-dimensional nilpotent subalgebra of $\borel$ is $\algn$.  
\end{theorem} 

\section{Five- and six-dimensional solvable subalgebras of $\soo$}

\subsection{Five-dimensional solvable subalgebras} 
A five-dimensional solvable subalgebra $\alga \subset \borel$ must contain some 
semisimple elements.  By Lemma \ref{lemmaCartan},  we can assume it contains  
$\algt$ if it contains a Cartan subalgebra and $0 \ne T \in \algt$ otherwise.  

If  $\alga \subset \borel$ contains $\algt$,  then it must be spanned by $\algt$ and 
three of the root vectors $X_{\alpha}$, $X_{\beta}$,  $X_{\alpha+ \beta}$,  and $X_{\alpha+ 2\beta}$.   
As above,  the possibilities are $\alga = \algt + \algn_{\p}$ and   
$\alga = \langle \algt, X_{\beta}, X_{\alpha+ \beta}, X_{\alpha+ 2\beta} \rangle$.   

If  $\alga \subset \borel$ does not contain $\algt$,  then it must contain $0 \ne T \in \algt$. 
We can find a basis for $\alga$ consisting of $T$ and four elements of $\algn $.  But 
these elements span $\algn$,  and $\alga = \langle T,  \algn \rangle$.   

Now it is easy to check that the normalizer of $\algn$ in $Sp(4, \Cb)$ is $B$.  
So if Ad$(g)$ takes $\langle T,  \algn \rangle$ to $\langle T',  \algn \rangle$,  for some $T, T' \in \algt$,  
then $g$ must preserve the nilpotent radical $\algn$ and therefore must be in $B$.  Using 
Lemma \ref{lemmaCartan},  this means that $T$ and $T'$ must be 
conjugate under $B$.  Of course,  for any nonzero $r \in \Cb$,  we have that 
$\langle T,  \algn \rangle = \langle rT,  \algn \rangle$.
So the subalgebra $\langle T,  \algn \rangle$ is determined by the nonzero scalar multiples of the 
$B$-conjugacy class of $T$.  

Using the list in Table \ref{table:ssclasses},  we find the following list of representatives: 
$T_{1,a}$, $a \ne 0, \pm 1$; 
$T_{1,0}$, 
$T_{0,1}$, 
$T_{1,1} $, and 
$T_{1,-1}$.

We summarize.  

\begin{theorem}\label{prop:dim5}  
Up to equivalence,  every five-dimensional solvable subalgebra is one of the following.  
In \S\S\ref{5dimequivalences},  we shall show that they are in fact pairwise nonisomorphic. 
\begin{equation}
\begin{array}{llllllllll}
&\algt + \algn_{\p},   &  \langle \algt , X_{\beta}, X_{\alpha+ \beta}, X_{\alpha+ 2\beta} \rangle, \\
&\langle T,  \algn \rangle, & {\rm where} \,\,T \,\, {\rm is \,\,one \,\, of\,\, the \,\, following: }\\
&&T_{1,a}, \quad  a \ne 0, \pm 1; \\
&&
T_{1,0}, \quad
T_{0,1}, \quad 
T_{1,1}, \quad
T_{1,-1}  
\\
\end{array}
\end{equation}
\end{theorem}

\subsection{Six-dimensional solvable subalgebras} 
The only six-dimensional subalgebra of $\borel$ is $\borel$ itself.  

\section{Isomorphisms}\label{isomorphs}

\subsection{Dimension $2$}

We first identify the two-dimensional subalgebras of $\soo$ with respect de Graaf's classification \cite{degraafa}. For each algebra of de Graaf's classification that appears, we then identify which algebra it is isomorphic to  in the classification described by \v{S}nobl and Winternitz in \cite{levi}.

The classification of two-dimensional algebras amounts to whether they are abelian or not.  
Of course, $\algt$ is abelian,  so $\algt \cong K^{1}$.\\

Since $[T_{3,1}, X_{\alpha} + X_{\beta}] = 2( X_{\alpha} + X_{\beta} )$,  
$\langle T_{3,1}, X_{\alpha} + X_{\beta} \rangle \cong K^{2}$.

Since $[T_{a,1} , X_{\alpha}] = 2 X_{\alpha}$,
$\langle T_{a,1} , X_{\alpha} \rangle \cong K^{2}$.  

Since $[T_{a,1} , X_{\beta} ] = (a-1)X_{\beta}$, 
$\langle T_{a,1} , X_{\beta} \rangle \cong K^{2}$, provided $a \ne 1$.

Since $[T_{1,0}, X_{\alpha}] = 0$, 
$\langle  T_{1,0}, X_{\alpha} \rangle \cong K^{1}$. 

Since $[ T_{1,0}, X_{\beta}] = X_{\beta}$,  
$ \langle  T_{1,0}, X_{\beta} \rangle \cong K^{2}$.
 
Since $[ T_{1,0}, X_{\alpha  +2 \beta} ] = 2 X_{\alpha  +2 \beta}$, 
$\langle  T_{1,0}, X_{\alpha  +2 \beta} \rangle \cong K^{2}$.

Since $[  T_{1,1}, X_{\alpha}  ] = 2  X_{\alpha} $,
$\langle  T_{1,1}, X_{\alpha} \rangle \cong K^{2}$.

Since $[ T_{1,1}, X_{\beta}] = 0$, 
$\langle  T_{1,1}, X_{\beta} \rangle \cong K^{1}$.

Since $[ T_{1,1}, X_{\alpha+ \beta} ] = 2 X_{\alpha+ \beta}$,
$\langle  T_{1,1}, X_{\alpha+ \beta} \rangle \cong K^{2}$. 

Since $[  T_{1,1}+X_{\beta}, X_{\alpha + 2 \beta} ] = 2  X_{\alpha + 2 \beta}$,
$\langle T_{1,1}+X_{\beta}, X_{\alpha + 2 \beta} \rangle \cong K^{2}$.

Since $[ T_{1,0}+X_{\alpha }, X_{\alpha + \beta } ] =  X_{\alpha + \beta }$,
$\langle T_{1,0}+X_{\alpha }, X_{\alpha + \beta } \rangle \cong K^{2}$.

Since $[ T_{1,0}+X_{\alpha }, X_{\alpha + 2\beta } ] = 2 X_{\alpha + 2\beta }$,
$\langle T_{1,0}+X_{\alpha }, X_{\alpha + 2\beta } \rangle \cong K^{2}$. 

Since $[  X_{\alpha}, X_{\alpha + \beta} ] = 0$, 
$\langle X_{\alpha}, X_{\alpha + \beta} \rangle \cong K^{1}$.

Since $[ X_{\alpha},  X_{\alpha + 2 \beta} ] = 0$, 
$\langle X_{\alpha},  X_{\alpha + 2 \beta}  \rangle \cong K^{1}$.

Since $[ X_{\beta} + X_{\alpha}, X_{\alpha+ 2 \beta} ] = 0$,  
$\langle X_{\beta} + X_{\alpha}, X_{\alpha+ 2 \beta} \rangle \cong K^{1}$.

\vspace{2mm}

The (only) nonabelian two-dimensional solvable algebra is $K^{2}$ in de Graaf's classification,
given by $[x_{1}, x_{2}] = x_{2}$  
and $\mathfrak{s}_{2,1}$ in \v{S}nobl and Winternitz's classification, which is given by 
$[ e_{2}, e_{1}] = e_{1}$.  Clearly they are isomorphic by 
\begin{equation}
\begin{array}{llll}
x_{1} & \longleftrightarrow & e_{2} \\
x_{2}& \longleftrightarrow & e_{1}. 
\end{array}
\end{equation}

\subsection{Dimension $3$}   We first identify the three-dimensional subalgebras of $\soo$ with respect de Graaf's classification \cite{degraafa}. For each algebra of de Graaf's classification that appears, we then identify which algebra it is isomorphic to  in the classification described by \v{S}nobl and Winternitz in \cite{levi}.

To see that $\langle \algt, X_{\alpha} \rangle \sim L^{3}_{0}$,  let 
\begin{equation} 
x_{1} =T_{1,0} + X_{\alpha}, \quad
x_{2} = X_{\alpha} , \quad 
x_{3} = \frac{1}{2}T_{0,1}, 
\end{equation}
and observe that these assignments respect the structure of $L^{3}_{0}$,  whose
only nonzero brackets are 
$[x_{3}, x_{1}] = x_{2}$;  $[x_{3}, x_{2}] = x_{2}$.   
 
To see that $\langle \algt, X_{\beta} \rangle \sim L^{3}_{0}$,  let 
\begin{equation} 
x_{1} =T_{1,1} + X_{\beta}, \quad
x_{2} = X_{\beta},  \quad
x_{3} = T_{1,0}, 
\end{equation}
and observe that these elements also respect the structure of $L^{3}_{0}$. 
\smallskip

\begin{lemma}\label{lemma:alg3}
Consider the algebra $\alga$ with basis $\{ T, A, B \}$ and relations 
$[A,B] =0$, $[T,A] = 2A$ and $[T,B] = rB$, for some $r \in \Cb$.  
If $r \ne -2$,  let $b = - \frac{2r}{(r+2)^{2}}$.  

Then 
 \begin{numcases}  {\alga \cong}
\label{lem.alg3.L3b}  L^{3}_{b},  &{\it if} \,\, $r \ne  \pm 2$, \\ 
\label{lem.alg3.L2}  L^{2},  & {\rm if} \,\, $r = 2$, \\   
\label{lem.alg3.L41}  L^{4}_{1}, & {\rm if} \,\, $r = -2$. 
\end{numcases}
\end{lemma}

\begin{proof}
Fix $k \in \Cb$, $k \ne 0$.

If we let $x_{1} = A+B$,  $x_{2} = 2kA + rkB$, and $x_{3} = kT$, 
then it is easy to check that 
$[x_{3} , x_{1} ] = x_{2}$,
$[x_{3}, x_{2}] = 4k^{2} A + r^{2} k^{2} B$, 
$[x_{1}, x_{2}] = 0$.

If we want these elements to satisfy the additional condition 
$[x_{3}, x_{2}] = bx_{1} + x_{2}$, for some $0 \ne b \in \Cb$,  this forces
\begin{equation}
4k^{2} = 2k+b,  \quad r^{2} k^{2} = rk + b. 
\end{equation}
From these equations, it is easy to find that, provided $r \ne \pm 2$,   
\begin{equation}\label{DefnOfbk}
k = \frac{1}{r+2} , \quad b = - \frac{2r}{(r+2)^{2}}.
\end{equation}
So given $r \ne  \pm 2$,  and $k$ and $b$ defined as in Eq. \eqref{DefnOfbk},  the algebra $\alga$ is given by the relations
\begin{equation}
[x_{3} , x_{1} ] = x_{2}, \quad 
[x_{3}, x_{2}] =  bx_{1} + x_{2}, \quad 
[x_{1}, x_{2}] = 0.
\end{equation}
These are the relations for the algebra $L^{3}_{b}$.  

If $r = 2$,  let 
$x_{1} = A$,
$x_{2} = B$,
$x_{3} = \frac{1}{2} T$.
Then 
\begin{equation}
[x_{3}, x_{1}] = x_{1}, \quad 
[x_{3}, x_{2}] = x_{2},
\end{equation}
and these are the defining relations for the algebra $L^{2}$.

If $r = -2$,  let 
$x_{1} = A+B$,
$x_{2} = A-B$,
$x_{3} = \frac{1}{2} T$.
Then 
\begin{equation}
[x_{3}, x_{1}] = x_{2}, \quad 
[x_{3}, x_{2}] = x_{1},
\end{equation}
and these are the defining relations for the algebra $L^{4}_{1}$. 
\end{proof}

We apply the remarks about the above algebra $\alga$ to various examples.  
\vspace{0.02mm}

Fix $a \ne 0, \pm 1$,  and consider 
the algebra 
$\langle T_{a,1},  X_{\alpha}, X_{\alpha + \beta} \rangle$. 
Letting 
$T = T_{a,1}$,  $A = X_{\alpha}$,  $B = X_{\alpha + \beta}$,  we find that 
$\langle T_{a,1},  X_{\alpha}, X_{\alpha + \beta} \rangle$ is an example of 
the algebra $\alga$ of Lemma \ref{lemma:alg3},  with $r = a+1$.   
Letting 
$b =   - \frac{2r}{(r+2)^{2}} =  -2 \frac{a+1}{(a+3)^{2}}$,  the lemma tells us (c.f.,  Eq. \eqref{lem.alg3.L3b})  that 
$\langle T_{a,1},  X_{\alpha}, X_{\alpha + \beta} \rangle \sim L^{3}_{b}$,  
provided $r \ne \pm 2$,  i.e.,  $a \ne -3$  (note that $a=1$ is excluded by hypothesis).  It also tells us (c.f., Eq. \eqref{lem.alg3.L41})  that 
$\langle T_{a,1},  X_{\alpha}, X_{\alpha + \beta} \rangle \sim L^{4}_{1}$, if $a = -3$. 
\smallskip

Now fix $a \ne 0, \pm 1$,  and consider 
the algebra 
$\langle T_{a,1},  X_{\alpha}, X_{\alpha + 2\beta} \rangle$. 
Letting 
$T = T_{a,1}$,  $A = X_{\alpha}$,  $B = X_{\alpha + 2 \beta}$,  we find that 
$\langle T_{a,1},  X_{\alpha}, X_{\alpha + 2 \beta} \rangle$ is an example of 
the algebra $\alga$ of Lemma \ref{lemma:alg3},  with $r = 2a$.   
Letting $b = - \frac{2r}{(r+2)^{2}} = - \frac{a}{(a+1)^{2}}$,  we find that 
$\langle T_{a,1},  X_{\alpha}, X_{\alpha + 2 \beta} \rangle \cong L^{3}_{b} $.    
\smallskip

Letting $T = T_{3,1}$, $A = X_{\alpha}+ X_{\beta}$,  $B = X_{\alpha+ 2\beta}$, 
we find an example of 
the algebra $\alga$ of Lemma \ref{lemma:alg3},  with $r = 6$.   We conclude that
$\langle T_{3,1} , X_{\alpha}+ X_{\beta},  X_{\alpha+ 2\beta} \rangle
\cong L^{3}_{-3/16}$.
\smallskip

Next,  letting $x_{1}  = X_{\alpha} + X_{\alpha + \beta}$,  
$x_{2}  =  X_{\alpha + \beta}$,  
 $x_{3} = T_{1,0}$,
 we find that $[x_{3}, x_{1}] = x_{2}$, 
 $[x_{3}, x_{2}] = x_{2}$, $[x_{1}, x_{2}] = 0$,  which are the defining relations for $L^{3}_{0}$.
 So 
$\langle  T_{1,0}, X_{\alpha}, X_{\alpha + \beta} \rangle \cong L^{3}_{0}$.  
\smallskip

Similarly,  with 
$x_{1}  = X_{\alpha} + X_{\alpha + 2 \beta}$,  
$x_{2}  =  X_{\alpha + 2 \beta}$,  
 $x_{3} = \frac{1}{2} T_{1,0}$,
 we find that $[x_{3}, x_{1}] = x_{2}$, 
 $[x_{3}, x_{2}] = x_{2}$, $[x_{1}, x_{2}] = 0$,  which are the defining relations for $L^{3}_{0}$.
Accordingly,  
$\langle  T_{1,0}, X_{\alpha}, X_{\alpha + 2 \beta} \rangle \cong L^{3}_{0}$. 
 \smallskip

Letting $T = T_{1,0}$, $A = X_{\alpha+ 2 \beta}$,  $B = X_{\alpha+ \beta}$, 
we find an example of 
the algebra $\alga$ of Lemma \ref{lemma:alg3},  with $r = 1$.   We conclude that
$\langle T_{1,0},  X_{\alpha+ 2 \beta},  X_{\alpha+ \beta} \rangle
\cong L^{3}_{-2/9}$.
\smallskip

Next,  letting $x_{1}  = X_{\alpha + \beta} + X_{\alpha + 2 \beta}$,  
$x_{2}  =  X_{\alpha + 2 \beta}$,  
 $x_{3} = \frac{1}{2}  T_{1,-1}$,
 we find that $[x_{3}, x_{1}] = x_{2}$, 
 $[x_{3}, x_{2}] = x_{2}$, $[x_{1}, x_{2}] = 0$,  which are the defining relations for $L^{3}_{0}$.
 So 
$\langle  T_{1,-1}, X_{\alpha + \beta}, X_{\alpha +2 \beta} \rangle \cong L^{3}_{0}$.  
\smallskip

Consider the algebra 
$\langle T_{1,-1},  X_{\alpha}, X_{\alpha + 2\beta} \rangle$. 
Letting 
$T = T_{1,-1}$,  $A =  X_{\alpha + 2 \beta}$,  $B = X_{\alpha}$,  we find that 
$\langle T_{1,-1},  X_{\alpha}, X_{\alpha + 2 \beta} \rangle$ is an example of 
the algebra $\alga$ of Lemma \ref{lemma:alg3},  with $r = -2$.   
From Eq. \eqref{lem.alg3.L41}, we find that 
$\langle T_{1,-1},  X_{\alpha}, X_{\alpha + 2\beta} \rangle \cong L^{4}_{1}$.
\smallskip

Consider the algebra 
$\langle T_{1,-1},  X_{\beta}, X_{\alpha + 2\beta} \rangle$. 
Letting 
$T = T_{1,-1}$,  $A = X_{\beta}$,  $B =  X_{\alpha + 2 \beta}$,  we find that 
$\langle T_{1,-1},  X_{\beta}, X_{\alpha + 2 \beta} \rangle$ is an example of 
the algebra $\alga$ of Lemma \ref{lemma:alg3},  with $r = 2$.   
From Eq. \eqref{lem.alg3.L2}, we find that 
$\langle T_{1,-1},  X_{\beta}, X_{\alpha + 2\beta} \rangle \cong L^{2}$.
\smallskip

Similarly, consider  
$\langle T_{1,1},  X_{\alpha}, X_{\alpha + 2\beta} \rangle$. 
Letting 
$T = T_{1,1}$,  $A = X_{\alpha}$,  $B =  X_{\alpha + 2 \beta}$,  we find that 
$\langle T_{1,1},  X_{\alpha}, X_{\alpha + 2 \beta} \rangle$ is an example of 
the algebra $\alga$ of Lemma \ref{lemma:alg3},  with $r = 2$.   
From Eq. \eqref{lem.alg3.L2}, we find that 
$\langle T_{1,1},  X_{\alpha}, X_{\alpha + 2\beta} \rangle \cong L^{2}$.

 \smallskip

If we let 
$x_{1} = 2 (X_{\alpha + \beta} -  X_{\alpha+2\beta})$, 
$x_{2} = X_{\alpha+\beta} $, 
and $x_{3} = \frac{1}{4} (T_{1,1} + X_{\beta})$,  
then we find that 
$[x_{3}, x_{1} ] =  x_{2}$,
$[x_{3}, x_{2} ] =  - \frac{1}{4} x_{1} + x_{2}$,
$[x_{1}, x_{2} ] =  0$.
So  $\langle T_{1,1} + X_{\beta}, X_{\alpha+\beta} , X_{\alpha+2\beta} \rangle \cong L^{3}_{-1/4}$.  
\smallskip

Consider the algebra 
$\langle T_{1,-1} + X_{\alpha + \beta},  X_{\alpha}, X_{\alpha + 2 \beta} \rangle$. 
Letting 
$T = T_{1,-1} + X_{\alpha + \beta}$,  $A =  X_{\alpha + 2 \beta}$,  $B = X_{\alpha}$,  we find that 
$\langle T_{1,-1} + X_{\alpha + \beta},  X_{\alpha}, X_{\alpha + 2 \beta} \rangle$ is an example of 
the algebra $\alga$ of Lemma \ref{lemma:alg3},  with $r = -2$.   
From Eq. \eqref{lem.alg3.L41}, we find that 
$\langle T_{1,-1} + X_{\alpha + \beta},  X_{\alpha}, X_{\alpha + 2\beta} \rangle \cong L^{4}_{1}$.
\smallskip

Letting $T = T_{1,0} + X_{\alpha}$, $A = X_{\alpha+ 2 \beta}$,  $B = X_{\alpha+ \beta}$, 
we find an example of 
the algebra $\alga$ of Lemma \ref{lemma:alg3},  with $r = 1$.   We conclude that
$\langle T_{1,0} + X_{\alpha},  X_{\alpha+  \beta},  X_{\alpha + 2 \beta} \rangle
\cong L^{3}_{-2/9}$.
\smallskip

To see that $\langle  X_{\beta}, X_{\alpha +\beta}, X_{\alpha + 2 \beta}  \rangle \sim L^{4}_{0}$,  let 
\begin{equation} 
x_{1} = X_{\alpha +\beta}, \quad 
x_{2} = X_{\alpha + 2 \beta},  \quad 
x_{3} =  \frac{1}{2} X_{\beta}, 
\end{equation} 
and observe that these elements generate an algebra isomorphic to $L^{4}_{0}$,  whose
only nonzero bracket is 
$[x_{3}, x_{1}] = x_{2}$.

To see that $\langle  X_{\alpha} + X_{\beta}, X_{\alpha +\beta}, X_{\alpha + 2 \beta}  \rangle \sim L^{4}_{0}$,  let 
\begin{equation} 
x_{1} = X_{\alpha +\beta},   \quad 
x_{2} = X_{\alpha + 2 \beta},  \quad 
x_{3} =  \frac{1}{2} (X_{\alpha} + X_{\beta}), 
\end{equation}
and observe that these elements generate an algebra isomorphic to $L^{4}_{0}$.

The algebra 
$L^{1}$ in de Graaf's classification 
is the abelian algebra. 
The algebra 
 $L^{2}$ in de Graaf's classification is given by 
 $[x_{3}, x_{1}] = x_{1}$, 
  $[x_{3}, x_{2}] = x_{2}$. 
It is isomorphic to $\mathfrak{s}_{3,1}$ in \v{S}nobl and Winternitz's classification, with the parameter 
having the value $A=1$.  
The algebra $\mathfrak{s}_{3,1}$ is given by  
$[e_{3}, e_{1}] = e_{1}$,  
$[e_{3}, e_{2}] =  A e_{2}$,  subject to the following conditions:
\begin{equation}\label{Eq:ParameterConditionsSW}
 0 < |A|  \le 1,  ~\text{and if}~ |A| = 1, ~\text{text then}~ 0 \le \arg (A) \le \pi.  
 \end{equation}
The isomorphism is just $x_{i} \longleftrightarrow e_{i}$,  for $i = 1 , 2 , 3$. 

The algebra 
$L^{3}_{0}$ in de Graaf's classification 
is isomorphic to $\mathfrak{n}_{1,1} \oplus \mathfrak{s}_{2,1}$ in \v{S}nobl and Winternitz's classification, with the parameter 
having the value $A=0$.  
The isomorphism is given by 
\begin{equation}
\begin{array}{llll}
x_{1} - x_{2}& \longleftrightarrow & (e_{1},0)  \\
x_{3}& \longleftrightarrow &(0, e_{2})  \\
x_{2}& \longleftrightarrow &(0, e_{1}). 
\end{array}
\end{equation}

The algebra 
$L^{3}_{-1/4}$ in de Graaf's classification 
is isomorphic to $\mathfrak{s}_{3,2}$ in \v{S}nobl and Winternitz's classification, which is 
given by 
$[e_{3}, e_{1}] = e_{1}$,  
$[e_{3}, e_{2}] = e_{1} + e_{2}$.  
  The 
correspondence is given by 
\begin{equation}
\begin{array}{llll}
 2x_{3} & \longleftrightarrow &e_{3}  \\
 x_{1} - 2x_{2} & \longleftrightarrow  & e_{1}  \\
 x_{1} - 4x_{2}& \longleftrightarrow & e_{2}. 
\end{array}
\end{equation}

The algebra 
 $L^{4}_{0}$ in de Graaf's classification is given by 
$[x_{3}, x_{1}] = x_{2}$.
It is isomorphic to $\mathfrak{n}_{3,1}$ in \v{S}nobl and Winternitz's classification, 
which is given by
$[e_{2}, e_{3}] = e_{1}$.  The isomorphism is 
\begin{equation}
\begin{array}{llll}
 x_{1} & \longleftrightarrow &e_{3}  \\
x_{2} & \longleftrightarrow  & e_{1}  \\
 x_{3} & \longleftrightarrow & e_{2}. 
\end{array}
\end{equation}

The algebra 
 $L^{4}_{1}$ in de Graaf's classification is given by 
$[x_{3}, x_{1}] = x_{2}$,
$[x_{3}, x_{2}] =  x_{1}$.  It  
is isomorphic to $\mathfrak{s}_{3,1}$ in \v{S}nobl and Winternitz's classification, with the parameter 
having the value $A= -1$.  
The isomorphism is 
\begin{equation}
\begin{array}{llll}
 x_{3} & \longleftrightarrow &e_{3}  \\
x_{1} + x_{2} & \longleftrightarrow  & e_{1}  \\
x_{1} - x_{2}  & \longleftrightarrow & e_{2}. 
\end{array}
\end{equation}

In de Graaf's classification,  $L^{3}_{\alpha}$ is given by 
$[x_{3}, x_{1}] = x_{2}$,
$[x_{3}, x_{2}] = \alpha x_{1} + x_{2}$.


\begin{lemma}\label{SWinverseconditions}
Suppose $\alpha \ne 0, - \frac{1}{4}$. Then there is one choice of the square root 
$\sqrt{1 + 4 \alpha}$ for which the complex number
\begin{equation}
\lambda = \frac{1+ 2 \alpha + \sqrt{1 + 4 \alpha}}{-2 \alpha}.
\end{equation}
satifies the conditions of Eq. \eqref{Eq:ParameterConditionsSW}.
\end{lemma}

\begin{proof}
Fix one value of the square root and
let $\lambda^{+} = \frac{1 + \sqrt{1 + 4 \alpha}}{2}$,  $\lambda^{-} = \frac{1 - \sqrt{1 + 4 \alpha}}{2}$.  Then 
$ \lambda^{+}  \lambda^{-} = - \alpha$,  and either 
$ \frac{ \lambda^{+}}{\lambda^{-}}$ or $ \frac{ \lambda^{-}}{\lambda^{+}}$ satisfies the conditions of 
Eq. \eqref{Eq:ParameterConditionsSW}.  If the former,  
the initial choice of square root is correct.  If the latter,  then 
make the other choice of square root.  Either way,  then 
$\lambda = \frac{ \lambda^{+}}{\lambda^{-}}$ satisfies the conditions. 
\end{proof}

If $\alpha \ne 0, - \frac{1}{4}$, then $L^{3}_{\alpha} \cong \mathfrak{s}_{3,1, A=\lambda}$, with 
parameter
\begin{equation}
\lambda = \frac{1+ 2 \alpha + \sqrt{1 + 4 \alpha}}{-2 \alpha}
\end{equation}
as described in Lemma \ref{SWinverseconditions}.  

The isomorphism is given by 
\begin{equation}
\begin{array}{llll}
- \frac{\lambda^{+}}{\alpha} x_{3} & \longleftrightarrow &e_{3}  \\
 \alpha x_{1} + \lambda^{+} x_{2} & \longleftrightarrow  & e_{2}  \\
 \alpha x_{1} + \lambda^{-} x_{2} & \longleftrightarrow & e_{1},  
\end{array}
\end{equation}
where $\lambda^{+}$ and $\lambda^{-}$ are as in the proof of 
Lemma \ref{SWinverseconditions}.  Note that $\lambda^{+} + \lambda^{-} = 1$. 


 \subsection{Dimension $4$}\label{4dimequivalences} We first identify the four-dimensional subalgebras of $\soo$ with respect de Graaf's classification \cite{degraafa}. For each algebra of de Graaf's classification that appears, we then identify which algebra it is isomorphic to  in the classification described by \v{S}nobl and Winternitz in \cite{levi}.

To see that $  \langle \algt , X_{\alpha}, X_{\alpha + \beta} \rangle \sim M^{8}$, let
\begin{equation} 
x_{1} = T_{-1/2,1/2},  \quad 
x_{2} = X_{\alpha}, \quad 
x_{3} = T_{1,0},  \quad 
x_{4} = X_{\alpha + \beta},
\end{equation}	
and observe that these elements generate an algebra isomorphic to $M^{8}$,  whose
only nonzero brackets are 
$[x_{1}, x_{2}] =  x_{2}$ and  $[x_{3}, x_{4}] = x_{4}$.

To see that $\langle \algt , X_{\alpha}, X_{\alpha + 2\beta} \rangle \sim M^{8}$, let
\begin{equation} 
x_{1} = T_{0,1/2},  \quad 
x_{2} = X_{\alpha},   \quad 
x_{3} = T_{1/2,0},  \quad 
x_{4} = X_{\alpha +2 \beta},
\end{equation}
and observe that these elements generate an algebra isomorphic to $M^{8}$.  
	
In $ \langle T_{a,1}, \algn_{\p} \rangle$,  $a \ne 0 \pm 1$,
let 
\begin{equation}
\begin{array}{llll}
x_{1} &=& \frac{9}{4}(a+1)^{2} X_{\alpha}  + 9 X_{\alpha + \beta} +   \frac{9}{4}\frac{(a+1)^{2}}{a^{2}}  X_{\alpha + 2 \beta} \\
x_{2} &=& \frac{3}{2}(a+1) X_{\alpha}  + 3 X_{\alpha + \beta} +   \frac{3}{2}\frac{(a+1)}{a}  X_{\alpha + 2 \beta} \\
x_{3} &=&  X_{\alpha}  + X_{\alpha + \beta} +   X_{\alpha + 2 \beta} \\
x_{4} &=& \frac{1}{3(a+1)} T_{a,1}. 
\end{array}
\end{equation} 
Then
$[x_{4}, x_{1}] = x_{2}$,
$[x_{4}, x_{2}] = x_{3}$,  and 
$[x_{4}, x_{3}] =     \frac{2}{3(a+1)} X_{\alpha} + \frac{1}{3} X_{\alpha + \beta}   + \frac{2a}{3(a+1)} X_{\alpha + 2 \beta}$.  

It is not difficult to check that $[x_{4}, x_{3}] 
= Ax_{1} + Bx_{2} + x_{3}$, 
where $A = \frac{4a}{27(a+1)^{2}}$ and
$B = - \frac{2(a^{2}+4a+1)}{9(a+1)^{2}}$.  
These are the nonzero defining relations for $M^{6}_{A,B}$,  and we find that 
 $ \langle T_{a,1}, \algn_{\p} \rangle \sim M^{6}_{A,B} =$\\ $M^{6}_{4a/(27(a+1)^{2}), - 2(a^{2}+4a+1)/(9(a+1)^{2})}$.
  
In   $\langle T_{a,1}, X_{\beta}, X_{\alpha + \beta}, X_{\alpha + 2 \beta} \rangle$, with $a \ne 0 \pm 1$,    
let 
$x_{1} = X_{\beta} + X_{\alpha + \beta}$, 
$x_{2} = \frac{8a}{a^{2}-1}  X_{\alpha + 2 \beta}$, 
$x_{3} = \frac{2a}{a-1} X_{\beta} + \frac{2a}{a+1} X_{\alpha+\beta}$, 
$x_{4} = \frac{1}{2a} T_{a,1}$.  
Then it is easy to verify that the only nonzero relations are 
$[x_{4}, x_{1}] = x_{1} + \frac{1-a^{2}}{4a^{2}} x_{3}$,
$[x_{4}, x_{2}] = x_{2}$,
$[x_{4}, x_{3}] = x_{1}$,
$[x_{3}, x_{1}] = x_{2}$.  
These are the defining relations for $M^{13}_{(1-a^{2})/(4a^{2})}$,  so 
$\langle T_{a,1}, X_{\beta}, X_{\alpha + \beta}, X_{\alpha + 2 \beta} \rangle \sim 
M^{13}_{(1-a^{2})/(4a^{2})}$.
  
In $ \langle T_{3,1}, X_{\alpha} + X_{\beta} , X_{\alpha + \beta}, X_{\alpha + 2 \beta} \rangle$, let
$x_{1} = X_{\alpha} + X_{\beta} + 2X_{\alpha + \beta}$, 
$x_{2} = 6 X_{\alpha + 2 \beta}$, 
$x_{3} = 3(X_{\alpha} + X_{\beta}) + 3 X_{\alpha+\beta}$, 
$x_{4} = \frac{1}{6} T_{3,1}$.  

Then it is easy to verify that the only nonzero relations are 
$[x_{4}, x_{1}] = x_{1} - \frac{2}{9} x_{3}$,
$[x_{4}, x_{2}] = x_{2}$,
$[x_{4}, x_{3}] = x_{1}$,
$[x_{3}, x_{1}] = x_{2}$.  
These are the defining relations for $M^{13}_{-2/9}$,  so 
$\langle T_{3,1}, X_{\alpha} + X_{\beta} , X_{\alpha + \beta}, X_{\alpha + 2 \beta} \rangle
\sim M^{13}_{-2/9}$. 

In $\langle T_{0,1}, \algn_{\p} \rangle$,   
let 
$x_{1} = \frac{9}{4} X_{\alpha}  +  9 X_{\alpha + \beta} +    X_{\alpha + 2 \beta} $,
$x_{2} = \frac{3}{2}   X_{\alpha}  + 3 X_{\alpha + \beta}  $,
$x_{3} =  X_{\alpha}  + X_{\alpha + \beta}  $,
$x_{4} = \frac{1}{3} T_{0,1}$. 

Then
$[x_{4}, x_{1}] = x_{2}$,
$[x_{4}, x_{2}] = x_{3}$,  and 
$[x_{4}, x_{3}] =     \frac{2}{3} X_{\alpha} + \frac{1}{3} X_{\alpha + \beta}$.  

It is not difficult to check that $[x_{4}, x_{3}] 
= - \frac{2}{9} x_{2} +x_{3}$. 
These are the nonzero defining relations for $M^{6}_{0,-2/9}$, so 
$\langle T_{0,1}, \algn_{\p} \rangle \sim M^{6}_{0,-2/9}$.    

In $\langle T_{0,1}, X_{\beta}, X_{\alpha + \beta}, X_{\alpha + 2 \beta} \rangle$, let 
 $x_{1} = - X_{\beta} + X_{\alpha + \beta}$, 
 $x_{2} =  4 X_{\alpha + 2 \beta}$,
  $x_{3} = X_{\beta} + X_{\alpha + \beta}$,
 $x_{4} = T_{0,1}$.  
 Then 
 $[x_{4}, x_{1}] = x_{3}$, 
  $[x_{4}, x_{3}] =  x_{1}$,  and 
 $[x_{3}, x_{1}] = x_{2}$.  These are the nonzero relations for $M^{14}_{1}$,  so we find that     
 $\langle T_{0,1}, X_{\beta}, X_{\alpha + \beta}, X_{\alpha + 2 \beta} \rangle \sim M^{14}_{1}$. 

 In $\langle T_{1,0}, X_{\beta}, X_{\alpha + \beta}, X_{\alpha + 2 \beta} \rangle$,  let  
 $x_{1} = X_{\alpha + \beta}$, 
 $x_{2} = X_{\alpha + 2 \beta}$,
  $x_{3} = \frac{1}{2} X_{ \beta}$,
 $x_{4} = T_{1,0}$.  
 Then 
 $[x_{4}, x_{1}] = x_{1}$, 
  $[x_{4}, x_{2}] = 2 x_{2}$, 
   $[x_{4}, x_{3}] = x_{3}$, and 
 $[x_{3}, x_{1}] = x_{2}$.  These are the nonzero relations for $M^{12}$,  so we find that     
 $\langle T_{1,0}, X_{\beta}, X_{\alpha + \beta}, X_{\alpha + 2 \beta} \rangle 
 \sim M^{12}$. 
 
 In $\langle T_{1,1}, \algn_{\p} \rangle$,
 let $x_{1} = X_{\alpha}$, 
  $x_{2} = X_{\alpha +\beta}$, 
   $x_{3} = X_{\alpha+ 2 \beta}$, 
    $x_{4} = \frac{1}{2} T_{1,1}$.
Then the nonzero relations are 
$[x_{4} , x_{1}] = x_{1}$,
$[x_{4} , x_{2}] = x_{2}$,
$[x_{4} , x_{3}] = x_{3}$,      
showing that $\langle T_{1,1}, \algn_{\p} \rangle \sim M^{2}$. 

In 
$\langle T_{1,-1}, \algn_{\p} \rangle$ 
let 
$x_{1} =  X_{\alpha}  +   X_{\alpha + \beta} +    X_{\alpha + 2 \beta} $,
$x_{2} =  X_{\alpha}  - X_{\alpha + 2 \beta}  $,
$x_{3} =  X_{\alpha}  + X_{\alpha + 2 \beta}  $,
$x_{4} = \frac{1}{2} T_{1,-1}$. 

Then
$[x_{4}, x_{1}] = x_{2}$,
$[x_{4}, x_{2}] = x_{3}$,  and 
$[x_{4}, x_{3}] =      X_{\alpha} - X_{\alpha + 2 \beta}  = x_{2} $. 
These are the nonzero defining relations for $M^{7}_{0,1}$, so 
$\langle T_{1,-1}, \algn_{\p} \rangle \sim M^{7}_{0,1}$.    

In $\langle T_{1,1}, X_{\beta}, X_{\alpha + \beta}, X_{\alpha + 2 \beta} \rangle$, let
$x_{1} = X_{\alpha + \beta}$, 
$x_{2} = X_{\alpha + 2 \beta}$,
$x_{3} = \frac{1}{2} X_{\beta} + X_{\alpha + \beta}$, 
$x_{4} = \frac{1}{2} T_{1,1}$.  Then
$[x_{4}, x_{1} ] = x_{1}$, 
$[x_{4}, x_{2} ] = x_{2}$, 
$[x_{4}, x_{3} ] = x_{1}$, and 
$[x_{3}, x_{1} ] = x_{2}$, 
which are the nonzero relations for $M^{13}_{0}$.  We conclude that  
$\langle T_{1,1}, X_{\beta}, X_{\alpha + \beta}, X_{\alpha + 2 \beta} \rangle \sim M^{13}_{0}$.  

In $\langle T_{1,1} + X_{\beta}, X_{\alpha}, X_{\alpha+ \beta}, X_{\alpha+ 2\beta} \rangle$,  let
$x_{1} = 54 X_{\alpha}$, 
$x_{2} = 18 X_{\alpha} + 9 X_{\alpha + \beta}$, 
$x_{3} = 6 X_{\alpha} + 6 X_{\alpha + \beta} + 3 X_{\alpha + 2 \beta}$, 
$x_{4} = \frac{1}{6} (T_{1,1} + X_{\beta})$.  
Then
$[x_{4}, x_{1}] = x_{2}$,  $[x_{4}, x_{2}] = x_{3}$, 
and  $[x_{4}, x_{3}] = 2X_{\alpha} + 3  X_{\alpha + \beta} + 3 X_{\alpha + 2 \beta}$, 
which equals $\frac{1}{27} x_{1} - \frac{1}{3} x_{2} + x_{3}$.  
These are the only nonzero relations,  so we find that
$\langle T_{1,1} + X_{\beta}, X_{\alpha}, X_{\alpha+ \beta}, X_{\alpha+ 2\beta} \rangle
\sim M^{6}_{1/27, -1/3}$. 

In $\langle  T_{1,0} + X_{\alpha} , X_{\beta},  X_{\alpha+ \beta}, X_{\alpha+ 2\beta} \rangle$, let
$x_{1} = - \frac{1}{2} X_{\beta} + \frac{1}{2} X_{\alpha + \beta} $, 
$x_{2} = - X_{\alpha + 2 \beta} $, 
$x_{3} = - X_{\beta}$, 
$x_{4} =  \frac{1}{2} (T_{1,0} + X_{\alpha}) $.  Then
$[x_{4}, x_{1}] = x_{1} - \frac{1}{4} x_{3}$,  $[x_{4}, x_{2}] = x_{2}$, 
and  $[x_{4}, x_{3}] = x_{1}$,  and
$[x_{3}, x_{1}] = x_{2}$.  
These are the nonzero relations for $M^{13}_{-1/4}$.  We conclude that  
$\langle  T_{1,0} + X_{\alpha} , X_{\beta},  X_{\alpha+ \beta}, X_{\alpha+ 2\beta} \rangle \sim M^{13}_{-1/4}$.

In $ \algn$,  let 
$x_{1} = X_{\alpha}$, 
$x_{2} = X_{\alpha + \beta}$, 
$x_{3} = 2 X_{\alpha + 2 \beta}$, 
$x_{4} = X_{\beta}$.
Then 
$[x_{4}, x_{1}] = x_{2}$ and  $[x_{4}, x_{2}] = x_{3}$ 
are the only nonzero relations,  and we find that
$\algn \sim M^{7}_{0,0}$.

The algebra 
$M^{2}$ in de Graaf's classification is given by 
$[x_{4}, x_{1}] = x_{1}$, $[x_{4}, x_{2}] = x_{2}$, $[x_{4}, x_{3}] = x_{3}$.  

It is isomorphic to  $\mathfrak{s}_{4,3}$ with parameters $A= B =1$ 
in \v{S}nobl and Winternitz's classification,  which is given by    
$[e_{4}, e_{1}] = e_{1}$,
$[e_{4}, e_{2}] =  e_{2}$,
$[e_{4}, e_{3}] = e_{3}$.
The isomorphism is 
$x_{i} \longleftrightarrow  e_{i} $,  for $i = 1, \dots, 4$.

The algebra 
$M^{8}$ in de Graaf's classification is given by 
$[x_{1}, x_{2}] = x_{2}$, $[x_{3}, x_{4}] = x_{4}$.  It is isomorphic to $K^{2} \oplus K^{2}$,  so it 
is isomorphic to $\mathfrak{s}_{2,1} \oplus \mathfrak{s}_{2,1}$ in \v{S}nobl and Winternitz's classification, 
which is also isomorphic to $\mathfrak{s}_{4,12}$ over $\Cb$.    

The algebra 
$M^{6}_{\alpha, \beta}$ in de Graaf's classification is given by 
$[x_{4}, x_{1}] = x_{2}$,  
$[x_{4}, x_{2}] = x_{3}$,  
$[x_{4}, x_{3}] = \alpha x_{1} + \beta x_{2} +x_{3}$.   

This description of $M^{6}_{\alpha,\beta}$ shows that its nilradical $\mathfrak {n}$ is 
abelian,  with basis $\{ x_{1}, x_{2}, x_{3} \}$.  The action of $x_{4}$ on $\mathfrak {n}$ 
is given,  relative to this basis, by the matrix 
\begin{equation}
C = \begin{pmatrix}
0&0& 
 \beta 
\\
1 & 0 &
 \alpha
\\
0&1&1 
\end{pmatrix}.\end{equation} This is a companion matrix,  with characteristic polynomial 
$\lambda^{3}   -   \lambda^{2}    - \alpha   \lambda  -  \beta$.

Since the roots add up to $1$,  the only way there could be a single root of 
multiplicity $3$ would be for it to be $\frac{1}{3}$.  This happens with 
$ \beta  = \frac{1}{27}$,
 $ \alpha  = -\frac{1}{3}$,
and it is not difficult to see that in this case, the Jordan Canonical Form of $C$ 
has a single block. In particular,  ${\rm ad}(3x_{4})$ has an eigenvector $u_{1} \in \mathfrak{n}$ 
with eigenvalue $1$, and there are vectors $v,w \in \mathfrak{n}$ so that 
$[ 3x_{4}, v] = u_{1} +v$,
$[ 3x_{4}, w] = v + w$.

This algebra is isomorphic to 
$\mathfrak{s}_{4,2}  $, which is given by 
$[e_{4}, e_{1}] = e_{1}$,
$[e_{4}, e_{2}] = e_{1} + e_{2}$,
$[e_{4}, e_{3}] = e_{2} + e_{3}$. 
The isomorphism is given by 
\begin{equation}
\begin{array}{llll}
3 x_{4}  & \longleftrightarrow &e_{4}  \\
u_{1} & \longleftrightarrow & e_{1}  \\
v & \longleftrightarrow  &  e_{2}  \\
w & \longleftrightarrow  &  e_{3}.  
\end{array}
\end{equation}

Depending on $\alpha$ and $\beta$,  the matrix $C$ may have three distinct 
eigenvalues,  two distinct eigenvalues with one of them associated to a $2 \times 2$ 
Jordan block,  or the single eigenvalue $\lambda = \frac{1}{3}$,  as discussed above. 

Suppose $ \beta  \ne 0$,  so that the eigenvalues are all nonzero.  

If there are three distinct (nonzero) eigenvalues,  we can assume they satisfy
$|r'| \ge |s'| \ge |t'| > 0$.  If $|r'| > |s'|$,  then $\frac{1}{r'} x_{4}$ has eigenvalues 
$1, s, t$,  with $1 > |s| \ge |t| > 0$.  If $|r'| = |s'| > |t'|$,  then 
$\frac{1}{r'} x_{4}$ has eigenvalues 
$1, s, t$,  with $1 = |s| \ge |t| > 0$ and we can assume $0 < {\rm arg}(s) < 2 \pi$.  
If $|r'| = |s'| = |t'|$,  then 
$\frac{1}{r'} x_{4}$ has eigenvalues 
$1, s, t$,  with $1 = |s| = |t| $ and we can assume $0 < {\rm arg}(s)  < {\rm arg}(t) < 2 \pi$.  
In each of these cases,  let us write $u_{1}, u_{s}, u_{t} \in \mathfrak{n}$ for the corresponding eigenvectors.  

In each of these cases,  $M^{6}_{\alpha,\beta} \cong \mathfrak{s}_{4,3}$, with parameters 
$A = s$, $B = t$.    
This algebra is given by 
$[e_{4}, e_{1}] = e_{1}$,
$[e_{4}, e_{2}] = A e_{2}$,
$[e_{4}, e_{3}] = B e_{3}$. 
The isomorphism is given by 
 \begin{equation}
\begin{array}{llll}
\frac{1}{r'} x_{4}  & \longleftrightarrow &e_{4}  \\
u_{1} & \longleftrightarrow & e_{1}  \\
u_{s} & \longleftrightarrow  &  e_{2}  \\
u_{t} & \longleftrightarrow  &  e_{3}.  
\end{array}
\end{equation}

If there are two distinct (nonzero) eigenvalues,  we can assume that $r'$ is associated 
to a $2 \times 2$ Jordan block and $s'$ is associated 
to a $1 \times 1$ Jordan block.  Then $\frac{1}{r'} x_{4}$ has an eigenvector $u_{s}$ with 
eigenvalue $s = \frac{s'}{r'}$,  an  eigenvector $u_{1}$ with 
eigenvalue $1$,  and a vector $v$ so that $[x_{4}, v] = u_{1} + v$.  
In this case, $M^{6}_{\alpha,\beta} \cong \mathfrak{s}_{4,4}$, with parameter 
$A= s$.  

This algebra is given by 
$[e_{4}, e_{1}] = e_{1}$,
$[e_{4}, e_{2}] = e_{1} + e_{2}$,
$[e_{4}, e_{3}] = s e_{3}$. 
The isomorphism is given by 
\begin{equation}
\begin{array}{llll}
\frac{1}{r'} x_{4}  & \longleftrightarrow &e_{4}  \\
u_{1} & \longleftrightarrow & e_{1}  \\
v & \longleftrightarrow  &  e_{2}  \\
u_{s} & \longleftrightarrow  &  e_{3}.  
\end{array}
\end{equation}

If $\beta  = 0$,  $ \alpha \ne  0, - \frac{1}{4}$, then $M^{6}_{0,  \alpha }$ is isomorphic to 
$\mathfrak{n}_{1,1} \oplus \mathfrak{s}_{3,1}$ in \v{S}nobl and Winternitz's classification, 
with parameter
\begin{equation}
A = 
\lambda = \frac{1+ 2  \alpha + \sqrt{1 + 4  \alpha }}{-2 \alpha }.
\end{equation}
With the choice of square root described in 
Lemma \ref{SWinverseconditions},  the 
isomorphism is given by 
\begin{equation}
\begin{array}{llll}
\alpha  x_{1} + x_{2} - x_{3} & \longleftrightarrow &(e_{1},0)  \\
- \frac{ 
\lambda^{+}
}{ \alpha } x_{4} & \longleftrightarrow &(0, e_{3})  \\
 \alpha  x_{2} + \lambda^{+} x_{3} & \longleftrightarrow  & (0, e_{2})  \\
 \alpha  x_{2} + \lambda^{-} x_{3} & \longleftrightarrow & (0, e_{1} ). 
\end{array}
\end{equation} 

If $\beta = 0$,  $\alpha = - \frac{1}{4}$, then $M^{6}_{0, -1/4}$ is isomorphic to 
$\mathfrak{n}_{1,1} \oplus \mathfrak{s}_{3,2}$ in \v{S}nobl and Winternitz's classification.  
The isomorphism is given by 
\begin{equation}
\begin{array}{llll}
x_{1} - 4 x_{2} + 4 x_{3} & \longleftrightarrow &(e_{1},0)  \\
2 x_{4} & \longleftrightarrow &(0, e_{3})  \\
2 x_{2} - 4 x_{3} & \longleftrightarrow  & (0, e_{1})  \\
-4 x_{3} & \longleftrightarrow & (0, e_{2} ). 
\end{array}
\end{equation}

We give some examples.

Suppose $\beta =  1$, $\alpha = - 1$.   The algebra $M^6_{1,-1}$ is isomorphic to 
$\mathfrak{s}_{4,3}$. 
The isomorphism is given by 
\begin{equation}
\begin{array}{llll}
x_{4}  & \longleftrightarrow &e_{4}  \\
x_{1} +  x_{3} & \longleftrightarrow & e_{1}  \\
x_{1} +(-1-i) x_{2} + i x_{3} & \longleftrightarrow  &  e_{2}  \\
x_{1} +(-1+i) x_{2} - i x_{3} & \longleftrightarrow  &  e_{3}.  
\end{array}
\end{equation}

Now suppose $\beta =  -1$, $\alpha = 1$.   The algebra $M^6_{-1, 1}$ is isomorphic to 
$\mathfrak{s}_{4,4}$.
The isomorphism is given by 
\begin{equation}
\begin{array}{llll}
x_{4}  & \longleftrightarrow &e_{4}  \\
x_{1} -  x_{3} & \longleftrightarrow & e_{1}  \\
- \frac{1}{2} x_{1} - x_{2} - \frac{1}{2}  x_{3} & \longleftrightarrow  &  e_{2}  \\
x_{1} -2 x_{2} + x_{3} & \longleftrightarrow  &  e_{3}.  
\end{array}
\end{equation}

The algebra 
$M^{13}_{\alpha}$ in de Graaf's classification is given by 
$[x_{4}, x_{1}] = x_{1} + \alpha x_{3}$,  
$[x_{4}, x_{2}] = x_{2}$,  
$[x_{4}, x_{3}] =  x_{1}$,   
$[x_{3}, x_{1}] =  x_{2} $.   

First consider the case in which $\alpha = 0$.  Then $M^{13}_{0}$ is isomorphic to  
$\mathfrak{s}_{4,11}$
in \v{S}nobl and Winternitz's classification,  which is given by  
$[e_{4}, e_{1}] = e_{1}$,
$[e_{4}, e_{2}] =  e_{2}$,
$[e_{2}, e_{3}] = e_{1}$.

The isomorphism is 
\begin{equation}
\begin{array}{llll}
x_{4} & \longleftrightarrow &e_{4}  \\
x_{1} & \longleftrightarrow &e_{2}  \\
x_{2} & \longleftrightarrow  & e_{1}  \\
x_{1} - x_{3} & \longleftrightarrow & e_{3}. 
\end{array}
\end{equation}

Now suppose $\alpha = - \frac{1}{4}$.  Then 
$M^{13}_{-1/4}$ is isomorphic to  
$\mathfrak{s}_{4,10}$
in \v{S}nobl and Winternitz's classification,  which is given by  
$[e_{4}, e_{1}] = 2e_{1}$,
$[e_{4}, e_{2}] =  e_{2}$,
$[e_{4}, e_{3}] =  e_{2} + e_{3}$,
$[e_{2}, e_{3}] = e_{1} $.

The isomorphism is 
\begin{equation}
\begin{array}{llll}
2x_{4} & \longleftrightarrow &e_{4}  \\
-2x_{2} & \longleftrightarrow &e_{1}  \\
2x_{1} - x_{3} & \longleftrightarrow  & e_{2}  \\
 x_{3} & \longleftrightarrow & e_{3}. 
\end{array}
\end{equation}

Now suppose $\alpha \ne 0, - \frac{1}{4}$.   Then $M^{13}_{\alpha}$ is isomorphic to  
$\mathfrak{s}_{4,8}$, with parameter $\lambda = \frac{1+ 2\alpha +\sqrt{1 + 4 \alpha}}{-2 \alpha}$ 
as described in Lemma \ref{SWinverseconditions}.
Note that, with choice of square root specified in that lemma and 
the notation given in the proof of that lemma, 
$(1+\lambda)\lambda^+=\lambda$
and $(1+\lambda)\lambda^-=1$. The isomorphism is 
\begin{equation}
\begin{array}{llll}
(1+\lambda)~x_{4} & \longleftrightarrow &e_{4}  \\
\sqrt{1+4 \alpha} ~x_{2} & \longleftrightarrow &e_{1}  \\
\lambda^{-}~x_{1} + \alpha x_{3} & \longleftrightarrow  & e_{2}  \\
\lambda^{+}~x_{1} + \alpha x_{3} & \longleftrightarrow  & e_{3}. 
\end{array}
\end{equation}

The algebra 
$M^{14}_{1}$ in de Graaf's classification is given by 
$[x_{4}, x_{1}] =  x_{3}$,  
$[x_{4}, x_{3}] = x_{1}$,   
$[x_{3}, x_{1}] =  x_{2} $.

It is isomorphic to  $\mathfrak{s}_{4,6}$ 
in \v{S}nobl and Winternitz's classification,  which is given by  
$[e_{4}, e_{2}] = e_{2}$,
$[e_{4}, e_{3}] = - e_{3}$,
$[e_{2}, e_{3}] = e_{1}$.

The isomorphism is 
\begin{equation}
\begin{array}{llll}
2x_{1} & \longleftrightarrow &e_{1}  \\
  x_{1} + x_{3} & \longleftrightarrow &e_{2}  \\
  x_{1} - x_{3}  & \longleftrightarrow  & e_{3}  \\
x_{4} & \longleftrightarrow & e_{4}. 
\end{array}
\end{equation}


The algebra 
$M^{12}$ in de Graaf's classification is given by 
$[x_{4}, x_{1}] = x_{1}$,  
$[x_{4}, x_{2}] = 2x_{2}$,  
$[x_{4}, x_{3}] =  x_{3}$,   
$[x_{3}, x_{1}] =  x_{2} $.

It is isomorphic to  $\mathfrak{s}_{4,8}$,  with parameter $A = 1$,   
in \v{S}nobl and Winternitz's classification,  which is given by  
$[e_{4}, e_{1}] = 2e_{1}$,
$[e_{4}, e_{2}] = e_{2}$,
$[e_{4}, e_{3}] = e_{3}$, 
$[e_{2}, e_{3}] = e_{1}$. 

The isomorphism is 
\begin{equation}
\begin{array}{llll}
x_{4} & \longleftrightarrow &e_{4}  \\
 x_{1} & \longleftrightarrow &e_{3}  \\
x_{2} & \longleftrightarrow  & e_{1}  \\
x_{3} & \longleftrightarrow & e_{2}. 
\end{array}
\end{equation}


The algebra $M^{7}_{0,1}$ in de Graaf's classification is given by 
$[x_{4}, x_{1}] = x_{2}$,  
$[x_{4}, x_{2}] = x_{3}$,  
$[x_{4}, x_{3}] =  x_{2}$.

It is isomorphic to  $\mathfrak{n}_{1,1} \oplus \mathfrak{s}_{3,1}$,  with parameter $A = -1$  
in \v{S}nobl and Winternitz's classification. 
The isomorphism is 
\begin{equation}
\begin{array}{llll}
x_{1} - x_{3} & \longleftrightarrow &(e_{1}, 0)  \\  
x_{4} & \longleftrightarrow &(0, e_{3})  \\
x_{2} + x_{3} & \longleftrightarrow  & (0, e_{1} )  \\
x_{2} - x_{3} & \longleftrightarrow & (0, e_{2} ). 
\end{array}
\end{equation}


The algebra 
$M^{7}_{0,0}$ in de Graaf's classification is given by 
$[x_{4}, x_{1}] = x_{2}$,  
$[x_{4}, x_{2}] = x_{3}$.

It is isomorphic to  $\mathfrak {n}_{4,1}$
in \v{S}nobl and Winternitz's classification,  which is given by  
$[e_{4}, e_{2}] = - e_{1}$,  
$[e_{4}, e_{3}] = - e_{2}$.   

The isomorphism is 
\begin{equation}
\begin{array}{llll}
-x_{4} & \longleftrightarrow &e_{4}  \\
  x_{1} & \longleftrightarrow &e_{3}  \\  
x_{2} & \longleftrightarrow & e_{2}   \\ 
  x_{ 3} & \longleftrightarrow  & e_{1}. 
\end{array}
\end{equation}
\subsection{Dimension $5$}\label{5dimequivalences}
Since the classification of de Graaf \cite{degraafa} does not include solvable algebras of 
dimension greater than $4$,  we will only identify the five-dimensional solvable subalgebras with respect 
to \v{S}nobl and Winternitz's classification \cite{levi}.

The following isomorphism establishes $\mathfrak{s}_{5,41, A=B=\frac{1}{2}}\cong \langle \algt , \algn_{\p} \rangle$:
\begin{equation}
\begin{array}{ccccccccc}
e_1 & \longleftrightarrow & X_{\alpha} \\
e_2 & \longleftrightarrow&  X_{\alpha+2 \beta}\\
e_3 & \longleftrightarrow&  -X_{\alpha+\beta}\\
e_4& \longleftrightarrow &  \frac{1}{2} T_{0,1}\\
e_5 & \longleftrightarrow &  \frac{1}{2}T_{1,0}.
\end{array}
\end{equation}

The following isomorphism establishes  $\mathfrak{s}_{5, 44} \cong \langle \mathfrak{t}, X_{\beta}, X_{\alpha+\beta},$ $X_{\alpha+2\beta} \rangle$:
\begin{equation}
\begin{array}{ccccccccc}
e_1 & \longleftrightarrow &  2X_{\alpha+2\beta}\\
 e_2& \longleftrightarrow &  X_{\beta}\\
 e_3& \longleftrightarrow &  X_{\alpha+\beta}\\
 e_4& \longleftrightarrow &  \frac{1}{2}T_{1,-1}\\
e_5& \longleftrightarrow &  T_{0,-1}.
\end{array}
\end{equation}

The following isomorphism establishes $\mathfrak{s}_{5, 35, A=\frac{2}{a-1}}\cong  \langle T_{a,1}, \mathfrak{n} \rangle$:
\begin{equation}
\begin{array}{cccccccc}
e_1 & \longleftrightarrow &  2X_{\alpha+2\beta}\\
 e_2& \longleftrightarrow &  -X_{\alpha+\beta}\\
 e_3& \longleftrightarrow &  X_{\alpha}\\
e_4 & \longleftrightarrow &  X_{\beta}\\
e_5 & \longleftrightarrow &  \frac{1}{a-1} T_{a,1}. 
\end{array}
\end{equation}

The following isomorphism establishes $\mathfrak{s}_{5, 35, A=-1}\cong \langle T_{1,-1}, \mathfrak{n} \rangle $:
\begin{equation}
\begin{array}{cccccccc}
e_1& \longleftrightarrow & 2X_{\alpha+2\beta} \\
e_2 & \longleftrightarrow &  -X_{\alpha+\beta}\\
e_3 & \longleftrightarrow &  X_{\alpha}\\
 e_4& \longleftrightarrow &  X_{\beta}\\
e_5 & \longleftrightarrow & \frac{1}{2}T_{1,-1}.
\end{array}
\end{equation}

The following isomorphism establishes $\mathfrak{s}_{5, 37} \cong \langle T_{1,1}, \mathfrak{n} \rangle$:
\begin{equation}
\begin{array}{cccccccccc}
e_1 & \longleftrightarrow &  2X_{\alpha+2\beta}\\
 e_2& \longleftrightarrow &  -X_{\alpha+\beta}\\
 e_3& \longleftrightarrow & X_{\alpha}\\
 e_4& \longleftrightarrow &  X_{\beta}\\
e_5& \longleftrightarrow & \frac{1}{2}T_{1,1}.
\end{array}
\end{equation}

The following isomorphism establishes $\mathfrak{s}_{5, 36} \cong \langle T_{1,0}, \mathfrak{n} \rangle$:
\begin{equation}
\begin{array}{cccccccc}
e_1 & \longleftrightarrow & 2X_{\alpha+2\beta} \\
e_2 & \longleftrightarrow &  -X_{\alpha+\beta}\\
 e_3& \longleftrightarrow &  X_{\alpha}\\
e_4 & \longleftrightarrow &  X_{\beta}\\
 e_5 & \longleftrightarrow &  T_{1,0}.
\end{array}
\end{equation}

The following isomorphism establishes $\mathfrak{s}_{5, 33} \cong \langle T_{0,1}, \mathfrak{n} \rangle$:
\begin{equation}
\begin{array}{cccccccc}
e_1 & \longleftrightarrow &  2X_{\alpha+2\beta}\\
 e_2 & \longleftrightarrow &  -X_{\alpha+\beta}\\ 
e_3 & \longleftrightarrow &  X_{\alpha}\\
 e_4 & \longleftrightarrow &  X_{\beta}\\
e_5  & \longleftrightarrow &  -T_{0,1}.
\end{array}
\end{equation}

\subsection{Dimension $6$}\label{6dimequivalence}
The only solvable subalgebra of $\borel$ of dimension $6$ is 
$\borel$ itself.   Referencing the classification in \cite{levi}, the following isomorphism establishes $\mathfrak{s}_{6, 242} \cong \mathfrak{b}$: 
\begin{equation}
\begin{array}{ccccccccc}
e_1 & \longleftrightarrow &  2X_{\alpha+2\beta}\\
 e_2& \longleftrightarrow &  -X_{\alpha+\beta}\\
 e_3& \longleftrightarrow &  X_{\alpha}\\
 e_4& \longleftrightarrow &  X_{\beta}\\
e_5& \longleftrightarrow &  T_{1,0}\\
e_6 & \longleftrightarrow &T_{\frac{1}{2}, \frac{1}{2}}.
\end{array}
\end{equation}

\begin{landscape}
\begin{small}
\renewcommand{\arraystretch}{1.5}
\begin{table}[h!]
\begin{tabular}{|c|c|c|c|c|c|c|}
  \hline 
  Representative & Conditions & Equivalences & Isomorphism Class & Isomorphism Class\\ 
  &&&(de Graaf \cite{degraafa}) & (\v{S}nobl and Winternitz \cite{levi})\\
  \hline
 \multicolumn{5}{|c|}{Semisimple}  \\
  \hline
$\langle T_{a,1} \rangle$ 
&
 $a \ne 0, \pm 1$ 
    &
 $\langle T_{a,1}    \rangle \sim \langle T_{a', 1} \rangle $
& $J$  &$\mathfrak{n}_{1,1}$\\
&
& 
if and only if 
$a' = \pm a, \pm a^{-1}$ 
& &\\
  \hline
  $\langle T_{1,0} \rangle$ &  &
  & $J$&$\mathfrak{n}_{1,1}$ \\
  \hline
$\langle T_{ 1, 1} \rangle$ &  &
& $J$ &$\mathfrak{n}_{1,1}$\\
\hline 
 \multicolumn{5}{|c|}{Nilpotent} \\
  \hline
  $\langle X_{\alpha} \rangle$ & &  &$J$& $\mathfrak{n}_{1,1}$\\  
      \hline
$\langle X_{\beta} \rangle$ & &  &$J$ &$\mathfrak{n}_{1,1}$\\  
      \hline
$\langle X_{\alpha} +X_{\beta} \rangle$ & & &$J$&$\mathfrak{n}_{1,1}$ \\  
      \hline
  \multicolumn{5}{|c|}{ Non-Trivial Jordan Decomposition} \\
  \hline
  $\langle T_{1,0} + X_{\alpha}  \rangle $ &  &
      & $J$ &$\mathfrak{n}_{1,1}$\\
\hline 
$\langle T_{1,1} + X_{\beta} \rangle$ &  &
  & $J$& $\mathfrak{n}_{1,1}$ \\
\hline 
\end{tabular}
\vspace{2mm}
\caption{One-dimensional subalgebras of $\soo$,  up to conjugacy by $Sp(4, \mathbb{C})$}\label{onedimuu}
\end{table}
\end{small}
\end{landscape}
\begin{landscape}
\begin{small}
\renewcommand{\arraystretch}{1.5}
\begin{table}[h!]
\begin{tabular}{|c|c|c|c|c|c|}
  \hline 
  Representative & Conditions & Equivalences & Isomorphism Class & Isomorphism Class\\ 
  &&&(de Graaf \cite{degraafa}) &(\v{S}nobl and Winternitz \cite{levi})\\
  \hline
 \multicolumn{5}{|c|}{Semisimple}  \\
  \hline
   $\algt$ &&&$K^{1}$ &2$\mathfrak{n}_{1,1}$ \\ 
  \hline
   \multicolumn{5}{|c|}{Containing both Semisimple and Nilpotent Elements}  \\
  \hline
     $\langle T_{3,1}, X_{\alpha} + X_{\beta} \rangle$ &&&$K^{2}$&$\mathfrak{s}_{2,1}$\\
\hline
$\langle T_{a,1} , X_{\alpha} \rangle$ & $  a \ne 0, \pm 1$  & $\langle T_{a,1} , X_{\alpha} \rangle \sim \langle T_{-a,1} , X_{\alpha} \rangle$ &$K^{2}$ &$\mathfrak{s}_{2,1}$\\
\hline

$\langle T_{a,1} , X_{\beta} \rangle$ & $ a \ne 0, \pm 1$ &   $\langle T_{a,1} , X_{\beta} \rangle \sim \langle T_{a^{-1},1} , X_{\beta} \rangle$ &$K^{2}$ &$\mathfrak{s}_{2,1}$\\
\hline

 $\langle  T_{1,0}, X_{\alpha} \rangle$ &&&$K^{1}$&2$\mathfrak{n}_{1,1}$\\ 
\hline
  $ \langle  T_{1,0}, X_{\beta} \rangle$ &&&$K^{2}$&$\mathfrak{s}_{2,1}$\\ 
\hline
 $\langle  T_{1,0}, X_{\alpha  +2 \beta} \rangle $ &&&$K^{2}$&$\mathfrak{s}_{2,1}$\\
\hline
 $\langle  T_{1,1}, X_{\alpha} \rangle$ &&&$K^{2}$&$\mathfrak{s}_{2,1}$\\
\hline 
 $\langle  T_{1,1}, X_{\beta} \rangle$ &&&$K^{1}$&2$\mathfrak{n}_{1,1}$\\
\hline 
 $\langle  T_{1,1}, X_{\alpha+ \beta} \rangle$ &&&$K^{2}$&$\mathfrak{s}_{2,1}$\\
\hline
  \multicolumn{5}{|c|}{Containing no Semisimple Elements, but not Nilpotent}  \\
  \hline
  $\langle T_{1,1}+X_{\beta}, X_{\alpha + 2 \beta} \rangle$ &&&$K^{2}$&$\mathfrak{s}_{2,1}$\\
\hline 
 $\langle T_{1,0}+X_{\alpha }, X_{\alpha + \beta } \rangle$&&&$K^{2}$&$\mathfrak{s}_{2,1}$\\
\hline 
 $\langle T_{1,0}+X_{\alpha }, X_{\alpha + 2\beta } \rangle$&&&$K^{2}$&$\mathfrak{s}_{2,1}$\\
\hline
  \multicolumn{5}{|c|}{Nilpotent}  \\
  \hline
  $\langle X_{\alpha}, X_{\alpha + \beta} \rangle $ &&& $K^{1}$ &2$\mathfrak{n}_{1,1}$\\
\hline 
 $\langle X_{\alpha} ,  X_{\alpha + 2 \beta}  \rangle $ &&& $K^{1}$ & 2$\mathfrak{n}_{1,1}$\\
\hline 
 $\langle X_{\beta} + X_{\alpha}  , X_{\alpha+ 2 \beta} \rangle$ &&& $K^{1}$& 2$\mathfrak{n}_{1,1}$\\
\hline 
\end{tabular}
\vspace{2mm}
\caption{Two-dimensional subalgebras of $\soo$,  up to conjugacy by $Sp(4, \mathbb{C})$}\label{twodimuu}
\end{table}
\end{small}
\end{landscape}


\begin{landscape}
\begin{scriptsize}
\renewcommand{\arraystretch}{1.4}
\begin{table}[h!]
\begin{tabular}{|c|c|c|c|c|c|}
  \hline 
  Representative & Conditions & Equivalences & Isomorphism Class & Isomorphism Class\\ 
  &&&(de Graaf \cite{degraafa}) &(\v{S}nobl and Winternitz \cite{levi})\\
  \hline
 \multicolumn{5}{|c|}{Containing a Cartan Subalgebra}  \\
  \hline
$\langle \algt, X_{\alpha} \rangle$ &   & & $L^{3}_{0}$ & $\mathfrak{n}_{1,1}\oplus \mathfrak{s}_{2,1}$\\
  \hline
  $\langle \algt, X_{\beta} \rangle$ &   & & $L^{3}_{0}$ &$\mathfrak{n}_{1,1}\oplus \mathfrak{s}_{2,1}$\\
  \hline
   \multicolumn{5}{|c|}{Containing a Regular Semisimple Element but not a Cartan Subalgebra}  \\
  \hline
   $\langle  T_{a,1} , X_{\alpha}, X_{\alpha+ \beta} \rangle$ &$ a \ne 0, \pm 1, -3$&& $L^{3}_{-2(a+1)/(a+3)^{2}}$ & $\mathfrak{s}_{3,1, A=\frac{1+2\alpha  + \sqrt{1+4\alpha}}{-2\alpha}}$\\
   &&&&where $\alpha=-2(a+1)/(a+3)^{2}$\\
\hline 
 $\langle  T_{-3,1} , X_{\alpha}, X_{\alpha+ \beta} \rangle$ &&& $L^{4}_{1}$ & $\mathfrak{s}_{3,1, A=-1}$\\
\hline 
 $\langle T_{a,1} , X_{\alpha}, X_{\alpha+ 2\beta} \rangle$  & $a \ne 0, \pm 1$ &  $\langle T_{a,1} , X_{\alpha}, X_{\alpha+ 2\beta} \rangle 
\sim  \langle T_{a^{-1},1} , X_{\alpha}, X_{\alpha+ 2\beta} \rangle$ & $L^{3}_{ - a/(a+1)^{2}}$&  $\mathfrak{s}_{3,1, A=\frac{1+2\alpha   + \sqrt{1+4\alpha}}{-2\alpha}}$\\
 &&&& where $\alpha=- a/(a+1)^{2} $\\
\hline 
 $\langle T_{3,1} , X_{\alpha}+ X_{\beta},  X_{\alpha+ 2\beta} \rangle$ &&&  $L^{3}_{-3/16}$ &  $\mathfrak{s}_{3,1, A=1/3}$ \\
  \hline
  \multicolumn{5}{|c|}{Containing a Non-regular Semisimple Element but not a Cartan Subalgebra}  \\      
      \hline
 $\langle  T_{1,0}, X_{\alpha}, X_{\alpha + \beta} \rangle$ &&&  $L^{3}_{0}$  & $\mathfrak{n}_{1,1}\oplus \mathfrak{s}_{2,1}$\\  
 \hline
 $\langle  T_{1,0}, X_{\alpha}, X_{\alpha + 2 \beta} \rangle$ &&&  $L^{3}_{0}$ &$\mathfrak{n}_{1,1}\oplus \mathfrak{s}_{2,1}$  \\  
 \hline
 $\langle  T_{1,0}, X_{\alpha + \beta}, X_{\alpha + 2\beta} \rangle$ &&&   $L^{3}_{-2/9}$ & $\mathfrak{s}_{3,1, A=1/2}$\\  
\hline
 $\langle T_{1,-1} , X_{\alpha  + \beta},  X_{\alpha + 2 \beta} \rangle$ &&&   $L^{3}_{0}$&  $\mathfrak{n}_{1,1}\oplus \mathfrak{s}_{2,1}$\\
 \hline
 $\langle T_{1,-1} , X_{\alpha},  X_{\alpha + 2\beta} \rangle$ &&&   $L^{4}_{1}$ & $\mathfrak{s}_{3,1, A=-1}$\\
 \hline
 $\langle T_{1,-1} , X_{\beta} ,  X_{\alpha + 2\beta} \rangle$ &&& $L^{2}$ & $\mathfrak{s}_{3,1, A=1}$\\
 \hline
 $\langle T_{1,1} , X_{\alpha}  , X_{\alpha + 2 \beta} \rangle$ &&&  $L^{2}$  & $\mathfrak{s}_{3,1, A=1}$\\
\hline
  \multicolumn{5}{|c|}{Containing no Semisimple Elements but not Nilpotent}  \\      
      \hline
 $\langle T_{1,1} + X_{\beta}, X_{\alpha+\beta} , X_{\alpha+2\beta} \rangle$ &&& $L^{3}_{-1/4}$ & $\mathfrak{s}_{3,2}$\\ 
\hline 
 $\langle T_{1,-1} + X_{\alpha+\beta} , X_{\alpha}, X_{\alpha+2\beta} \rangle$ &&& $L^{4}_{1}$& $\mathfrak{s}_{3,1, A=-1}$ \\
\hline 
 $\langle T_{1,0} + X_{\alpha} , X_{\alpha + \beta}, X_{\alpha+2\beta} \rangle$ &&& $L^{3}_{-2/9}$ &$\mathfrak{s}_{3,1, A=1/2}$\\ 
\hline
  \multicolumn{5}{|c|}{Nilpotent }  \\
  \hline
  $\algn_{\p}$ & &  &$L^{1}$ & $3\mathfrak{n}_{1,1}$\\  
      \hline
$\langle X_{\beta},  X_{\alpha + \beta}, X_{\alpha + 2 \beta} \rangle$ & &  &$L^{4}_{0}$ &$\mathfrak{n}_{3,1}$\\  
      \hline
$\langle X_{\alpha} + X_{\beta},  X_{\alpha + \beta}, X_{\alpha + 2 \beta} \rangle$ & & &$L^{4}_{0}$& $\mathfrak{n}_{3,1}$\\  
      \hline
\end{tabular}
\vspace{2mm}
\caption{Three-dimensional solvable subalgebras of $\soo$,  up to conjugacy by $Sp(4, \mathbb{C})$}\label{threedimuu}
\end{table}
\end{scriptsize}
\end{landscape}


\begin{landscape}
\begin{scriptsize}
\renewcommand{\arraystretch}{1.5}
\begin{table}[h!]
\begin{tabular}{|c|c|c|c|c|}
  \hline 
  Representative & Conditions & Equivalences & Isomorphism Class&Isomorphism Class\\ 
  &&&(de Graaf \cite{degraafa}) &(\v{S}nobl and Winternitz \cite{levi})\\
  \hline
 \multicolumn{5}{|c|}{Containing a Cartan Subalgebra}  \\
  \hline
   $  \langle \algt , X_{\alpha}, X_{\alpha + \beta} \rangle $ &&&$M^{8}$& $\mathfrak{s}_{4,12}$\\ 
\hline
 $\langle \algt , X_{\alpha}, X_{\alpha + 2\beta} \rangle$ &&&$M^{8}$& $\mathfrak{s}_{4,12}$ \\ 
\hline
 \multicolumn{5}{|c|}{Containing a Regular Semisimple Element but not a Cartan Subalgebra}  \\
  \hline
   $    \langle T_{a,1}, \algn_{\p} \rangle$ & $a \ne 0, \pm 1$ &  $    \langle T_{a,1}, \algn_{\p} \rangle \sim \langle T_{a^{-1},1}, \algn_{\p} \rangle$ 
 &  $M^{6}_{A,B}$,  where & see Subsection \ref{4dimequivalences}\\
     &  & $\sim \langle T_{-a,1}, X_{\beta}, X_{\alpha + \beta}, X_{\alpha + 2 \beta} \rangle$ & $A = \frac{4a}{27(a+1)^{2}}$,  
    $B = -\frac{2(a^{2}+4a+1)}{9(a+1)^{2}}$&\\
  \hline
  $\langle T_{a,1}, X_{\beta}, X_{\alpha + \beta}, X_{\alpha + 2 \beta} \rangle$ & $a \ne 0, \pm 1$ & $\langle T_{a,1}, X_{\beta}, X_{\alpha + \beta}, X_{\alpha + 2 \beta} \rangle$ &$M^{13}_{(1-a^{2})/(4a^{2})}$&  $\mathfrak{s}_{4, 8, A= \frac{1+2\alpha   + \sqrt{1+4\alpha}}{-2\alpha}}$ \\
   &  & $\sim \langle T_{-a,1}, X_{\beta}, X_{\alpha + \beta}, X_{\alpha + 2 \beta} \rangle$ && where $\alpha=(1-a^{2})/(4a^{2})$ \\
  \hline
    $ \langle T_{3,1}, X_{\alpha} + X_{\beta} , X_{\alpha + \beta}, X_{\alpha + 2 \beta} \rangle$ &&  & $M^{13}_{-2/9}$ & $\mathfrak{s}_{4, 8,  A=1/2}$\\
\hline

  \multicolumn{5}{|c|}{Containing a Non-regular Semisimple Element but not a Cartan Subalgebra}  \\      
      \hline
 $\langle T_{0,1}, \algn_{\p} \rangle$ &&&  $M^{6}_{0,-2/9}$ & $\mathfrak{n}_{1,1}\oplus \mathfrak{s}_{3, 1, A=1/2}$ \\
\hline
 $\langle T_{0,1}, X_{\beta}, X_{\alpha + \beta}, X_{\alpha + 2 \beta} \rangle$ &&&  $M^{14}_{1}$&  $\mathfrak{s}_{4,6}$ \\ 
\hline
 $\langle T_{1,0}, X_{\beta}, X_{\alpha + \beta}, X_{\alpha + 2 \beta} \rangle$ &&&  $M^{12}$  & $\mathfrak{s}_{4,8, A=1}$\\
\hline
 $\langle T_{1,1}, \algn_{\p} \rangle$ &&&  $M^{2}$ &  $\mathfrak{s}_{4, 3, A=B=1}$\\
\hline
 $\langle T_{1,-1}, \algn_{\p} \rangle$ &&&  $M^{7}_{0,1}$ & $\mathfrak{n}_{1,1}\oplus \mathfrak{s}_{3,1, A=-1}$  \\
\hline
 $\langle T_{1,1}, X_{\beta}, X_{\alpha + \beta}, X_{\alpha + 2 \beta} \rangle$ &&&  $M^{13}_{0}$&  $\mathfrak{s}_{4, 11}$ \\
\hline
 \multicolumn{5}{|c|}{Containing no Semisimple Elements but not Nilpotent}  \\
  \hline
  $\langle T_{1,1} + X_{\beta}, X_{\alpha}, X_{\alpha+ \beta}, X_{\alpha+ 2\beta} \rangle$ &&& $M^{6}_{1/27, -1/3}$ & $\mathfrak{s}_{4, 2}$\\
\hline
$\langle  T_{1,0} + X_{\alpha} , X_{\beta},  X_{\alpha+ \beta}, X_{\alpha+ 2\beta} \rangle$ &&& $M^{13}_{-1/4}$ & $\mathfrak{s}_{4, 10}$ \\
\hline
\multicolumn{5}{|c|}{Nilpotent Subalgebras}  \\
  \hline
  $ \algn$ &&&  $M^{7}_{0,0}$ &  $\mathfrak{n}_{4,1}$\\
\hline
 \end{tabular}
\vspace{2mm}
\caption{Four-dimensional solvable subalgebras of $\soo$,  up to conjugacy by $Sp(4, \mathbb{C})$}\label{fourdimuu}
\end{table}
\end{scriptsize}
\end{landscape}



\begin{landscape}
\renewcommand{\arraystretch}{1.5}
\begin{table}[h!]
\begin{tabular}{|c|c|c|c|c|c|}
  \hline 
 Dimension & Representative & Conditions &  Isomorphism Class\\ 
 && & (\v{S}nobl and Winternitz \cite{levi})\\
  \hline
 \multicolumn{4}{|c|}{Containing a Cartan Subalgebra}  \\
  \hline
$5$  &  $  \langle \algt , \algn_{\p} \rangle $ & & $\mathfrak{s}_{5,41, A=B=\frac{1}{2}}$\\ 
\hline
$5$ & $\langle \algt , X_{\beta}, X_{\alpha + \beta}, X_{\alpha + 2\beta} \rangle$ &&  $\mathfrak{s}_{5,44}$\\ 
\hline
 \multicolumn{4}{|c|}{Containing a Regular Semisimple Element but not a Cartan Subalgebra}  \\
  \hline
 $5$   &$    \langle T_{a,1}, \algn \rangle$ & $a \ne 0, \pm 1$    
 &  $\mathfrak{s}_{5, 35, A=\frac{2}{a-1}}$ \\
  \hline
   \multicolumn{4}{|c|}{Containing a Non-regular Semisimple Element but not a Cartan Subalgebra}  \\
  \hline
$5$  & $\langle T_{1,-1}, \algn \rangle$ &&    $\mathfrak{s}_{5, 35, A=-1}$\\
\hline
$5$& $\langle T_{1,1}, \algn \rangle$ &&     $\mathfrak{s}_{5, 37}$\\
\hline
$5$& $\langle T_{1,0}, \algn \rangle$ &&   $\mathfrak{s}_{5, 36}$  \\
\hline
$5$& $\langle T_{0,1}, \algn \rangle$ &&  $\mathfrak{s}_{5, 33}$ \\
\hline 
\hline
$6$ & $ \borel$ &&    $\mathfrak{s}_{6, 242}$\\
\hline
 \end{tabular}
\vspace{2mm}
\caption{Five- and six-dimensional solvable subalgebras of $\soo$,  up to conjugacy by $Sp(4, \mathbb{C})$}\label{fivesixdimuu}
\end{table}
\end{landscape}


\section{Conclusions}
 The semisimple subalgebras of the rank $2$ symplectic Lie algebra $\soo$ are well-known \cite{degraafc,degraafd}, and we recently classified its Levi decomposable subalgebras \cite{dc2}.  In this article, we classified the solvable subalgebras of $\soo$, up to inner automorphism  (equivalently,  up to conjugacy by the symplectic group $Sp(4,\mathbb{C})$).   By Levi's theorem, this completes the classification of the subalgebras of $\soo$.  

We summarized the classification in Tables \ref{onedimuu} to \ref{fivesixdimuu}.  The classification is given with respect to the partial classification of solvable Lie algebras of de Graaf \cite{degraafa}, and  that described by \v{S}nobl and Winternitz in \cite{levi}.

We have already classified the subalgebras of the special orthogonal algebra $\mathfrak{so}(4,\mathbb{C})$ \cite{drsof}, and the subalgebras of the special linear algebra $\mathfrak{sl}(3,\mathbb{C}) $ \cite{a2}.  And, Mayanskiy \cite{may} recently posted a classification of the subalgebras of the exceptional Lie algebra $G_2$. Hence, with this article, the classification of the subalgebras of the  rank $2$ semisimple Lie algebras is complete.

\section*{Acknowledgements}

The work of A.D. is partially supported by a research grant from the Professional Staff
Congress/City University of New York (Grant No. TRADA-47-36). The
work of J.R. is partially supported by the Natural Sciences and Engineering Research Council
(Grant No. 3166-09).  The authors would also like to thank the referee for his/her valuable comments.

\end{document}